\documentclass[preprint,12pt]{elsarticle} 

\usepackage{amssymb}
\usepackage{amsmath}
 \usepackage{amsthm}
\theoremstyle{thmstyleone}%
\newtheorem{theorem}{Theorem}
\newtheorem{proposition}[theorem]{Proposition}%
\newtheorem{lemma}[theorem]{Lemma}
\newtheorem{corollary}[theorem]{Corollary} 

\theoremstyle{thmstyletwo}%
\newtheorem{remark}{Remark}%
\theoremstyle{thmstylethree}%
\newtheorem{definition}{Definition}%

\usepackage{lineno}
\modulolinenumbers[5]

\begin{document}

\begin{frontmatter}



\title{Regularity estimates of fractional heat semigroups related with uniformly elliptic operators} 

	
\author[label1]{Honglei Shi}
\ead{Honglei__Shi@163.com}

\author[label1]{Pengtao Li}
\ead{ptli@qdu.edu.cn}

\author[label1]{Kai Zhao\corref{cor1}}
\ead{zhaokaiqdu@163.com}
\cortext[cor1]{Corresponding author}

\address[label1]{School of Mathematics and Statistics, Qingdao University, Qingdao 266071, Shandong, P.R. China}



\begin{abstract}
Let $L = -div( A(x) \cdot \nabla ) + V(x)$ be a second-order uniformly elliptic operator on $\mathbb{ R }^{n}$ $(n\geq 3)$, where $A(x)$ is a real symmetric matrix satisfying standard ellipticity conditions, and $V$ is a nonnegative potential belonging to the reverse H\"older class.
	For $ \alpha \in (0,1) $, we study regularity estimates of the fractional heat semigroups $ \{ exp ( - t L ^ {\alpha } )\} _ { t > 0 }$, via the subordination formula and the fundamental solution of the associated uniformly parabolic equation $ \partial_t u + Lu = 0 .$
	This approach avoids the use of Fourier transforms and is applicable to second-order differential operators whose heat kernels satisfy Gaussian upper bounds. As an application, we characterize the Campanato-type space $\Lambda _ { L , \gamma } \left( \mathbb { R } ^ { n } \right)$ via the fractional heat semigroups $\{exp ( - t L ^ {\alpha } ) \} _ { t > 0 } $.
\end{abstract}


\end{frontmatter}



	\section{Introduction}\label{sec1}

In the theory of PDEs and mathematical physics, elliptic operators play a fundamental role. In particular, the operator $L := - div (A(x)\nabla) $ has widespread applications in various mathematical and physical models, including heat conduction, elasticity, electromagnetism, and fluid dynamics. The study of its spectral properties, regularity theory, and associated variational problems has not only advanced pure mathematics but also profoundly influenced engineering and physical sciences.
Early investigations of such operators date back to the late 19th and early 20th centuries, where the study of Dirichlet and Neumann problems laid the foundation for boundary value theory (see \cite{evans}). Subsequently, De Giorgi \cite{gio}, 	Nash \cite{nash}, and Moser \cite{moser1} independently developed novel methods for establishing the regularity of weak solutions, forming the cornerstone of modern elliptic operator theory. The study of general regularity theory for nonhomogeneous and nonsmooth coefficients remains an active research area in \cite{gil}.
In recent decades, this operator has gained significant attention in stochastic analysis and geometric analysis. For instance, diffusion processes in random environments are closely linked to uniformly elliptic operators in \cite{kipnis}, while in Riemannian geometry, its connection with the Laplace-Beltrami operator reveals important geometric invariants in \cite{gri1}. Furthermore, generalized forms of this operator have been extensively studied in nonlocal problems and fractional PDEs in \cite{caff}.

Let $ L $ be a divergence-type operator with a perturbation term:
$ L := - div (A(x) \cdot \nabla) + V(x) $,
where the potential function $ V $ belongs to the reverse H\"older class $ RH_q $ with $ q \in ( 1 , \infty ) $. The class $ RH_q $ is defined as the set of all non-negative locally $ L^q $-integrable functions $ V $ for which there exists a constant $ C > 0 $ such that for every ball $ B \subset \mathbb{R}^n $, the following inequality holds:
\begin{equation}\label{Bq}
	\left( \frac{1}{|B|} \int_B V^q \, dx \right)^{1/q}
	\leq \frac{C}{|B|} \int_B V \, dx .
\end{equation}
In the study of the influence of the potential function $ V (x) $, Simon \cite{simon} conducted a systematic investigation into the spectral theory of Schr\"odinger operators, exploring the impact of the potential function on the distribution of eigenvalues and spectral measures. Additionally, Davies \cite{davies} studied elliptic operators with singular potentials and established corresponding energy estimates.
In \cite{kur 1}, Kurata and Sugano explored the pointwise estimates and regularity properties of the fundamental solutions associated with the operator $ L $. Moreover, the authors further derived estimates for the operators $ V L^{-1} $, $ V^{1/2} \nabla L^{-1} $, and $ \nabla^2 L^{-1} $ under specific conditions on the coefficients $ a_{ij}(x) $ and the potential $ V $. These estimates were established in the context of weighted Lebesgue spaces and Morrey spaces, respectively.
Furthermore, for potentials $ V $ belonging to the reverse H\"older class, Kurata \cite{kur 0} established pointwise estimates for the heat kernels associated with the operator $ L $. Based on these estimates, the author also derived a weighted smoothing estimate for the semigroup generated by $ L $.
In \cite{ausc1}, Auscher et al. resolved the Kato square root problem, laying the foundation for interdisciplinary research between elliptic operators and harmonic analysis.
In a more general inhomogeneous context, Zhang \cite{zhang} investigated the relationship between the fundamental solutions of variable-coefficient elliptic operators and weighted Sobolev spaces.

When $ A(x) = E $ (the identity matrix), $ L = -\Delta + V $, which is the classical second-order differential operator known as the Schr\"odinger operator. The Schr\"odinger operator $ L $ holds significant importance in the fields of partial differential equations and mathematical physics, as the corresponding Schr\"odinger equation is fundamental to quantum mechanics. With the rapid development of nanotechnology and condensed matter physics, the study of the Schr\"odinger operator $ L $ has gained increasing attention and applications in modern times. Assume the potential function $ V $ belongs to the reverse H\"older class $RH_{q}$ with $ q > n / 2$, where $ n $ is the dimension of the space $ \mathbb{ R }^{n}$.
In \cite{shen}, Shen derived estimates for the fundamental solution associated with $ L $ and, as an application, proved the $ L^{p} $-boundedness of the Riesz transforms related to $ L $.
In \cite{sug}, Sugano investigated the $ L^{p} $-boundedness of the operator
$ V^{\alpha} L^{-\beta} $ for $ 0 \leq \alpha \leq \beta \leq 1 $, as well as the
$ L^{p}$-boundedness of the operator for $ 0 \leq \alpha \leq 1/2 \leq \beta \leq 1 $ and $ \beta - \alpha > 1/2 $. Subsequently, in 2009, Liu extended the results of \cite{sug} to stratified Lie groups in \cite{liu 1}.

Consider the degenerate Schr\"odinger operator defined by \[ L = -\frac{1}{\omega} \operatorname{div}(A(x) \nabla) + V  , \] where $ \omega $ is a weight function belonging to the Muckenhoupt class $ A_2 $, and $ A(x) = (a_{ij}(x)) $ is a real, symmetric matrix-valued function depending on $ x $. The matrix $ A(x) $ satisfies the ellipticity condition \[ C^{-1} \omega(x) |\xi|^2 \leq \sum_{i,j=1}^n a_{ij}(x) \xi_i \bar{\xi}_j \leq C \omega(x) |\xi|^2 \]  for all $ x, \xi \in \mathbb{R}^n $ and for some constant $ C > 0 $. The potential function $ V $ is assumed to be nonnegative and to belong to a reverse H\"older class with respect to the weighted measure $ \omega(x)\,dx $.
To analyze the behavior of nonnegative solutions to the corresponding degenerate elliptic equation, Fabes, Kenig, and Serapioni in \cite{fabes 2} studied the operator $L$ for the special case $ V = 0 $.
Furthermore, Fabes, Jerison, and Kenig in \cite{fabes 1} derived the fundamental solution $ \Gamma_0 $ of $L$ for the special case $ V = 0 $ in a ball. Notably, if $ \omega \in RD_v $ with $ v > 2 $,  \[ \Gamma_0(x, y) \simeq \frac{|x - y|^2}{\omega(B(x, |x - y|))},  \]
where $ RD_v $ denotes the class of functions satisfying the reverse doubling condition with respect to $ v $. For further details about $ L $ with $ V = 0 $, we refer to \cite{chen 1, chen 2, uribe 1, uribe 2} and the references therein.
In \cite{huang 1}, Huang, Li and Liu introduced the area integral associated with the heat semigroup $ \{T_t\}_{t>0} $ generated by the degenerate Schr\"odinger operator $ L $. They also established boundedness, Lipschitz regularity, time derivatives for the kernels of the semigroup. Harboure, Salinas, and Viviani in \cite{harb 2} studied boundedness results for operators related to the semigroup generated by the degenerate Schr\"odinger operator $ L $, under suitable assumptions on the weight $ \omega $ and the nonnegative potential $ V $.
Assuming that the weight $ \omega $ satisfies both the doubling condition and reverse doubling condition, Wang, Li, and Liu in \cite{wang} investigated various regularity estimates for the fractional heat semigroup $ \{exp(-tL^{\alpha}) \}_{t>0} $ using the subordinative formula, where $ L^\alpha $ denotes the fractional powers of $ L $ for $ \alpha \in (0, 1) $.

The main aim of this paper is  study regularity estimates of fractional  semigroups $ \{ exp(-tL^{\alpha}) \}_{t > 0} $ associated with the uniformly elliptic operators in divergence form \[ L : = L _{0} + V(x) = - div (A (x) \cdot \nabla ) + V ( x ) . \]
Let $ A = E $ (the identity matrix) and $ V = 0 $. For $ \alpha \in (0,1) $, $ \{exp(-tL^{\alpha})\}_{t>0} $ comes back to the classical fractional heat semigroup
$ \{exp(-t (-\Delta )^\alpha)\}_{t>0} $.
When $ V = 0 $, $ L = -\Delta $ is the classical Laplace operator, which is a fundamental object of study in the theory of partial differential equations and mathematical physics.
The fractional Laplacian $(- \Delta)^{\alpha}$ is extensively employed to model a diverse range of complex phenomena through partial differential equations. It has also found applications in the study of various physical systems and engineering challenges, such as L\'{e}vy flights, stochastic interfaces, and anomalous diffusion problems. Furthermore, the fractional heat semigroups $\{ exp (-t (-\Delta)^{\alpha}) \}_{t > 0}$ play a significant role in the fields of partial differential equations, harmonic analysis, potential theory, and modern probability theory. For more details and related applications of the fractional heat semigroups $\{exp(-t (-\Delta)^{\alpha}) \}_{t > 0}$, readers may refer to \cite{chang, gri, jiang 1, zhai}.
For general $ A( \cdot ) $ and potential $ V \in B_{q} $, the pointwise estimate and the
Lipschitz estimate of $ K_{ \alpha , t } ^ { L } ( \cdot , \cdot ) $ have been obtained by
Wang, Li, and Liu in \cite{wang} with $ \omega(x) = 1 $, where $ K_{ \alpha , t } ^ { L } ( \cdot , \cdot ) $ denotes the integral kernel of $ \{exp(-tL^{\alpha})\}_{t>0} $. For the regularity estimate of $ K_{ \alpha , t } ^ { L } ( \cdot , \cdot ) $, by the functional
calculus, Huang, Li, Liu and Shi in \cite{huang} obtained several estimates for
$ | \partial_{t} K_{ \alpha , t } ^ { L } ( \cdot , \cdot ) | $. For further developments on
this topic, we refer to \cite{jiang 2, zhang 0, li}.
In the regularity estimates for the kernels of the semigroup, one of the key tools is the subordinative formula (\ref{subor}). Thanks to the analytic properties of the heat semigroup $ \{ exp (-tL) \}_{t > 0} $, the estimates for time derivatives can be derived using the Cauchy integral formula. These estimates can then be extended to the fractional heat semigroups by applying the subordinative formula (\ref{subor}). However, for the spatial derivatives, the method used for time derivatives is no longer applicable. Instead, a more refined and technical approach is required, as detailed in Lemmas \ref{x K bdd} and \ref{x K smo}.

Let $ K_t^L(\cdot, \cdot) $ denote the kernel of the heat semigroup $ \{exp(-tL)\}_{t>0} $, and let $ K_{\alpha, t}^L(\cdot, \cdot) $ denote the kernel of the fractional heat semigroup $ \{exp(-tL^{\alpha})\}_{t>0} $.
Note that $K _ { t } ^ { L } ( \cdot , \cdot )$ is a positive, symmetric function on $\mathbb { R } ^ { n } \times \mathbb { R } ^ { n } $.
The fractional heat kernels can be defined via a method independent of the Fourier multipliers.
By \cite[page 20]{gri}, for $\alpha \in (0,1) $, there exists a nonnegative continuous function $\eta _ {t} ^ { \alpha } ( \cdot )$ satisfying (\ref{subordinative}) below such that
$$ K_{\alpha , t}^{L} ( x, y )=\int^{\infty}_{0} \eta _{t}^{\alpha} ( s ) K_{s}^{L} ( x, y ) \, d s, \quad x, y \in\mathbb{R}^{n}. 	$$
We point out that the regularity estimate obtained in \cite{huang} corresponds to the deviation in the regularity of the time-variable. Davies pointed out in \cite[page 107]{davies 1} that
a second order elliptic differential operator has a heat kernel which is a positive continuous function of the space and time variables jointly. Differentiability in the space variables is not a general fact, but depends upon some degree of smoothness of the coefficients of the differential operator (cf. \cite{davies 1}).
In Section \ref{sec 3.1}, we focus on the regularity estimates on the kernels of $\{ exp ( - t L ^ {\alpha } )  \} _ { t > 0 }$ for $ \alpha = 1$. Inspired by Gao-Jiang \cite{gao} and Lin \cite{lin 0}, we extend spatial derivative estimation of Schr\"odinger operators in \cite{duong 1} to uniform elliptic operator; see Lemma \ref{x K bdd}.
Motivated by the work of Dziuba\'nski-Preisner \cite{dziub 2}, we derive several related $ L ^ {p} $-estimates for $ 1 \leq p  < \infty $ as corollaries of Lemma \ref{x K bdd}; see Proposition \ref{Lbdd}.
By a simple computation, we can similarly obtain the $ L ^ {p} $-estimates of
$ K_{t}^{L} ( \cdot , \cdot ) $ and $ Q _ {t,m} ^ { L } ( \cdot , \cdot ) $ from Propositions \ref{K bdd} \& \ref{Q bdd}; see Proposition \ref{corollary}.
Due to the $L ^ { p }$-boundedness of the operator $ \nabla ^ { 2 } L _ { 0 }^ { - 1 } $ obtained in \cite[Theorem B]{avell}, where $ L _ { 0 } = - div (A (x) \cdot \nabla )$, we can further establish the Lipschitz continuity of $ \nabla _ { x } K _ { t } ^ { L }  ( \cdot , \cdot ) $ under the assumption that $\left(a _ { i j } ( x )\right)$ satisfies the conditions $(A1)$ \& $(A2)$ \& $(A3)$ in Lemma \ref{x K smo}.

Throughout this paper, we assume that $A(x) = ( a_{ij} (x) )$ is a real symmetric matrix depending on $x$ which satisfies $$C^{-1}\left |\xi \right |^{2} \leq \sum_{ i , j } ^ {n} a_{ij} (x) {\xi}_{i} \overline{\xi}_{j}\leq C \left |\xi \right |^{2} $$ for some positive constant $C$ and all $x,\xi $ in $\mathbb{ R }^{n}$, $V$ is a nonzero, nonnegative potential belonging to the reverse H\"older class $B _ { q }$, $q > n / 2,$ which is defined in (\ref{Bq}). Additionally, we always assume that $a _ { i j } ( x )$ is a real-valued measurable function satisfying the following conditions (A1) and (A2).

(A1). There exists a constant $\lambda \in ( 0 , 1 ]$ such that for $ x  \in \mathbb { R } ^ { n } $ and $\xi = ( \xi_{ 1 }, \xi_{ 2 },\ldots, \xi_{ n } ) \in \mathbb { R } ^ { n }$,
\[a _ { i j } ( x ) = a _ { j i } ( x ) , \; \lambda | \xi | ^ { 2 } \leq \sum _ { i , j = 1 } ^ { n } a _ { i j } ( x ) \xi _ { i } \xi _ { j } \leq \lambda ^ { - 1 } | \xi | ^ { 2 }  .\]

(A2). There exist constants $\alpha \in ( 0 , 1 ]$ and $K > 0$ such that
\[ a _ { i j } \in  C ^ { 1 +  \alpha }  \left( \mathbb { R } ^ { n } \right) ,  \;
\left\| a _ { i j } \right\| _ { C ^ {  1 + \alpha }  }
= \left\| a _ { i j }  \right\|_{\infty} +  \left\| \nabla a _ { i j }  \right\|_{\infty}
+  \sup \limits_{ x \neq y } \displaystyle \frac{ | a_{ij} (x) -  a_{ij} (y)| }{ | x - y | ^ { \alpha } }	\leq K .   \]

The following estimates are well-known under the assumptions (A1) and (A2)
(see \cite[Theorem (3.3)]{gru}). There exist constants $C _ { 1 } = C _ { 1 } ( n , \lambda ) , C _ { 2 } = C _ { 2 } ( n , \lambda , K )$ and $C _ { 3 } = C _ { 3 } ( n , \lambda , K )$ such that
\begin{equation} \label{green}
	\left\{\begin{aligned}
		&0 \leq \Gamma _ { 0 } ( x , y ) \leq \frac { C _ { 1 } } { | x - y | ^ { n - 2 } } ;\\
		& \left| \nabla _ { x } \Gamma _ { 0 } ( x , y ) \right|
		\leq \frac { C _ { 2 } } { | x - y | ^ { n - 1 } } ;\\
		& \left| \nabla _ { x } \nabla _ { y } \Gamma _ { 0 } ( x , y ) \right|
		\leq \frac { C _ { 3 } } { | x - y | ^ { n  } } .
	\end{aligned}\right.
\end{equation}

To study the Lipschitz property of $K _ { t } ^ { L }  ( \cdot , \cdot ) $, we need the $L ^ { p }$-boundedness of the operator $ \nabla ^ { 2 } L _ { 0 } ^ { - 1 } $, where $L_{0} = -div ( A(x) \cdot \nabla )$. By \cite[Theorem B]{avell}, we impose the additional condition $( A 3)$ on $a _ { i j } ( x )$.

(A3). There exists a constant $\alpha \in ( 0 , 1 ]$ such that for all $ x \in \mathbb { R } ^ { n } $ and $ z \in \mathbb { Z } ^ { n } $,
\[ a _ { i j } ( x + z ) =  a _ { i j } ( x  ) , \quad \sum_ { i = 1 } ^ { n } \partial_{ i } \left( a _ { i j } (x) \right) = 0  \;
( j = 1 , 2 , \cdots , n )  . \]

Generally, for a second-order differential operator $L$ on $\mathbb{R}^{n}$ and other settings, the fractional operators $L^{\alpha}$ is defined as follows.
For $ \alpha \in (0,1)$, the fractional power of $L $ denoted by
$L ^ { \alpha }$ is defined as
\begin{equation}\label{def fra}
	L^ { \alpha } ( f ) = \frac{ 1 }{ \Gamma ( - \alpha ) } \int_{ 0 } ^ { \infty }
	\left(exp(-t \sqrt{L} ) f(x) -f(x)\right)
	\frac{dt}{t^{1+ 2\alpha }}  \quad  \forall f \in L^{2} ( \mathbb{R}^{n} ) .
\end{equation}
Due to the fact that time-fractional operators are proven to be very useful for modeling purposes, the field of fractional calculus has seen growing interest recently. For instance, the time-fractional heat equation,
\begin{equation} \label{equ heat}
	\partial_t^{\beta} u(x,t) = \Delta u(x,t)
\end{equation}
is frequently used to model the heat propagation in inhomogeneous media. Unlike its classical counterpart, this equation reflects sub-diffusive behavior and is closely linked to anomalous diffusion processes in media with fractal or random structures.
It is known that, as opposed to the classical heat equation, Equation (\ref{equ heat}) is known to exhibit sub-diffusive behaviour and is related to anomalous diffusions or diffusions in non-homogeneous media, with random fractal structures.
Based on this situation, we further explore the time fractional equation
\begin{equation}
	\partial _ { t } ^ { \beta } u ( x , t ) = div ( A(x) \cdot \nabla ) u ( x , t ) .
\end{equation}
In Section \ref{sec 3.2}, we define the fractional derivative of $K_{\alpha,t}^{L}(\cdot ,\cdot )$ for $\beta > 0$ as follows:
\begin{equation} \label{10}
	\partial_{ t }^{\beta}K_{\alpha,t}^{L}(x,y):=\frac{exp(-i\pi (m-\beta))}{\Gamma(m-\beta )}
	\int_{0}^{\infty} \partial_{ t }^{m} K_{\alpha,t+u}^{L}(x,y)u^{m-\beta} \frac{du}{u}, \quad  m=[\beta]+1 .
\end{equation}
Following the work \cite{li}, we obtain the boundedness and the Lipschitz smoothness of $ K_{ \alpha , t } ^ { L } (\cdot , \cdot) $ through the subordination formula (\ref{subor});
see (i) and (ii) of Proposition \ref{bdd frac diff} .
Additionally, we establish regularity estimates for the kernel of the operator
$t ^ { \beta } \partial _ { t } ^ { \beta } exp(-tL^{\alpha}) $, denoted by $ D _ {\alpha,t}^{L,\beta} \left( \cdot , \cdot \right)$; see Proposition \ref{bdd D}.
Moreover, using Lemmas \ref{x K bdd} \& \ref{x K smo}, we establish the desired estimates of the kernels of $ t ^ { 1 / ( 2 \alpha ) } \nabla_{ x } exp(-tL^{\alpha})$ for $ \alpha \in ( 0, 1 ) $; see Propositions \ref{t x frac}. At the end of Section \ref{sec 3}, according to (\ref{def frac L}), we derive analogous estimates for $ \widetilde{D}_{\alpha, t} ^ {L, \beta } (\cdot , \cdot )$, which is the kernel of $t^{\beta } L ^{\alpha \beta } exp(-tL^{\alpha})$; see Lemma \ref{bdd}.

As a concrete application, Section \ref{sec 4} focuses on a characterization of the Campanato type spaces $ \Lambda_{L,\gamma} \left( \mathbb { R }^{n} \right) $ associated with $L $, in terms of the fractional heat semigroup $\{ exp ( - t L ^ {\alpha } ) \} _ { t > 0 } $. Over the past decades, significant attention has been directed toward characterizing function spaces related to Schr\"odinger operators through semigroup methods and Carleson measures. 	Notably, for potentials $V\in B_{q}$ with $q>n/2$, Dziuba\'{n}ski-Garrig\'{o}s-Mart\'{i}nez-Torrea-Zienkiewicz \cite{dziub 1} provided a Carleson measure characterization of the space $ B M O_{L}\left( \mathbb { R } ^ { n } \right)$ with $ L = - \Delta + V $, employing the operator family $\{ t \partial _ { t } exp ( - t L ) \} _ { t > 0 } $.
Building upon the foundational work of Dziuba\'nski et al. \cite{dziub 1}, Wu and Yan \cite{wu} extended the Carleson measure characterization to the BMO type space related to generalized Schr\"odinger operators by replacing the potential $V$ with a general Radon measure $\mu $.
In the context of the Heisenberg group, Lin and Liu \cite{lin} obtained analogous results. Furthermore, by fractional derivatives of the Poisson semigroup, Ma et al. \cite{ma 1} derived characterizations of Campanato-type spaces associated with the operator $L = - \Delta + V $.
For further information, we refer readers to \cite{duong 1, duong 2, huang 3, jiang 2, song, yang 1, yang 3}.
Under the assumption that $ L = - div ( A ( x ) \cdot  \nabla ) + V $ with $V \in B _ { q } $ for some $ q > n $, we apply the regularity estimates in Section \ref{sec 3} to characterize
$\Lambda_{L,\gamma}(\mathbb{R}^n)$ as follows:
\begin{equation}\label{characterizations}
	\begin{aligned}
		& f \in \Lambda_{L,\gamma} \left( \mathbb { R }^{n} \right)  \\
		&\sim \sup _ { B } \frac { 1 } { | B | ^ { 1 + 2 \gamma / n } } \int _ { 0 } ^ { r _ { B } ^ {2\alpha} }
		\int _ { B } \left| t ^ { \beta / \alpha } L ^ {\beta} exp ( - t L ^ {\alpha } ) ( f ) ( x ) \right| ^ { 2 }
		\frac { d x d t } { t } < \infty \\
		&\sim \sup _ { B } \frac { 1 } { | B | ^ { 1 + 2 \gamma / n } } \int _ { 0 } ^ { r _ { B } ^ {2\alpha} }
		\int _ { B } \left|  t ^ { \beta } \partial _ { t } ^ { \beta } exp ( - t L ^ {\alpha } ) ( f ) ( x ) \right| ^ { 2 } \frac { d x d t } { t } < \infty \\
		&\sim \sup _ { B } 	\frac { 1 } { | B | ^ { 1 + 2 \gamma / n } }  \int _ { 0 } ^ {  r _ { B } ^ { 2 \alpha } } \int _ { B } \left| t ^ { 1 / ( 2 \alpha ) } \nabla_{ \alpha }   exp ( - t L ^ {\alpha } )  ( f ) ( x ) \right|
		^ { 2 }
		\frac { d x d t } { t }   < \infty
	\end{aligned}
\end{equation}
for $0 < \gamma < \min \{ 1 , 2 \alpha , 2 \alpha \beta \} $; see Theorems \ref{fD1}.
From the characterizations in (\ref{characterizations}), we further obtain the following characterizations of the Campanato-Sobolev space $ \Lambda ^ { \kappa } _ { L , \gamma } \left( \mathbb { R } ^ { n } \right)  $ for measurable functions $f$ on $\mathbb { R } ^ { n }$,
\begin{align*}
	& \displaystyle \left\| f \right\| _ { \Lambda ^ { \kappa } _ { L , \gamma } }
	 = \displaystyle  \left\| L ^ { \kappa } f \right\| _ { \Lambda _ { L , \gamma } }  \\
	& \sim \left( \sup _ { ( x _ {0} , r ) \in \mathbb { R } ^ { n + 1 } _ {+} } r ^ { - ( 2 \gamma + n ) }
	\int _ { B ( x_{0} , r ) } \int _ { 0 } ^ { r ^ { 2 \alpha } }
	\left| t ^ { \beta / \alpha } L ^ {\beta} exp ( - t L ^ {\alpha } ) L ^ { \kappa }  f ( x ) \right|
	^ { 2 } \frac{ d t d x }{ t }  \right) ^ { 1 / 2 }   \\
	& \sim \left( \sup _ { ( x _ {0} , r ) \in \mathbb { R } ^ { n + 1 } _ {+} } r ^ { - ( 2 \gamma + n ) }
	\int _ { B ( x_{0} , r ) } \int _ { 0 } ^ { r ^ { 2 \alpha } }
	\left|   t ^ { \beta } \partial _ { t } ^ { \beta } exp ( - t L ^ {\alpha } )  L ^ { \kappa }  f ( x ) \right|
	^ { 2 } \frac{ d t  d x }{ t }  \right) ^ { 1 / 2 }   \\
	& \sim \left( \sup _ { ( x _ {0} , r ) \in \mathbb { R } ^ { n + 1 } _ {+}} r ^ { - ( 2 \gamma + n ) }
	\int _ { B ( x_{0} , r ) } \int _ { 0 } ^ { r ^ { 2 \alpha } } \left| t ^ { 1 / ( 2 \alpha ) } \nabla_{ \alpha }   exp ( - t L ^ {\alpha } )  L ^ { \kappa } f ( x ) \right|
	^ { 2 } \frac{ d t d x }{ t }  \right) ^ { 1 / 2 }   ,
\end{align*}
where $0 < \gamma < \min \{ 1 , 2 \alpha , 2 \alpha \beta \} $, $ \kappa \in (0,\alpha ) $ and $ \alpha \in (0,1) $; see Corollary \ref{Pt character}.

\begin{remark}
	In Section \ref{sec 3.2}, we investigate the regularity estimates for the kernel of operator semigroup $\{ exp(-t L) \}_{t > 0}$, which generalize several results on the regularities of  Schr\"odinger operators. When $\alpha = 1 / 2 $ and $ A (x) = E$, $K _ { 1 / 2 ,t} ^{L} ( \cdot ,\cdot ) = P _ { t } ^ { L } ( \cdot , \cdot )$ coincides with the Poisson kernel associated with the Schr\"odinger operator. In this case, Proposition \ref{t x frac} reduces to \cite[Lemma 3.9]{duong 1}. Furthermore, we generalize \cite[Proposition 3.6, (b)-(d)]{ma 1} and establish the regularities of $\partial _ { t } ^ { \beta  } K _ { \alpha, t } ^ { L } \left( \cdot, \cdot \right)$; see Proposition \ref{bdd D}.
\end{remark}

We collect the notation to be used throughout this paper:
\begin{itemize}
	\item $\cup \sim \vee$ indicates that there exists a constant $c > 0$ such that
	$c ^ { - 1 } \vee \leq \cup \leq c \vee $. Moreover, $\cup \leq c \vee$ is denoted by $\cup \lesssim \vee$. Similarly, $\vee \gtrsim \cup$ stands for $\vee \geq c \cup$;
	\item Throughout the paper, the positive constant $C$ may differ in each occurrence, with its dependence restricted to fixed parameters including the dimension $n$, $\alpha$, $\beta$, and similar quantities;
	\item Let $B$ be a ball of radius $r$. In the following, for any constant $c > 0$, we use $B_{cr}$ to represent the ball with the same center as $B$ but with radius $cr$.
	\item $\partial _ { j } = \nabla _ { j } = \nabla _ { x _ { j } } = \frac { \partial } { \partial x _ { j } } , |{\nabla}_{x} u(x,t)|^2 = \sum\limits_{j=1}^{n} \left( \frac{\partial u}{\partial x_j} \right)^2 , \mathbb{R}^{n+1}_{+}= \mathbb{R}^{n}\times (0,+\infty) , $
	\item $\nabla f(x,t;y,s)=(\partial_{x_1} f,\partial_{x_2}f, \ldots, \partial_{x_n}f)(x,t;y,s),   |x|=\left(\sum_{i=1}^{n}{  x_{i}}^{2}\right)^{1/2},$
	\item $B=B(x_0,r)=\left\{y\in \mathbb{R}^{n}:|y-x_0|<r\right\}, \; |(x,t)|=d\left((x,t),(0,0)\right) , \\$
	\item $Q=Q_r(x_0,t_{0})=\left\{ (x,t)\in{ {\mathbb{R}^{n+1}_{+}}}:\left(|x-x_0|^2+|t-t_0|\right)^{{1}/{2}}<r \right\}.$
\end{itemize}

\section{Preliminaries}
\subsection{Schr\"odinger operators }
By H\"older's inequality, we can obtain $B _ { q_{1} } \subset B _ { q_{2} }$ for $q _ { 1 } \geq q _ { 2 } > 1 $. One remarkable feature about the class $B _ { q }$ is that if $V \in B _ { q }$ for some $q > 1$ then there exists $\epsilon > 0$, depending only on $n$ and the constant $C $ in (\ref{Bq}) such that $V \in B _ { q + \epsilon } $. If $V \in B_q$ with some $q > 1$, the measure $V(x) \; dx$ is doubling, i.e., for all $r > 0$ and $x \in \mathbb{R}^n$,
\begin{equation}\label{equ doubling}
	\int _ { B ( x , 2 r ) } V ( y ) d y \leq C _ { 0 } \int _ { B ( x , r ) } V ( y ) d y .
\end{equation}

The auxiliary function $m ( x , V )$ is defined by
\begin{equation}\label{def m(x,V)}
	\frac { 1 } { m ( x , V ) } : = \sup \left\{ r > 0 : \frac { 1 } { r ^ { n - 2 } } \int _ { B ( x , r ) }  V ( y ) d y \leq 1 \right\} .
\end{equation}
The function $m ( x , V )$ reflects the scale of $V ( x )$ essentially, but behaves better. It was introduced by Shen \cite{shen} and will play a crucial role in our proof.
Clearly, $0 < m ( x , V ) < \infty$ for every $x \in \mathbb { R } ^ { n } $, and $  r ^ { 2 - n }  \int _ { B ( x , r ) } V ( y ) d y = 1 $ if $r = 1 / m ( x , V ) $. For simplicity, we sometimes denote $1 / m ( x , V )$ by $\rho ( x )$ in the proofs.
Below we state some important properties of $m ( x , V )$ which will be used in the sequel.

\begin{lemma}{\rm (\cite [Lemmas 1.2 \& 1.4 \& 1.8] {shen})}\label{special mV}
	\item[(1)] There exists a constant $C > 0 $ such that for $x \in \mathbb { R } ^ { n } $,
	\begin{equation*}
		\frac { 1 } { r ^ { n - 2 } } \int _ { B ( y , r ) } V ( x ) d x \leq  C \left( \frac { r } { R } \right) ^ { 2 - n / q } \frac { 1 } { R ^ { n - 2 } } \int _ { B ( y , R ) } V ( x ) d x , \quad 0 < r < R < \infty .
	\end{equation*}
	\item[(2)] There exist constants $l _ { 0 } > 0$ such that for $x$, $y$ in $\mathbb{ R }^{n}$,
	\[\left\{ \begin{aligned}
		m ( y , V ) &\lesssim ( 1 + | x - y | m ( x , V ) ) ^ { l _ { 0 } } m ( x , V ) ; \\
		m ( y , V ) &\gtrsim  m ( x , V ) ( 1 + | x - y | m ( x , V ) )
		^ { - l _ { 0 } / \left( 1 + l _ { 0 } \right) } .
	\end{aligned} \right.\]
	Specially, $ \rho (y) \sim \rho (x) $ if $ | x - y | \lesssim \rho (x) $.
	\item[(3)] There exist constants $l _ { 0 } > 0 $ such that for $R \geq m ( x , V ) ^ { - 1 } $,
	\[\frac { 1 } { R ^ { n - 2 } } \int _ { B ( x , R ) } V ( y ) d y \lesssim ( R m ( x , V ) ) ^ { l _ { 0 } } .\]
\end{lemma}

\begin{lemma}{\rm (\cite [Lemma 1] {guo})}\label{xiao V 2}
	Suppose $V \in B _ { q } $, $q > n / 2.$ Let $m _ { 0 } > \log _ { 2 } C _ { 0 } + 1 $, where $ { C } _ { 0 }$ is the constant in (\ref{equ doubling}). Then for any $x _ { 0 } \in \mathbb { R } ^ { n } , R > 0 $, \[  \left( 1 + R m \left( x _ { 0 } , V \right) \right) ^ { - m _ { 0 } }  \int _ { B ( x_{0} , R ) } V ( x ) d x \lesssim R ^ { n - 2 } .\]
\end{lemma}

\begin{lemma} {\rm (\cite [Lemma 6] {li})} \label{lemma}
	There exist constants $\delta ^ {\prime}>0$ and $ m_{0} >0$ such that
	\begin{equation*}
		\frac { 1 } { t ^ { n / 2 } } \int _ { \mathbb { R } ^ { n } } exp( - c | x - y | ^ { 2 } / t ) V ( y ) d y \lesssim \left\{
		\begin{aligned}
			& { t }^{-1} \left( \frac { \sqrt { t } } { \rho ( x ) } \right) ^ { \delta ^ {\prime} } ,
			& \sqrt { t } < m ( x , V ) ^ { - 1 } ; \\
			& { t }^{-1}  \left( \frac { \sqrt { t } } { \rho ( x ) } \right) ^ { m_{0} } ,
			& \sqrt { t } \geq m ( x , V ) ^ { - 1 } .
		\end{aligned}
		\right.
	\end{equation*}
\end{lemma}
Furthermore, by the perturbation theory in \cite[Proposition 2.12]{engel}, 
\begin{equation}\label{pertur}
	\begin{aligned}
		h _ { t } ( x , y ) - K _ { t } ^ { L } ( x , y )
		&= \int _ { 0 } ^ { t } \int _ { \mathbb { R } ^ { n } } h _ { s } ( w , y ) V ( w )
		K _ {  t - s } ^ { L} ( w , x ) d w d s \\
		&= \int _ { 0 } ^ { t / 2 } \int _ { \mathbb { R } ^ { n } } h _ { s } ( w , y ) V ( w )
		K _ {  t - s } ^ { L } ( w , x ) d w d s  \\
		& \quad + \int _ { 0 } ^ { t / 2 } \int _ { \mathbb { R } ^ { n } } h _ {  t - s } ( w , y ) V ( w )
		K _ { s } ^ { L } ( w , x ) d w d s .
	\end{aligned}
\end{equation}
It is easy to see that
$0 \leq K _ { t } ^{L} ( x , y ) \leq h _ { t } ( x , y ) $,
where $h _ { t } ( \cdot , \cdot )$ are the integral kernels of the semigroup
$\{ exp( - t L  _ { 0 } ) \} _ { t > 0 } $ on $L ^ { 2 } ( \mathbb{ R }^{n+1}_{+} )$ generated by $L _ { 0 } $, where \[ L  _ { 0 } f ( x ) = - \sum _ { i , j } \partial _ { i }
\left( a _ { i j } ( \cdot ) \partial _ { j } f \right) ( x ) .\]
In \cite[(2.18)]{hebi}, the authors showed that for $ k \in \mathbb{N}$, there exist constants $ C_{k} > 0 $ and $ c > 0 $ such that the heat kernels of $ L_{0} $ satisfy
\[\left| \partial _ { t } ^ { k } h _ { t } ( x , y ) \right|
\leq \frac { C _ { k } } { t ^ { k + n / 2 }  }
\exp \left( \frac { - | x - y | ^ { 2 } } { c t } \right) .\]
In \cite[(3.4)]{dziub 3}, Dziuba\'nski further derived a smoothness estimate for $ h _{t} ( \cdot , \cdot ) $:
\begin{equation} \label{smo ht}
	\left| h _ { t } ( x , y ) - h _ { t } ( x , z ) \right|
	\lesssim { t ^ { - n / 2 } } ( | y - z | / \sqrt { t } ) ^ { v }
	\exp \left( - c ( | x - y | - 2 | y - z | ) _ { + } ^ { 2 } /  t \right) ,
\end{equation}
where the constants $ v $, $ c $, $ C $ are positive.
It follows from \cite[Theorem 2.1]{hebi} that
\[	\int_{\mathbb{R}^n} h_{t}(x , y) dy = 1.\]
Additionally, for some nonnegative constants $c _ { 1 } , c _ { 2 } , C _ { 1 }$ and $C _ { 2 } $, the kernels $h _ { t } ( \cdot , \cdot )$ satisfy the Gaussian estimates {\rm \cite [Theorem 2.7] {hebi} }:
\begin{equation}\label{def ht}
	\frac{c_{1}}{t^{n/2}}  \exp\left( - c_{2} {| x - y | ^ { 2 }} / { t } \right)
	\leq h_{t}\left(x,y\right) \leq \frac{C_{1}}{t^{n/2}}  \exp\left( - C_{2} {| x - y | ^ { 2 }} / { t } \right).
\end{equation}
Then we conclude that the kernel $K _ { t } ^{L} ( \cdot , \cdot )$ has a Gaussian upper bound. Furthermore, Dziuba\'{n}ski \cite{dziub 3} proved that
\begin{proposition}{\rm (\cite [Theorem 2.2] {dziub 3})}\label{K bdd}
	For $N > 0 $, there exist constants $ { C } _ { N }$ and $c $ such that
	\begin{equation} \label{equ K bdd}
		0 \leq K _ { t } ^ { L } ( x , y )  \leq \frac { C _ { N } \exp\left( -  { c | x - y | ^ { 2 }} / { t } \right) } { t ^ { n / 2 } \left( 1 +  { \sqrt { t } } / { \rho ( x ) } + { \sqrt { t } } / { \rho ( y ) } \right) ^ {  N }}  ,  \;  x , y \in \mathbb { R } ^ { n } .
	\end{equation}
\end{proposition}

\begin{proposition}{\rm (\cite [Proposition 3.2] {huang 1})}\label{K smo}
	For every $0 < \delta  < \delta _ { 0 } = \min \{ 1 , 2 - n / q , v \}$ and every $N > 0$, there exist constants $ { C } _ { N } > 0$ and $c $ such that for $| h |<\min\{ \sqrt{t}, |x-y|/2\} $,
	\begin{equation} \label{equ K smo}
		\left| K _ { t } ^ { L } ( x + h , y ) - K _ { t } ^ { L } ( x , y ) \right| \leq \frac { C _ { N } \left(  { | h | } / { \sqrt { t } } \right) ^ { \delta  } \exp\left( -  { c | x - y | ^ { 2 }} / { t } \right) } { t ^ { n / 2 } \left( 1 +  { \sqrt { t } } / { \rho ( x ) } + { \sqrt { t } } / { \rho ( y ) } \right) ^ {  N }}   .
	\end{equation}
\end{proposition}

Let $ Q  _ { t , m } ^ { L } ( x , y ) : = t ^ { m } \partial _ { t } ^ {m} K _ { t } ^ { L } ( x , y ) $, $m \in \mathbb { Z } _ { + } $. To characterize the Hardy type spaces related with $L$, in \cite{huang 1}, Huang, Li and Liu obtained the following regularity estimate of $K_{t}^{L}(\cdot,\cdot)$.
\begin{proposition}{\rm (\cite [Proposition 3.3] {huang 1})}\label{Q bdd}
\item[\rm(i)] For every $N > 0$, there exist constants $C _ { N } > 0$ and $c > 0$ such that
\[
\left| Q _ { t , m } ^ { L } ( x , y ) \right| \leq \frac { C _ { N } \exp\left( -  { c | x - y | ^ { 2 }} / { t } \right) } { t ^ { n / 2 } \left( 1 +  { \sqrt { t } } / { \rho ( x ) } + { \sqrt { t } } / { \rho ( y ) } \right) ^ {  N }} .
\]
\item[\rm(ii)] Let $0 < \delta \leq \delta _ { 0 }$, where $\delta _ { 0 }$ appears in Proposition \ref{K smo}. For every $N > 0$, there exist constants $C _ { N } > 0$ and $c > 0 $ such that for $| h | < \sqrt { t } $, 
\[
\left| Q _ { t , m } ^ { L } ( x + h , y ) - Q _ { t , m } ^ { L } ( x , y ) \right|  \leq \frac { C _ { N } \left(  { | h | } / { \sqrt { t } } \right) ^ { \delta  } \exp\left( -  { c | x - y | ^ { 2 }} / { t } \right) } { t ^ { n / 2 } \left( 1 +  { \sqrt { t } } / { \rho ( x ) } + { \sqrt { t } } / { \rho ( y ) } \right) ^ {  N }}   .
\]
\item[\rm(iii)] Let $0 < \delta \leq \delta _ { 0 }$, where $\delta _ { 0 }$ appears in Proposition \ref{K smo}. For every $N > 0$, there exists a constant $C _ { N } > 0 $ such that
\[
\left| \int _ { \mathbb { R } ^ { n } } Q _ { t , m } ^ { L } ( x , y ) d y \right|
\leq C _ { N } \left(  { \sqrt { t } } / { \rho ( x ) } \right) ^ { \delta  } \left( 1 +  { \sqrt { t } } / { \rho ( x ) } \right) ^ { - N } .
\]
\end{proposition}

\subsection{Fractional heat kernel associated with $L$}
This section begins by outlining the background on fractional heat semigroups and the fractional heat kernel associated with $L$. In the case where $V \neq 0$, the fractional heat semigroups for $L$ cannot be characterized via the Fourier multiplier method, which is applicable to the Laplace operator.
In \cite[page 20]{gri}, Grigor'yan proved that there exists a nonnegative continuous function $\eta _ {t} ^ { \alpha } ( \cdot )$ satisfying
\begin{equation}\label{subordinative}
	\left\{ \begin{aligned}
		&\eta _ { t } ^ { \alpha } ( s ) = \frac { 1 } { t ^ { 1 / \alpha } }
		\eta _ { 1 } ^ { \alpha } \left( s / t ^ { 1 / \alpha } \right) ; \\
		&\eta _ { t } ^ { \alpha } ( s ) \lesssim \frac { t } { s ^ { 1 + \alpha } } ,
		\forall s , t > 0 ; \\
		&\int _ { 0 } ^ { \infty } s ^ { - \gamma } \eta _ { 1 } ^ { \alpha } ( s ) d s < \infty ,
		\gamma > 0 ; \\
		&\eta _ { t } ^ { \alpha } ( s ) \simeq \frac { t } { s ^ { 1 + \alpha } } ,
		\forall s \geq t ^ { 1 / \alpha } > 0 ,
	\end{aligned} \right.
\end{equation}
such that $K _ { \alpha , t } ^ { L } ( \cdot, \cdot )$ can be expressed as
\begin{equation} \label{subor}
	K _ { \alpha, t } ^ { L } ( x , y ) = \int _ { 0 } ^ { \infty }
	\eta _ { t } ^ { \alpha } ( s ) K _ { s } ^ { L } ( x , y ) d s .
\end{equation}
Therefore, following Grigor'yan \cite{gri}, we use the subordinative formula (\ref{subor}) to express the fractional heat semigroup related with $L$.
Moreover, the Schr\"odinger operator $L$ can be regarded as the generator of the semigroup ${ exp(-tL) }_{t > 0}$, i.e.,
\[L ( f ) : = \lim _ { t \rightarrow 0 } \frac { f - e xp( - t L) f } { t } , \]
where the limit is taken in $L ^ { 2 } \left( \mathbb { R } ^ { n } \right) $.
Additionally, for $ s \in (0,\alpha) $, $ \alpha \in (0,1] $, the fractional power of $L $, denoted by $L ^ { s}$ is defined as
\begin{equation}\label{def frac L}
	L^ { s } ( f ) = C_{ s , \alpha } \int_{ 0 } ^ { \infty }
	\left(exp(-tL^{\alpha}) f(x) -f(x)\right)
	\frac{dt}{t^{1+ s / \alpha }} \quad \forall f\in L^{2} (\mathbb{ R }^{n}) .
\end{equation}

\begin{proposition} \label{bdd frac diff}  {\rm (\cite [Propositions 3.5] {wang})}
	Let $\alpha \in ( 0 , 1 )$ and $V \in B _ { q } $, $q > n / 2 $.
	\item[\rm(i)] For every $N > 0$, there exists a constant ${ C } _ { N } $ such that
	\begin{equation*}
		\left| K _ { \alpha , t } ^ { L } ( x , y ) \right| \leq \frac { C _ { N } t } { ( t ^ { 1 / (2\alpha ) } + | x - y |) ^ { n + 2 \alpha } } \left( 1 + \frac { t ^ { 1 / (2\alpha) } } { \rho ( x ) } + \frac { t ^ { 1 / (2\alpha) } } { \rho ( y ) } \right) ^ { - N } .
	\end{equation*}
	\item[\rm(ii)] For any $N > 0$, there exists a constant ${ C } _ { N } $ such that for every $0 < \delta < \delta _ { 0 } = \min \{ 1 , 2 - n / q , v \}$ and all $| h | \leq t ^ { 1 / (2 \alpha) } $,
	\begin{align*}
	 \left| K _ { \alpha , t } ^ { L } ( x + h , y ) - K _ { \alpha , t } ^ { L } ( x , y ) \right| 
	 \leq \frac { C _ { N } t ( | h | / t^{1/(2\alpha)}) ^ { \delta }  }
		{ ( t ^ { 1 / (2 \alpha) } + | x - y | ) ^ { n + 2 \alpha } \left( 1 + \frac { t ^ { 1 / (2 \alpha) } } { \rho ( x ) } + \frac { t ^ { 1 / (2 \alpha )} } { \rho ( y ) } \right) ^ { - N }} 	  .
	\end{align*}
\end{proposition}

\section{Regularities on fractional heat semigroups associated with $L$ } \label{sec 3}
This section is devoted to establishing pointwise estimates for the fractional heat kernel $K _ { \alpha , t } ^ { L } ( \cdot ,\cdot ) $. We begin by employing the subordination formula (\ref{subor}) to estimate $\partial _ { t } ^ { m } K _ { \alpha , t } ^ { L } ( \cdot , \cdot ) , m \geq 1 $. Subsequently, we derive the spatial gradient of
$K _ { \alpha , t } ^ { L } ( \cdot,\cdot ) $ via the solution to the parabolic equation
\begin{equation}\label{equation}
	\partial _ { t } u + L u = \partial _ { t } u - div(A(x) \cdot \nabla u)  + V u = 0 .
\end{equation}

\subsection{Estimation for $ \alpha = 1 $ }  \label{sec 3.1}

This section is devoted to analyzing the spatial gradient of $ K_t^L(\cdot, \cdot) $.

\begin{lemma} \label{x K bdd}
	Suppose that $\left(a _ { i j } ( x )\right)$ satisfies the conditions $(A1) \; \& \; (A2)$, and that $V \in B _ { q }$ for some $q > n$. For every $N > 0$, there exists a constant
	${ C } _ { N } > 0 $ such that for all $x , y  \in  \mathbb { R } ^ { n }$ and $t > 0$,
	the semigroup kernels $K _ { t } ^ { L } \left( \cdot , \cdot \right)$ satisfy
	the following estimate:
	\begin{equation}\label{u}
		\left| \nabla _ { x } K _ { t } ^ { L } ( x , y ) \right|
		+ \left| t  \nabla_{x} \partial_{ t } K_t^{L}(x, y) \right|
		\leq \frac{C _ { N } { t ^ { - ( n + 1 ) / 2 } } 	exp(- c \left| x  - y \right| ^ { 2 } / t ) }{\left( 1 +  { \sqrt { t } } / { \rho ( x ) }
			+  { \sqrt { t } } / { \rho ( y ) } \right) ^ { N }}.
	\end{equation}
\end{lemma}

\begin{proof}
	
	Let $\Gamma _ { 0 } ( \cdot , \cdot )$ be the fundamental solution of
	$- div \left( A \cdot \nabla \right)$ in $\mathbb { R } ^ { n } $.
	Assume that $u ( \cdot , \cdot )$ is a weak solution to
	\[\partial _ { t } u + L u =
	\partial _ { t } u + ( - div \left( A (x) \cdot \nabla u \right) )  + V u = 0 .\]
Firstly, for any point $y_0 \in \mathbb{R}^n$, define $u(x,t) := K_t^L(x, y_0)$. Then $u$ is a solution to the heat equation
\[
\partial_t u + L u = 0 \quad \text{in } \mathbb{R}^{n+1}_+.
\]
Fix $t > 0$ and set $\epsilon = {1}/ {2}$. Let $\eta \in C_0^\infty\left(B\left(x_0, \epsilon(\sqrt{t} + |x_0 - y_0|)\right)\right)$ be a smooth cutoff function such that $\eta \equiv 1$ on the ball
\[
B\left(x_0, \epsilon^2(\sqrt{t} + |x_0 - y_0|)\right).
\]
 Moreover, on the annulus 
\[
U = B\left(x_0, \epsilon(\sqrt{t} + |x_0 - y_0|)\right) \setminus B\left(x_0, \epsilon^2(\sqrt{t} + |x_0 - y_0|)\right),
\]
we have the estimates
\[
|\nabla \eta(x)| \lesssim (\sqrt{t} + |x_0 - y_0|)^{-1}, \quad
|\nabla^2 \eta(x)| \lesssim (\sqrt{t} + |x_0 - y_0|)^{-2}.
\]

	Noticing that $\left(a _ { i j } ( x )\right)$ satisfies the conditions $(A1)$ \& $ (A2)$ and $\partial _ { t } u + L u = 0 $,
	we can obtain
	\begin{align*}
		&- \sum _ { i ,j  = 1 } ^ { n } \nabla _ { x _ { i } }
		\left( a _ { i j } ( x ) \nabla _ { x _ { j } }  ( u \eta )\right) \\
		&= - \sum _ { i ,j  = 1 } ^ { n } a_{ij} (x)
		\left(\nabla _ { x _ { i } } \nabla _ { x _ { j } }  u \cdot \eta
		+ \nabla _ { x _ { i } } u \cdot \nabla _ { x _ { j } } \eta
		+  \nabla _ { x _ { i } } \eta  \cdot \nabla _ { x _ { j } } u
		+ \nabla _ { x _ { i } } \nabla _ { x _ { j } } \eta  \cdot u  \right)  \\
		& \quad - \sum _ { i ,j  = 1 } ^ { n } \nabla _ { x _ { i } } a_{ij} (x)
		\left( \nabla _ { x _ { j } } u  \cdot \eta   (x,t)
		+ \nabla _ { x _ { j } } \eta  \cdot  u  (x,t) \right)   \\
		&= -\partial_{t} u \cdot  \eta - Vu \eta
		- \sum _ { i ,j  = 1 } ^ { n }  \nabla _ { x _ { i } }  a_{ij} (x)
		\nabla _ { x _ { j } } \eta  \cdot u  \\
		& \quad   - 2 \sum _ { i ,j  = 1 } ^ { n }
		a_{ij} (x) \nabla _ { x _ { i } } u \cdot \nabla _ { x _ { j } } \eta
		- \sum _ { i ,j  = 1 } ^ { n }  a_{ij} (x)
		\left(  \nabla _ { x _ { i } }  \nabla _ { x _ { j } } \eta  \cdot u \right) .
	\end{align*}
	Since $\left(a _ { i j } ( x )\right)$ satisfies the conditions $(A1)$ \& $ (A2)$, there exists a constant $ C _ { n , K } $ such that
	\begin{equation}\label{A}
		\left\{\begin{aligned}
			| a_{ij} (x) | &\leq K ,
			\quad  x \in
			B\left(x_{0},\epsilon \left(\sqrt{t} + |x_{0}-y_{0}| \right) \right)  ; \\
			| \nabla  a_{ij} (x) | &\leq \frac{C _ { n , K } }{ \sqrt{t} + |x_{0}-y_{0}| } ,
			\quad  x \in U ,
		\end{aligned}\right.
	\end{equation}
	which, together with integration by parts, gives
	\begin{align*}
		&- \int _ { \mathbb{ R } ^ { n } } \sum _ { i ,j  = 1 } ^ { n } a_{ij} (y) \cdot
		\Gamma _ { 0 } ( x , y ) \cdot {\nabla}_{i} u ( y , t ) \cdot {\nabla}_{j} \eta ( y ) d y  \\
		&=\int _ { \mathbb{ R } ^ { n } }     \left(\sum _ { i ,j  = 1 } ^ { n }  a_{ij} (y)
		\cdot  {\nabla}_{i} \Gamma _ { 0 } ( x , y )  \cdot   {\nabla}_{j} \eta (y)  \right) u (y,t) d y  \\
		& \quad +\int _ { \mathbb{ R } ^ { n } }    \left(\sum _ { i ,j  = 1 } ^ { n } a_{ij} (y)
		\cdot \Gamma _ { 0 } ( x , y )  \cdot   {\nabla}_{i}  {\nabla}_{j} \eta (y)  \right) u (y,t) d y  \\
		& \quad +\int _ { \mathbb{ R } ^ { n } }    \left(\sum _ { i ,j  = 1 } ^ { n }
		{\nabla}_{i} a_{ij} (y)   \cdot \Gamma _ { 0 } ( x , y )  \cdot
		{\nabla}_{j} \eta (y)  \right) u (y,t) d y  .
	\end{align*}
	
	It follows from Lemma \ref{xiao V 2} and \cite[(1.7)] {shen} that
	\[\int _ { B \left( x , R \right) } \frac { V ( y ) } { | x - y | ^ { n - 1 } } d y
	\leq \frac { C } { R ^ { n - 1 } } \int _ { B \left( x , R \right) } V ( y ) d y
	\leq \frac { C } { R }
	\left( 1 + \frac { R } { \rho \left( x  \right) } \right) ^ { l _ { 0 } } , \;
	l _ { 0 } > 1 .\]
	
	By (\ref{green}), for any $x \in B\left(x_{0},\epsilon^{3} \left(\sqrt{t} + |x_{0}-y_{0}| \right) \right) $, we have
	\begin{align*}
		|\nabla_{ x } u(x,t) |
		&= {\Bigg|} \int _ { \mathbb { R } ^ { n } } \nabla_{ x } \Gamma _ { 0 } ( x , y )
		\Bigg\{ -\partial_{t} u (y,t) \cdot  \eta (y)- V ( y ) u ( y , t ) \eta ( y ) \\
		&\quad + \sum _ { i ,j  = 1 } ^ { n }  a_{ij} (y)
		\left(  \nabla _ { i }  \nabla _ { j } \eta (y)  \cdot u (y,t) \right)  \\
		& \quad +  \sum _ { i ,j  = 1 } ^ { n }
		{\nabla}_{i} a_{ij} (y)   \cdot {\nabla}_{j} \eta (y)  \cdot u (y,t)   \Bigg\} d y \\
		& \quad + 2 \int _ { \mathbb{ R } ^ { n } }     \sum _ { i ,j  = 1 } ^ { n }  a_{ij} (y)  \cdot
		\nabla_{ x } {\nabla}_{i} \Gamma _ { 0 } ( x , y )  \cdot
		{\nabla}_{j} \eta (y)  \cdot  u (y,t) d y   {\Bigg|}  \\
		& \lesssim (\sqrt{t} + |x_{0}-y_{0}|) \cdot
		\sup_{ y \in B\left(x_{0},\epsilon \left(\sqrt{t} + |x_{0}-y_{0}| \right) \right) }
		| \partial_{ t } u ( y , t ) |  \\
		& \quad  + \frac{	\left( 1 + { ( \sqrt{t} + |x_{0}-y_{0}| ) } / { \rho \left( x  \right) } \right) ^ { l _ { 0 } }}{ \sqrt{t} + |x_{0}-y_{0}| }
		\cdot \sup_{ y \in B\left(x_{0},\epsilon \left(\sqrt{t} + |x_{0}-y_{0}| \right) \right) }
		| u ( y , t ) | .
	\end{align*}
	Take $u \left( x _ { 0 } , t \right) = K _ { t } ^ { L } \left( x _ { 0 } , y _ { 0 } \right) .$	
	For any $y\in B\left(x_{0},\epsilon \left(\sqrt{t} + |x_{0}-y_{0}| \right) \right)$,
	$ \sqrt{t} + |x_{0}-y_{0}| \sim  \sqrt{t} + |y-y_{0}| $.
	By Proposition \ref{Q bdd} and Lemma \ref{special mV}, a straightforward computation yields
	\begin{align}\label{u1}
		\left| \nabla_{x}  K _ { t } ^{L} \left( x _ { 0 } , y _ { 0 } \right) \right|
		& \lesssim t^{- (n+1) / {2}} exp(-c|x_{0}-y_{0}|^{2} /t  )
		\left(1+ \frac{\sqrt{t}}{\rho(x_{0})} + \frac{\sqrt{t}}{\rho(y_{0})}  \right)
		^ { \frac {-N}  { l_{0} + 1 } + c }    .
	\end{align}
	
	Now, letting $ u ( x_{0} , t ) = t \partial_{ t } K_t^{L}(x_{0}, y_{0}) $, we use a similar argument as
	in (\ref{u1}) to obtain the estimate for the term $\left| t  \nabla_{x} \partial_{ t } K_t^{L}(x, y) \right|$
	in (\ref{u}). This completes the proof of Lemma \ref{x K bdd}.
\end{proof}

\begin{remark}
	When $A=E $ (the identity matrix), the estimate (\ref{u}) with an exponential function decay term is better than the spatial derivative estimate in \cite{li}.
\end{remark}

In \cite[Lemma 2.1]{dziub 2}, Dziuba\'nski and Preisner obtained $ L ^ {1} $-boundedness
and $ L ^ {2} $-boundedness of spatial gradient of the integral kernels of the operator semigroup $\{ \exp\left( - t L \right) \} _ { t > 0 } $ on $L ^ { 2 } ( \mathbb{ R }^{n+1}_{+} )$ generated by $ L $, where $ L = - \Delta + V $. Since the spatial derivative obtained above is a stronger pointwise estimate than the norm estimate, we can generalize the result of Dziuba\'nski and Preisner to  $ L ^ {p} \; ( p \in [1,\infty) )$ spatial gradient estimates of the integral kernels of $\{ \exp\left( - t L \right) \} _ { t > 0 } $ on $L ^ { 2 } ( \mathbb{ R }^{n} )$, where $ L = - \text{div} (A (x) \cdot \nabla ) + V ( x ) $.

\begin{proposition} \label{Lbdd}
	Suppose that $\left(a _ { i j } ( x )\right)$ satisfies the conditions $(A1) \; \& \; (A2)$, and that $V \in B _ { q }$
	for some $q > n$. Then for $\alpha > 0$ and $ p \in [1,\infty) $, the spatial derivative of $K _ { t } ^ { L } \left( \cdot , \cdot \right)$ satisfies
	\begin{equation*}
		\left\| \nabla _ { x } K _ { t } ^ { L } \left( \cdot , y \right) \right\|
		_ { L ^ { p } \left( \exp\left( \alpha  | x - y | / \sqrt { t } \right)  d x \right) }
		\lesssim t ^{ n / ( 2 p ) - ( n + 1 ) / 2 } .
	\end{equation*}
\end{proposition}
\begin{proof}
	By Lemma \ref{x K bdd} and the change of variables, we can obtain
	\begin{align*}
		\left\| \nabla _ { x } K _ { t } ^ { L } \left( \cdot , y \right) \right\|
		_ { L ^ { p } \left( \exp\left( \alpha  | x - y | / \sqrt { t } \right)  d x \right) }
		& \lesssim  \frac{\left( \int _ { 0 } ^ { \infty } \exp\left( - c u ^ { 2 } +  \alpha u  \right) u ^ { n - 1 } d u \right) ^ { 1 / p } }  { t ^{ ( n + 1 ) / 2 - n / ( 2 p )  } }  \\
		& \lesssim   \frac { \left(   M _ { 1 } + M _ { 2 } \right) ^ { 1 / p } }
		{ t ^{ ( n + 1 ) / 2 - n / ( 2 p )  } } ,
	\end{align*}
	where
	\begin{equation*}
		\left\{\begin{aligned}
			& M _ { 1  }:=  \int _ { -  \alpha / ( 2 c  ) } ^ {  \alpha / ( 2 c  ) }
			\exp\left( - c v ^ { 2 } \right) ( v +  \alpha / ( 2 c  )  ) ^ { n - 1 } d v  ;   \\ 
			& M _ { 2 }:=  \int _ {  \alpha / ( 2 c  ) } ^ { \infty }
			\exp\left( - c v ^ { 2 } \right) ( v +  \alpha / ( 2 c  )  ) ^ { n - 1 } d v   .
		\end{aligned}\right.
	\end{equation*}
	Obviously, $M _ { 2 } \lesssim  1 $. For $ M _ { 1 } $, we have
	\begin{align*}
	M _ { 1 } & = \sum _ { k = 0 } ^ { n - 1 } 
	\frac{ C _ { n - 1 } ^ { k }} { (  \alpha / ( 2 c  ) ) ^ { 1 + k - n } }  
	\int _ { -  \alpha / ( 2 c  ) } ^ { \alpha / ( 2 c  ) } 	\exp\left( - c v ^ { 2 } \right)  v ^ { k } d v  \\
	& = \sum _ { k = 0 } ^ { n - 1 } C _ { n - 1 } ^ { k }
	(  \alpha / ( 2 c  ) ) ^ { n - 1 - k } M _ { 1 , k } ,
	\end{align*}
	where
	\begin{equation*}
		M _ { 1 , k } =
		\left\{\begin{aligned}
			& 0 ,   &k \; is \; odd ;   \\ 
			&  2 \int _ { 0 } ^ {  \alpha / ( 2 c  ) }	\exp\left( - c v ^ { 2 } \right)  v ^ { k } d v  ,
			& k \; is \; even . \\
		\end{aligned}\right.
	\end{equation*}
	Then we can obtain $M _ { 1 } \lesssim  1 $, which completes the proof of Proposition \ref{Lbdd}.
\end{proof}

By Propositions \ref{K bdd} \& \ref{Q bdd}, we can obtain a similar result as follows.
\begin{proposition}\label{corollary}
	Suppose that $\left(a _ { i j } ( x )\right)$ satisfies the conditions $(A1) \; \& \; (A2)$, and that $V \in B _ { q }$ for some $q > n / 2 $. For $\alpha > 0$, $ p \in [1,\infty) $, we have the following $L ^ { p }$-estimate:
	\begin{equation*}
		\left\|  K _ { t } ^ { L } \left( \cdot , y \right) \right\|
		_ { L ^ { p } \left( \exp\left( \alpha  | x - y | / \sqrt { t } \right)  d x \right) }
		+ 
		\left\|  Q _ { t  , m } ^ { L } \left( \cdot , y \right) \right\|
		_ { L ^ { p } \left( \exp\left( \alpha  | x - y | / \sqrt { t } \right)  d x \right) }
		\lesssim t ^{ n / ( 2 p ) - n / 2 } .
	\end{equation*}
\end{proposition}

To study the Lipschitz property of the kernel function, the $L^p$-boundedness of the operator $\nabla^2 L_0^{-1}$ is required in our proof. Applying the boundedness equivalence in \cite[Theorem B]{avell}, we impose the additional condition $(A3)$ on the coefficients $a_{ij}(x)$, which further enables us to establish the Lipschitz continuity of $\nabla_x K_t^{L}(\cdot,\cdot)$ under the strengthened assumption that $\big(a_{ij}(x)\big)$ satisfies the conditions $(A 1) \; \& \; (A2) \; \& \; (A3) $.

\begin{lemma} \label{x K smo}
	Suppose that $\left(a _ { i j } ( x )\right)$ satisfies the conditions $(A 1) \; \& \; (A2) \; \& \; (A3) $, and that $V \in B _ { q }$ for some $q > n$. Let $\delta ^ { \prime } = 1 - n / q $. For every $N > 0 $, there exist constants $C _ { N } > 0$ and $c > 0 $ such that for all $x , y \in \mathbb { R } ^ { n } $, $t > 0$ and $| h | < | x - y | / 4 $,
	\begin{align*}
		\left| \nabla _ { x } K _ { t } ^ { L } ( x + h , y )
		- \nabla _ { x } K _ { t } ^ { L } ( x , y ) \right|   \leq
		\frac { C _ { N } \left( { | h |} / { \sqrt{t } } \right) ^{ \delta ' } exp( - c \left| x - y \right| ^ { 2 } / t ) } { t ^ { ( n + 1 ) / 2 } 	\left( 1 + { \sqrt { t } } / { \rho \left( x \right) } + { \sqrt { t } } / { \rho \left( y \right) } \right) ^ {  N } } .
	\end{align*}
\end{lemma}

\begin{proof}
	Let $\Gamma _ { 0 } ( \cdot , \cdot )$ be the fundamental solution of
	$- div \left( A \cdot \nabla \right)$ in $\mathbb { R } ^ { n } $.
	Assume that $u ( \cdot , \cdot )$ is a weak solution to the equation:
	\[\partial _ { t } u + L u =
	\partial _ { t } u + ( - div \left( A (x) \cdot \nabla u \right) )  + V u = 0 .\]
	Similarly to Lemma \ref{x K bdd}, we know that for any $y _ { 0 } \in \mathbb { R } ^ { n } $, the function $ u \left(x,t\right) := K_{t}^{L}\left(x,y_{0}\right)$ solves the heat equation:
	$\partial _ { t } u + L u = 0$ on $ \mathbb { R } ^ { n+1 } _{+} $.
	We can obtain
	\begin{align*}
		&- \sum _ { i ,j  = 1 } ^ { n } \nabla _ { x _ { i } }
		\left( a _ { i j } ( x ) \nabla _ { x _ { j } }  ( u \eta )\right)  \\
		& = -\partial_{t} u \cdot  \eta - Vu \eta - 2 \sum _ { i ,j  = 1 } ^ { n }
		a_{ij} (x) \nabla _ { x _ { i } } u \cdot \nabla _ { x _ { j } } \eta   \\
		& \quad - \sum _ { i ,j  = 1 } ^ { n }  a_{ij} (x)
		\left(  \nabla _ { x _ { i } }  \nabla _ { x _ { j } } \eta  \cdot u \right)
		- \sum _ { i ,j  = 1 } ^ { n }  \nabla _ { x _ { i } }  a_{ij} (x)
		\nabla _ { x _ { j } } \eta  \cdot u .
	\end{align*}
	Then, for any $x \in B\left(x_{0},\epsilon^{3} \left(\sqrt{t} + |x_{0}-y_{0}| \right) \right) $, it holds
	\begin{align*}
		{\nabla }_ {x} ^ {2}  u ( x , t )
		&= \int _ { \mathbb { R } ^ { n } }  {\nabla }_ {x} ^ {2}  \Gamma _ { 0 } ( x , y )
		\Bigg\{ -  \eta (y) \partial_{t} u (y,t)  - 2 \sum _ { i ,j  = 1 } ^ { n }  a_{ij} (y) \nabla _ { i }  u (y,t)
		\nabla _ { j } \eta (y) \\
		& \quad
		- V ( y ) u ( y , t ) \eta ( y )	- \sum _ { i ,j  = 1 } ^ { n }  a_{ij} (y)
		\nabla _ { i }  \nabla _ {j } \eta  \cdot u  	\\
		& \quad - \sum _ { i ,j  = 1 } ^ { n }  \nabla _ { i }  a_{ij} (y)
		\nabla _ { j } \eta  \cdot u  \Bigg\} d y  .
	\end{align*}
	Since $\left(a _ { i j } ( x )\right)$ satisfies equations (\ref{A}), we have
	$\left\| \nabla _ { x } ^ { 2 } u ( \cdot , t ) \right\| _ { q }
	\leq \sum \limits _ { i = 1 } ^ { 5 } \left\| S _ { i } ( \cdot , t ) \right\| _ { q }$, where
	\begin{equation*}
		\left\{\begin{aligned}
			& S_ { 1 } ( x, t ):= \int_ { \mathbb{ R } ^ {n } }
			\left| {\nabla }_ {x} ^ {2} \Gamma _ { 0 } ( x , y ) \right | \cdot | \eta (y) | \cdot
			| \partial_{t} u (y,t) | d y ;  \\
			& S_ { 2 } ( x, t ):= \int_ { \mathbb{ R } ^ {n } }
			\left| {\nabla }_ {x} ^ {2} \Gamma _ { 0 } ( x , y ) \right | \cdot | \eta (y) | \cdot
			| V (y) | \cdot | u (y,t) | d y ;  \\
			& S_ { 3 } ( x, t ):= \int_ { \mathbb{ R } ^ {n } }
			\left| {\nabla }_ {x} ^ {2} \Gamma _ { 0 } ( x , y ) \right | \cdot
			\sum _ { i ,j  = 1 } ^ { n } \left|  a_{ij} (y) \right|  \cdot
			|  \nabla _ { i }  \nabla _ { j }  \eta (y) | \cdot  |  u (y,t) | d y ;  \\
			& S_ { 4 } ( x, t ):= 2 \int_ { \mathbb{ R } ^ {n } }
			\left| {\nabla }_ {x} ^ {2}  \Gamma _ { 0 } ( x , y ) \right | \cdot
			\sum _ { i ,j  = 1 } ^ { n } \left|  a_{ij} (y) \right|  \cdot
			|  {\nabla}_{j} \eta (y)  | \cdot   | {\nabla}_{i}  u (y,t) | d y ;  \\
			& S_ { 5 } ( x, t ):=  \int_ { \mathbb{ R } ^ {n } }
			\left| {\nabla }_ {x} ^ {2}  \Gamma _ { 0 } ( x , y ) \right | \cdot
			\sum _ { i ,j  = 1 } ^ { n } \left| {\nabla}_{i} a_{ij} (y) \right|  \cdot
			|  {\nabla}_{j} \eta (y)  | \cdot   |   u (y,t) | d y .
		\end{aligned}\right.
	\end{equation*}
	
	Now, we estimate the terms $\left\| S _ { i } ( \cdot , t ) \right\| _ { q } , i =
	1 , 2 , 3 , 4 $, separately. For the term $\left\| S _ { 1 } ( \cdot , t ) \right\|
	_ { q  }$, owing to the $L ^ { p }$-boundedness of the operator $ \nabla ^ { 2 } L _ { 0 } ^ { - 1 } $ (cf. \cite[Theorem B]{avell}),
	\begin{align*}
		\| S_ { 1 } ( \cdot , t ) \|_ { q }
		&\lesssim \left\{ \sup _ { y \in B \left( x _ { 0 } , \epsilon \left(\sqrt{t}
			+ |x_{0}-y_{0}| \right) \right) }
		\left| \partial _ { t } u ( y , t ) \right| \right\}
		\left\|   \int_ { \mathbb{ R } ^ {n } }  | \eta ( y ) |
		| \nabla _ { x } ^ { 2 } \Gamma _ { 0 } ( \cdot , y ) | d y \right\| _ { q }  \\
		&\lesssim { \left(\sqrt{t} + |x_{0}-y_{0}| \right) } ^ {n / q }   \left\{ \sup _ { y \in B \left( x _ { 0 } , \epsilon \left(\sqrt{t} + |x_{0}-y_{0}| \right) \right) }
		\left| \partial _ { t } u ( y , t ) \right| \right\}    .
	\end{align*}
	
	For the term $S _ { 2 } $, by Lemma \ref{xiao V 2} and the condition $V \in B _ { q }$,
	we can obtain
	\begin{align*}
		\left\| S _ { 2 } ( \cdot , t ) \right\| _ { q }
		& \lesssim 
		\left(  \int _ { B \left( x _ { 0 } , \epsilon \left(\sqrt{t} + |x_{0}-y_{0}| \right) \right) }
		V ^ { q } ( y ) d y  \right) ^ { 1 / q } 
		\sup _ { y \in B \left( x _ { 0 } , \epsilon \left(\sqrt{t} + |x_{0}-y_{0}| \right) \right) } | u ( y , t ) |    \\
		& \lesssim \frac{ \left( 1 +  ( \sqrt{t} + |x_{0}-y_{0}| )  / { \rho \left( x _ { 0 } \right) } \right) ^ { m _ { 0 } }  }{ { \left(\sqrt{t} + |x_{0}-y_{0}| \right)} ^ { 2 - n / q } }
	   \sup _ { y \in B \left( x , \epsilon \left(\sqrt{t} + |x_{0}-y_{0}| \right) \right) } | u ( y , t ) |   .
	\end{align*}
	
	The estimate of $S _ { 3 }$ is similar to that of $S _ { 1 } $. Noting that $\eta = 1$ on
	$B \left( x _ { 0 } ,\epsilon ^ { 2 } \left(\sqrt{t} + |x_{0}-y_{0}| \right)\right) ,$
	we can obtain
	\begin{align*}
		\left\| S _ { 3 } ( \cdot , t ) \right\| _ { q }
		& \lesssim \left\| \left\{ \sup _ { y \in B \left( x _ { 0 } , \epsilon \left(\sqrt{t} + |x_{0}-y_{0}| \right) \right)} | u ( y , t ) | \right\}
		\Delta \eta \right\| _ { q } .
	\end{align*}
For $x \in B\left(x_0, \epsilon^3(\sqrt{t} + |x_0 - y_0|)\right)$ and
\[
\epsilon^2(\sqrt{t} + |x_0 - y_0|) < |y - x_0| < \epsilon(\sqrt{t} + |x_0 - y_0|),
\]
we have $|x - y| \sim \sqrt{t} + |x_0 - y_0|$. It follows that
\[
\left\| S_3(\cdot, t) \right\|_q \lesssim \left( \sqrt{t} + |x_0 - y_0| \right)^{\frac{n}{q} - 2} \sup_{y \in B\left( x_0, \epsilon(\sqrt{t} + |x_0 - y_0|) \right)} |u(y,t)|.
\]

	At last, following the same procedure, we obtain
	\begin{equation*}
		\left\{\begin{aligned}
			\| S _ { 4 } ( \cdot , t ) \| _ { q } 	
			& \lesssim { \left(\sqrt{t} + |x_{0}-y_{0}| \right) }  ^ { n / q - 1 }
			\sup _ { y \in B \left( x _ { 0 } , \epsilon \left(\sqrt{t} + |x_{0}-y_{0}| \right)  \right)}  | \nabla u ( y , t ) | ;  \\
			\left\| S _ { 5 } ( \cdot , t ) \right\| _ { q }
			& \lesssim { \left(\sqrt{t} + |x_{0}-y_{0}| \right) }  ^ { n / q - 2 }
			\sup _ { y \in B \left( x _ { 0 } , \epsilon \left(\sqrt{t} + |x_{0}-y_{0}| \right)  \right)} | u ( y , t ) | .
		\end{aligned}\right.
	\end{equation*}
	
	Let $ u (x,t) = K _ { t } ^{L} \left( x , y _ { 0 } \right) $. For any $y\in B\left(x_{0},\epsilon \left(\sqrt{t} + |x_{0}-y_{0}| \right) \right)$, we have
	$ \sqrt{t} + |x_{0}-y_{0}| \sim  \sqrt{t} + |y-y_{0}| $.
	From the estimate for $\left\| \nabla_{ x }^{2} u ( x , t )\right\| _ { q }$, Proposition \ref{Q bdd} and Lemma \ref{special mV}, we have	
	\begin{align*}
		& \left\| \nabla_{x} ^{2}  K _ { t } ^{L} \left( x , y _ { 0 } \right) \right\|_{q}  \\
		& \lesssim  \sup _ { y \in B \left( x _ { 0 } , \epsilon \left(\sqrt{t} + |x_{0}-y_{0}| \right)  \right) } \left\{ exp( - c | x - y_{0} | ^ { 2 } / t )
		\left( 1 + \frac { \sqrt { t } } { \rho ( x ) }
		+ \frac { \sqrt { t } } { \rho ( y_{0} ) } \right) ^ { - N } \right\}  \\
		& \quad \times \Bigg\{  \frac{\left(\sqrt{t} + |x_{0}-y_{0}| \right) ^ {n / q } } { t ^ { 1 + n / 2 } } + \frac{\left( \left( 1 + \frac { \sqrt{t} + |x_{0}-y_{0}| } { \rho \left( x _ { 0 } \right) } \right) ^ { m _ { 0 } } + 1 \right) }
		{t ^ {  n / 2 } }  \\
		& \quad \times \left(\sqrt{t} + |x_{0}-y_{0}| \right) ^ { n / q - 2 } 
		+ \frac{ \left(\sqrt{t} + |x_{0}-y_{0}| \right) ^ {n / q  - 1 } } { t ^ {  ( n + 1 ) / 2 } } \Bigg\} .
	\end{align*}
	Take $ 0 < \epsilon \left(\sqrt{t} + |x_{0}-y_{0}| \right) < c \rho \left( x _ { 0 } \right) $, where $ \epsilon = 1 / 2 $, $ c> 0 $.
	Since $ x\in B\left(x_{0},\epsilon^{3} \left(\sqrt{t} + |x_{0}-y_{0}| \right) \right)$,
	$\left| x - x _ { 0 } \right| < c \rho \left( x _ { 0 } \right)/ 4 $, that is,
	$\rho ( x ) \sim \rho \left( x _ { 0 } \right) $. Hence,
	\begin{align*}
		& \left| \nabla _ { x } K _ { t } ^ { L } \left( x _ { 0 } + h , y _ { 0 } \right)
		- \nabla _ { x } K _ { t } ^ { L } \left( x _ { 0 } , y _ { 0 } \right) \right|   \\
		& \lesssim | h | ^ { 1 - n / q }
		\left( \int _ { B \left(x_{0},\epsilon \left(\sqrt{t} + |x_{0}-y_{0}| \right) \right) }
		\left| \nabla _ { x } ^ { 2 } K _ { t } ^ { L } \left( x , y _ { 0 } \right) \right|
		^ { q } d x \right) ^ { 1 / q }  \\
		& \lesssim  t^{- ( n + 1 ) / 2  } \left( \frac{ | h | }{ \sqrt{ t } } \right) ^ { 1 - n / q }
		\left( 1 + \frac { \sqrt { t } } { \rho ( x _ { 0 } ) }
		+\frac { \sqrt { t } } { \rho ( y _ { 0 } ) } \right) ^ { - N }
		exp( - c | x _{0}  - y _ { 0 } | ^ { 2 } / t ) .
	\end{align*}
\end{proof}

\subsection{Estimation for $ \alpha \in ( 0 , 1 ) $ } \label{sec 3.2}
In this section, we first investigate the regularities of $K _ { \alpha , t } ^ { L } ( \cdot , \cdot ) $. 
Below we provide gradient estimates for the fractional heat kernel associated with the variable $t$.
Define the operators:
\[ 	D_{\alpha,t}^{L,\beta}(f) : = t^\beta \partial_t^\beta \left( exp(-tL^{\alpha}) f \right) ,
\quad \alpha \in (0,1), \, \beta > 0. \]
Let $ D_{\alpha,t}^{L,\beta}(\cdot, \cdot) $ denote the integral kernels of the operator $ D_{\alpha,t}^{L,\beta} $. Then the following proposition can be derived.

\begin{proposition} \label{bdd D}
	Let $\alpha \in ( 0 , 1 ) , V \in B _ { q } , q > n / 2$ and $\beta > 0 $.
	\item[\rm(i)] 	For every $N > \alpha \beta $, there exists a constant ${ C } _ { N } > 0 $ such that
	\begin{equation} \label{equ bdd D}
		\left| D _ { \alpha , t } ^ { L , \beta } ( x , y ) \right|
		\leq \frac { C _ { N } t ^ { \beta } } { ( t ^ { 1 / ( 2 \alpha ) } + | x - y | )
			^ { n + 2 \alpha \beta } } \left( 1 + \frac { t ^ { 1 / ( 2 \alpha ) } } { \rho ( x ) }
		+ \frac { t ^ { 1 / ( 2 \alpha ) } } { \rho ( y ) } \right) ^ { - N } .
	\end{equation}
	\item[\rm(ii)]
	Let $0 < \delta ^ { \prime } \leq  \min \left\{ 2 \alpha , \delta _ { 0 } \right\} .$
	For $N > \alpha \beta $, there exists a constant $C _ { N } > 0 $ such that for
	$| h | \leq t ^ { 1 / ( 2 \alpha ) } $,
	\begin{equation} \label{smo D}
		\left| D _ { \alpha , t } ^ { L , \beta }  ( x + h , y )
		- D _ { \alpha , t } ^ { L , \beta }  ( x , y ) \right|
		\leq \frac { C _ { N } t ^ { \beta } \left( { | h | } / {  t ^ { 1 / ( 2 \alpha ) } } \right) ^ { \delta ' }   	\left( 1 + \frac { t ^ { 1 / ( 2 \alpha ) }  } { \rho ( x ) }
			+ \frac { t ^ { 1 / ( 2 \alpha ) }  } { \rho ( y ) } \right) ^ { - N } } { \left( t ^ { 1 / ( 2 \alpha ) } + | x - y | \right) ^ { n + 2 \alpha \beta } } .
	\end{equation}
	\item[\rm(iii)]
	Let $0 < \delta ^ { \prime } \leq \min \left\{ 2 \alpha , 2 \alpha \beta , \delta _ { 0 } \right\} .$ For every $N > \delta $, there exists a constant $ { C } _ { N }  $ such that
	\begin{equation} \label{RD}
		\left| \int _ { \mathbb { R } ^ { n } }  D _ { \alpha , t } ^ { L , \beta } ( x , y )  d y \right|
		\leq C _ { N } \frac { \left(  t ^ { 1 / ( 2 \alpha ) } / \rho ( x ) \right) ^ { \delta ' } }
		{ \left( 1 + t ^ { 1 / ( 2 \alpha ) } / \rho ( x ) \right) ^ { N } }
		, \quad x \in \mathbb { R } ^ { n } .
	\end{equation}
\end{proposition}
\begin{proof}
	We omit the proof, as it follows similarly to \cite[Propositions 14–16]{li}.
\end{proof}

Next, we give the gradient estimate of $K _ { \alpha , t } ^ { L } ( \cdot , \cdot ) $ related with the spatial variables. Let $ D_{\alpha, t}^L := t^{1/(2\alpha)} \nabla_x exp(-tL^{\alpha}) $, and denote its kernel by $D_{\alpha, t}^L(x, y) := t^{1/(2\alpha)} \nabla_x K_{\alpha, t}^L(x, y)$.

\begin{proposition} \label{t x frac}
	Suppose $\alpha \in (0,1)$ and $V \in B _ { q }$ for some $q > n $.
	\item[\rm(i)] Suppose that $\left(a _ { i j } ( x )\right)$ satisfies the conditions $ (A1) \; \& \; (A2)$. For every $N > 0$, there exists a constant ${ C } _ { N } > 0 $ such that for all $x , y \in \mathbb { R } ^ { n } $ and $t > 0 $,
	\begin{equation*}
		\left| D _ { \alpha , t } ^ { L } ( x , y )  \right|
		\leq  \frac {  C_ {N} t }
		{ ( t ^ { 1 / (2 \alpha) } + | x - y | ) ^ { n + 2 \alpha } }
		\left( 1 + \frac { t ^ { 1 / (2 \alpha) } } { \rho ( x ) } +
		\frac { t ^ { 1 / (2 \alpha )} } { \rho ( y ) } \right) ^ { - N }.
	\end{equation*}
	\item[\rm(ii)] Suppose that $\left(a _ { i j } ( x )\right)$ satisfies the conditions $ (A1) \; \& \; (A2) \; \& \; (A3)$.
	Let $0 < \delta ^ { \prime } < \delta _ { 1 } : = 1 - n / q $. For every $N > 0$,
	there exists a constant $ { C } _ { N } > 0 $ such that for
	all $| h | < | x - y | / 4  $,
	\begin{align*}
		\left| D _ { \alpha , t } ^ { L } ( x + h , y )
		-  D _ { \alpha , t } ^ { L } ( x , y ) \right|
		\leq  \frac { C_ {N} t  \left(  { | h | } / { t ^ { 1 / ( 2 \alpha ) } } \right) ^ { \delta ' } \left( 1 + \frac { t ^ { 1 / (2 \alpha) } } { \rho ( x ) } +
			\frac { t ^ { 1 / (2 \alpha )} } { \rho ( y ) } \right) ^ { - N } } { ( t ^ { 1 / ( 2 \alpha ) } + | x - y | )  ^ { n +  2 \alpha  } } .
	\end{align*}
	\item[\rm(iii)] Let $\alpha \in ( 0 , 1 / 2 - n / ( 2 q ) ) $. For every $N > 0 $,
	\begin{align*}
		\left| D _ { \alpha , t } ^ { L } ( 1 ) ( x ) \right|
		\lesssim \min \left\{     \left( \frac { t ^ { 1 / ( 2 \alpha ) }  } { \rho ( x ) } \right)
		^ { 1 + 2 \alpha  } ,    \left( \frac {t ^ { 1 / ( 2 \alpha ) }  } { \rho ( x ) } \right)
		^ { - N } \right\}  .
	\end{align*}
\end{proposition}

\begin{proof}
	For (i), by the subordinative formula (\ref{subor}) and Lemma \ref{x K bdd}, we have
	\begin{align*}
		\left| \nabla _ { x } K _ { \alpha , t } ^ { L } ( x , y ) \right|
		& \leq \int _ { 0 } ^ { \infty }
		\frac{ exp( - c | x - y| ^ { 2 } / (t ^ { 1 / \alpha } u ) )   }
		{ t ^ {(n+1) / ( 2 \alpha )  }   u ^ { ( n + 3 ) / 2 + \alpha }   }  \\ 
		& \quad \times  \left( 1 + \frac { t ^ { 1 / (2\alpha) } \sqrt { u } } { \rho ( x ) } \right) ^ { - N }  	\left( 1 + \frac { t ^ { 1 / (2\alpha) } \sqrt { u } } { \rho ( y ) } \right) ^ { - N } d u   \\
		&\leq  \left( \frac { t ^ { 1 / (2\alpha) } } { \rho ( x ) } \right) ^ { - N }
		\left( \frac { t ^ { 1 / (2\alpha) }  } { \rho ( y ) } \right) ^ { - N }
		\frac { t ^ { 1 + N / \alpha } } { | x - y | ^ { 2 \alpha + 2 N + n + 1 } } .
	\end{align*}
	
	On the other hand, by Lemma \ref{x K bdd} and changing variables $\tau = s / t ^ { 1 / \alpha } $, we obtain
	\begin{align*}
		\left| \nabla_{x} K _ { \alpha , t } ^ { L } ( x , y ) \right|
		& \leq
		\int _ { 0 } ^ { \infty }  \frac{ C _ { N } s ^ { - ( n + 1 ) / 2 } \eta _ { 1 } ^ { \alpha  } \left( { s } / { t^{1/\alpha} } \right)  }{ t ^ { 1 / \alpha} \left( 1 + { \sqrt { s } } / { \rho ( x ) } \right) ^ {  N } 	\left( 1 +  { \sqrt { s } } / { \rho ( y ) } \right) ^ {  M } }  d s   \\
		& \leq \frac { C _ { N } \left( { t ^ { 1 / (2\alpha) }  } / { \rho ( x ) } \right) ^ { - N } } { t ^ { ( n + 1 ) / (2\alpha) } \left(  { t ^ { 1 / (2\alpha) }  } / { \rho ( y ) } \right) ^ {  N } } .
	\end{align*}
	
	Finally,
	\begin{align*}
	& \left(  { t ^ { 1 / (2\alpha) }  } / { \rho ( x ) } \right) ^ { N }
	\left(  { t ^ { 1 / (2\alpha) }  } / { \rho ( y ) } \right) ^ { N }
	\left| \nabla_{x} K _ { \alpha , t } ^ { L } ( x , y ) \right|  \\
	& \leq  \min \left\{ \frac{ C _ { N } }{ t ^ {  ( n + 1 ) / (2\alpha) } },
	\frac { C _ { N }  t ^ { 1 + N / \alpha } } { | x - y | ^ { n + 1 + 2 N + 2 \alpha } } \right\} .
	\end{align*}

	The arbitrariness of $ N $ indicates that
	\[\left| t ^{ 1/ (2 \alpha ) } \nabla _ { x } K _ { \alpha , t } ^ { L } \left( x _ { 0 } , y _ { 0 } \right) \right|
	\leq  \frac { C_ {N} t \left( 1 +  { t ^ { 1 / (2 \alpha) } } / { \rho ( x ) }  \right) ^ { - N }
		\left( 1 +  { t  ^ { 1 / (2 \alpha ) } } / { \rho ( y ) } \right) ^ { - N } } { ( t ^ { 1 / (2 \alpha) } + | x - y | ) ^ { n + 2 \alpha } }.  \]
	
	For (ii), by the subordinative formula (\ref{subor}) and Lemma \ref{x K smo}, we can obtain
	\begin{align*}
		\left| \nabla _ { x }  K _ { \alpha , t } ^ { L } ( x + h , y )
		- \nabla _ { x }  K _ { \alpha , t } ^ { L } ( x , y ) \right|
		\leq \frac { C _ { N } t ^ { 1 + N / \alpha } \left(  { | h | } / { | x - y | }  \right) ^ { \delta ^ { \prime } } \left(  \frac{ t ^ { 1 / ( 2 \alpha ) } } { \rho ( y ) } \right) ^ { - N } } { | x - y | ^ { 2 \alpha  + 2 N + n + 1 }
			\left(  { t ^ { 1 / ( 2 \alpha ) } } / { \rho ( x ) } \right) ^ {  N } } .
	\end{align*}

	On the other hand, we can deduce from Lemma \ref{x K smo} that
	\begin{align*}
		\left| \nabla _ { x }  K _ { \alpha , t } ^ { L } ( x + h , y )
		- \nabla _ { x }  K _ { \alpha , t } ^ { L } ( x , y ) \right|
		\leq \frac {C _ { N } \left(  { | h | } / {  t ^ { 1 / ( 2 \alpha ) }  } \right) ^ { \delta ^ { \prime } } 	\left(  { t ^ { 1 / ( 2 \alpha ) } } / { \rho ( y ) } \right) ^ { - N } } { t ^ { ( n + 1 ) / ( 2 \alpha ) } 	\left(  { t ^ { 1 / ( 2 \alpha ) } } / { \rho ( x ) } \right) ^ { N } } .
	\end{align*}
	
	Finally, since $N$ was arbitrary, we conclude that (ii) holds.

	At last, we will divide the proof of (iii) into two cases.
	
	Case 1: $t ^ { 1 / ( 2 \alpha ) } > \rho ( x ) $. By (i) of Proposition \ref{t x frac}, we use a direct computation to obtain
	\begin{align*}
		\left| t ^ { 1 / ( 2 \alpha ) } \nabla _ { x } \exp\left( - t L ^ { \alpha  } \right) ( 1 ) ( x ) \right| \lesssim  \left(  \frac { t ^ { 1 / ( 2 \alpha ) } } { \rho ( x ) } \right) ^ { - N }  .
	\end{align*}
	
	Case 2: $ t ^ { 1 / ( 2 \alpha ) } \leq \rho ( x ) $. It follows from (\ref{subor}) that
	\[t ^ { 1 / ( 2 \alpha ) } \nabla _ { x } \exp\left( - t L ^ { \alpha  } \right) ( 1 ) ( x )
	= t ^ { 1 / ( 2 \alpha ) } \nabla _ { x } \int _ { \mathbb{ R } ^ { n } }
	K _ { \alpha , t } ^ { L } ( x , y ) d y = I _ { 1 } + I _ { 2 } ,\]
	where
	\[\left\{ \begin{aligned}
		&I _ { 1 } : = t ^ { 1 / ( 2 \alpha ) } \int _ { { \rho ( x ) } ^ { 2 } } ^ { \infty }
		\eta _ { t } ^ { \alpha } ( s )  \left( \int _ { \mathbb { R } ^ { 2 } } \nabla _ { x } K _ { s } ^ { L } ( x , y ) d y \right) d s ; \\
		&I _ { 2 } : = t ^ { 1 / ( 2 \alpha ) } \int _ { 0 } ^ { { \rho ( x ) } ^ { 2 } }
		\eta _ { t } ^ { \alpha } ( s )  \left( \int _ { \mathbb { R } ^ { 2 } } \nabla _ { x } K _ { s } ^ { L } ( x , y ) d y \right) d s .
	\end{aligned} \right.\]
	By Lemma \ref{x K bdd}, taking $N$ sufficiently large, we can use the change of variable to obtain
	\begin{equation} \label{I1}
		\begin{aligned}
			I _ { 1 } & \lesssim t ^ { 1 / ( 2 \alpha ) } \int _ { \rho ^ { 2 } ( x ) } ^ { \infty } \eta _ { t } ^ { \alpha } ( s ) \int_{ 0 }^{\infty} \frac{ \exp\left( - c u^{2} \right) } {u^{ 1 - n }}du \frac { d s } { \sqrt { s } }  \\
			& \lesssim  \int _ { \rho ^ { 2 } ( x ) } ^ { \infty }
			\frac { t ^ { 1 + 1 / ( 2 \alpha ) } } { s ^ { \alpha  + 3 / 2 } } d s \lesssim
			\left( \frac { t ^ { 1 / ( 2 \alpha ) } } { \rho ( x ) } \right) ^ { 1 + 2 \alpha } .
		\end{aligned}
	\end{equation}
	For $I _ { 2 } $, it follows from the perturbation formula (see \cite{engel}, 2.12 Proposition), that
	\[h _ { u } ( x , y ) - K _ { u } ^ { L } ( x , y )
	= \int _ { 0 } ^ { u } \int _ { \mathbb { R } ^ { n } }  h _ { s } ( x , z ) V ( z )
	K _ { u - s } ^ { L} ( z , y ) d z d s , \]
	where $ h _ { u } ( x , y ) $ is the integral kernels of the semigroup
	$\{ \exp\left( - u L  _ { 0 } \right) \} _ { u > 0 } $ on $L ^ { 2 } ( \mathbb{ R }^{n+1}_{+} )$ generated by $L _ { 0 } $. Then, for $\delta = 2 - n / q > 1 $,
	\begin{align*}
		\left| \sqrt { u } \nabla _ { x }  \exp\left( - u L \right) ( 1 ) ( x ) \right|
		& = \left| \sqrt { u }  \int _ { \mathbb { R } ^ { n } } \nabla _ { x }
		\left( K_{u}^{L} (x,y) - h_{u} (x,y)  \right) d y \right|  \\
		& \lesssim \int _ { 0 } ^ { u } \sqrt { \frac { u } { s } }  \int _ { \mathbb { R } ^ { n } }
		 \frac{ exp ( - c |x-z|^{2} / s ) }{ s ^ { n / 2} } V (z) d z d s   \\
		& \lesssim \left( \frac { \sqrt { u } } { \rho ( x ) } \right) ^ { \delta }
	\end{align*}
	
	Therefore, noting that $0 < 2 \alpha < 1 - n / q $, we can use the change of variables to obtain
	\begin{align*}
		I _ { 2 }
		\lesssim t ^ { 1 / ( 2 \alpha ) } \int _ { 0 } ^ { \rho ^ { 2 } ( x ) }
		\frac { 1 } { t ^ { 1 / \alpha } } \eta _ { 1 } ^ { \alpha }
		\left( s / t ^ { 1 / \alpha } \right) \frac { 1 } { \sqrt { s } }
		\left( \frac { \sqrt { s } } { \rho ( x ) } \right) ^ { \delta } d s
		\lesssim \left( \frac { t ^ { 1 / ( 2 \alpha )  } } { \rho ( x ) } \right) ^ { 1 + 2 \alpha } ,
	\end{align*}
	which, together with (\ref{I1}), implies that
	\begin{equation*}
		\left|  D _ { \alpha , t } ^ { L } ( 1 ) ( x ) \right|  \lesssim
		\left( \frac { t ^ { 1 / ( 2 \alpha )  } } { \rho ( x ) } \right) ^ { 1 + 2 \alpha }.
	\end{equation*}
\end{proof}

Let $ \widetilde{D}_{\alpha, t} ^ {L, \beta } (\cdot , \cdot )$ denote the kernel of
$ t^{\beta / \alpha } L^{ \beta } exp( - t L ^ { \alpha } )  $.
\begin{lemma} \label{bdd}
	Let $\alpha \in ( 0 , 1 ) $, $ \beta > 0 $, and $V \in B _ { q } $, $q > n / 2 $.
	\item[\rm(i)] For every $N > 0$, there exists a constant ${ C } _ { N } > 0 $ such that
	\begin{equation} \label{bdd D/}
		\left| \widetilde{D}_{\alpha, t} ^ {L, \beta } (x,y) \right|
		\leq \frac { C _ { N } t^{\beta / \alpha} } { ( t ^ { 1 / (2\alpha ) } + | x - y |) ^ { n + 2 \beta } }
		\left( 1 + \frac {  t ^ { 1 / (2\alpha ) }  } { \rho ( x ) } +
		\frac {  t ^ { 1 / (2\alpha ) }  } { \rho ( y ) } \right) ^ { - N }.
	\end{equation}
	\item[\rm(ii)] Let $0 < \delta ^ { \prime } \leq \delta $. For $N > 0 $, there exists a constant $C _ { N } > 0 $ such that for all $| h | \leq t ^ { 1 / (2 \alpha)} $,
	\begin{equation}
		\left|  \widetilde{D}_{\alpha, t} ^ {L, \beta }  ( x + h , y )
		-  \widetilde{D}_{\alpha, t} ^ {L, \beta }  ( x , y ) \right|
		\leq \frac { C _ { N } t^{\beta / \alpha} \left(  { | h | } / {  t ^ { 1 / ( 2\alpha ) } } \right) ^ { \delta ^ { ' } } \left( 1 + \frac {  t ^ { 1 / (2\alpha ) }  } { \rho ( x ) } +
		\frac {  t ^ { 1 / (2\alpha ) }  } { \rho ( y ) } \right) ^ { - N } } { ( t ^ { 1 / (2\alpha ) } + | x - y |) ^ { n + 2 \beta } } .
	\end{equation}
	\item[\rm(iii)] Let $0 < \delta ^ { \prime } \leq \delta < \min \{ \delta_{ 0 }, 2 \beta \}$, where
	$ \delta_{ 0 }$ is a parameter in Proposition \ref{bdd frac diff} . For any $N > \delta$, there exists a constant $C _ { N } > 0 $ such that
	\begin{equation} \label{R bdd D/}
		\left| \int _ { \mathbb { R } ^ { n } }
		\widetilde{D}_{\alpha, t} ^ {L, \beta }  ( x , y ) d y \right|
		\leq C _ { N } \frac { \left( t ^ { 1 / (2\alpha ) } / \rho ( x ) \right) ^ { \delta ^ { ' } } }
		{ \left( 1 +  t ^ { 1 / (2\alpha ) } / \rho ( x ) \right) ^ { N } } .
	\end{equation}
\end{lemma}

\begin{proof}
	
	Due to the definition of the fractional operator (\ref{def frac L}), the proof of this lemma is divided into two cases: $ 0 < \beta < \alpha $ and $ \beta \geq \alpha $.

	For (i), when $\alpha \in (0,1)$, $ \beta \in ( 0 , \alpha ) $, the proof is similar to that of \cite[Proposition 2.11]{huang 2}.
	When $ \beta \geq \alpha $, let $ l \in [ 0 , 1 )$ such that
	$ l = \beta/\alpha - [ \beta/\alpha ] $. Since
	\begin{align} \label{D2}
		\widetilde{D}_{\alpha, t} ^ {L, \beta }
		= t^{\beta / \alpha} L ^ { \alpha [ \beta / \alpha ] }
		\int_{ 0 }^{\infty} \int_{0}^{s} \partial_{ r } exp( - ( t + r ) L ^ { \alpha } ) d r \frac{ d s }{ s ^ { 1 + l } } ,
	\end{align}
	then, by Proposition \ref{bdd D}, we have
	\begin{align*}
		\left| \widetilde{D}_{\alpha, t} ^ {L, \beta } (x,y) \right|
		& \leq  \int_{ 0 }^{\infty} \frac{ C_{ N } t^{\beta / \alpha} \left( 1 +  {  ( t + r ) ^ { 1 / (2\alpha ) }  } /  { \rho ( x ) } +
			{  ( t + r ) ^ { 1 / (2\alpha ) }  } / { \rho ( y ) } \right) ^ { - N } }
		{ \left( (r+t)^ { 1 / (2\alpha) } + |x-y| \right) ^ { n + 2\alpha ( 1 + [ \beta / \alpha ] ) }}
		\frac{ d r }{ r ^ { l } }  \\
		& \leq  I + II  ,
	\end{align*}
	where
	\[\left\{ \begin{aligned}
		&I  : = \int_{ 0 }^{ t + |x-y| ^ { 2 \alpha } } \frac{ C_{ N } t^{\beta / \alpha} \left( 1 +  {  t ^ { 1 / (2\alpha ) }  } / { \rho ( x ) } + { t ^ { 1 / (2\alpha ) }  } / { \rho ( y ) } \right) ^ { - N } }
		{ \left( r + t + |x-y| ^ { 2 \alpha } \right) ^ { n / ( 2\alpha ) +  1 + [ \beta / \alpha ] }}
		\frac{ d r }{ r ^ { l } }   ; \\
		&I I : = \int_{ t + |x-y| ^ { 2 \alpha } }^{\infty} \frac{ C_{ N } t^{\beta / \alpha} \left( 1 +  {  t ^ { 1 / (2\alpha ) }  } / { \rho ( x ) } + { t ^ { 1 / (2\alpha ) }  } / { \rho ( y ) } \right) ^ { - N }  }
		{ \left( r + t + |x-y| ^ { 2 \alpha } \right) ^ { n / ( 2\alpha ) +  1 + [ \beta / \alpha ] }}
		\frac{ d r }{ r ^ { l } }  .
	\end{aligned} \right.\]
	For $ I $, a direct calculus gives
	\begin{align*}
		I \leq   \frac { C_{ N } t^{\beta / \alpha} } { ( t ^ { 1 / (2\alpha ) } + | x - y |) ^ { n + 2 \beta } }
		\left( 1 + \frac {  t ^ { 1 / (2\alpha ) }  } { \rho ( x ) } +
		\frac { t ^ { 1 / (2\alpha ) }  } { \rho ( y ) } \right) ^ { - N } .
	\end{align*}
	For $ I I $, choose $ \lambda $ such that $ 1 - l < \lambda < 1 + n / ( 2 \alpha ) + [ \beta / \alpha ] $. Then
	\begin{align*}
		I I  & \leq \frac{ C_{ N } t^{\beta / \alpha}  \left( 1 + {  t ^ { 1 / (2\alpha ) }  } / { \rho ( x ) } + { t ^ { 1 / (2\alpha ) }  } / { \rho ( y ) } \right) ^ { - N } }
		{ (t + |x-y| ^ { 2 \alpha } ) ^ { n / ( 2\alpha ) +  1 + [ \beta / \alpha ] - \lambda } }
		\int_{ t + |x-y| ^ { 2 \alpha } }^{\infty}  r ^ { - \lambda - l } d r  \\
		& \leq   \frac { C_{ N } t^{\beta / \alpha} } { ( t ^ { 1 / (2\alpha ) } + | x - y |) ^ { n + 2 \beta } }
		\left( 1 + \frac {  t ^ { 1 / (2\alpha ) }  } { \rho ( x ) } +
		\frac { t ^ { 1 / (2\alpha ) }  } { \rho ( y ) } \right) ^ { - N } .
	\end{align*}
	
	For (ii), when $\alpha \in (0,1)$, $ \beta \in ( 0 , \alpha ) $, it follows from the semigroup property and Proposition \ref{bdd D}, (\ref{def frac L}) that
	\begin{align*}
		& \left| \widetilde{D}_{\alpha, t} ^ {L, \beta } (x+h,y) -  \widetilde{D}_{\alpha, t} ^ {L, \beta } (x,y)\right|   \\
		& \leq C   \int_{ 0 }^{\infty} \int_{0}^{s}
		\left| \partial_{ r }  K_{ \alpha, t+r }^{L} (x+h,y)
		- \partial_{ r }  K_{ \alpha, t+r }^{L} (x,y) \right| \frac{t^{\beta / \alpha} drds}{ s ^ { 1 + \beta / \alpha } }   .
	\end{align*}
	Similarly to the proof in (i), we can obtain
	\begin{align}\label{1}
		\left|  \widetilde{D}_{\alpha, t} ^ {L, \beta }  ( x + h , y )
		-  \widetilde{D}_{\alpha, t} ^ {L, \beta }  ( x , y ) \right|
		\leq \frac { C _ { N } t^{\beta / \alpha} \left(  { | h | } / {   t ^ { 1 / ( 2 \alpha ) }  } \right) ^ { \delta ^ { ' } } \left( 1 + \frac { t ^ { 1 / (2\alpha) } } { \rho ( x ) }
			+ \frac { t ^ { 1 / (2\alpha) } } { \rho ( y ) } \right) ^ { - N } } { ( t ^ { 1 / (2\alpha ) } + | x - y |) ^ { n + 2 \beta } } .
	\end{align}
	When $ \beta \geq \alpha $, similarly to the proof of (i), by (\ref{D2}) we can obtain (\ref{1}) as well.
	
	For (iii), when $\alpha \in (0,1)$, $ \beta \in ( 0 , \alpha ) $, the subordinative formula (\ref{subor}) and (\ref{def frac L}) give
	\begin{align*}
		\widetilde{D}_{\alpha, t} ^ {L, \beta }
		\simeq t^{\beta / \alpha}
		\int_{ 0 }^{\infty} \int_{0}^{s}  \partial_{u}
		\int_{ 0 }^{\infty} exp(- wL) \eta_{ t+u }^{\alpha} (w)d w du  \frac{ds}{s^{1+\beta / \alpha}} ,
	\end{align*}
	which, together with the change of variable $\tau = {w} / {(t+u)^{1/{\alpha}}}$, implies that
	\begin{align*}
		\widetilde{D}_{\alpha, t} ^ {L, \beta }
		& \simeq \int_{ 0 }^{\infty} t^{\beta / \alpha} \int_{ 0  }^ { s } \int_{0}^{\infty} exp(- (t+u) ^{ 1/\alpha} \tau L )
		(t+u) ^{ 1/\alpha - 1} \tau L  \eta_{ 1 }^{\alpha} (\tau)   \frac{d \tau d u ds}{s^{1+\beta / \alpha}}.
	\end{align*}
	Then by Lemma \ref{lemma} and taking $N>\delta ' $, we can obtain
	\begin{align*}
		& \left| \int _ { \mathbb { R } ^ { n } } \widetilde{D}_{\alpha, t} ^ {L, \beta }  ( x , y ) d y \right|
		=  C_{\alpha , \beta }  \left| t^{\beta / \alpha } L^{\beta} exp(-tL^{\alpha }  ) 1 (x) \right|   \\
		& \lesssim \left| \int_{\mathbb{R}^n}  \int_{ 0 }^{\infty}  \int_{ 0  }^ { s }
		\int_{0}^{\infty} K_{ (t+u) ^{ 1/\alpha} \tau } ^ {L} (x,y) (t+u) ^{ 1/\alpha - 1} \tau
		\eta_{ 1 }^{\alpha} (\tau)  d \tau d u  \frac{t^{\beta / \alpha} ds}{s^{1+\beta / \alpha}} V(y)  d y \right|  .
	\end{align*}
	
	Since the function $ \eta _ { 1 } ^ { \alpha } ( \cdot )$ is continuous, the integral
	$\int _ { 0 } ^ { 1 }  \eta _ { 1 } ^ { \alpha } ( \tau )
	\tau ^ { \delta ^ { \prime } / 2 } d \tau < \infty $.
	On the other hand, recalling that
	$ \eta _ { 1 } ^ { \alpha } ( \tau ) \leq 1 / \tau ^ { 1 + \alpha } $,
	we obtain
	\[\int _ { 1 } ^ { \infty } \eta _ { 1 } ^ { \alpha } ( \tau ) \tau ^ { \delta '/ 2 } d \tau \lesssim \int _ { 1 } ^ { \infty } \frac { 1 } { \tau ^ { 1 + \alpha } }
	\tau ^ { \delta '/ 2 }  d \tau  < \infty .\]
	Then, if $ t ^ { 1 /( 2 \alpha) } \leq \rho (x) $,
	\begin{align*}
		&\left| \int _ { \mathbb { R } ^ { n } } \widetilde{D}_{\alpha, t} ^ {L, \beta }  ( x , y ) d y \right|
		\lesssim \Big| \int_{0}^{\infty}  \int_{ 0  }^ { s }  \int_{0}^{\infty} \frac{ t^{\beta / \alpha}   \eta_{ 1 }^{\alpha} (\tau) }{ t+u } \left( \frac { \sqrt { (t+u) ^{ 1/\alpha} \tau } } { \rho ( x ) } \right) ^ {\delta '}
		d \tau d u \frac{ds}{s^{1+\beta / \alpha}} \Big| \\
		&\lesssim I_{1} + I_{2} ,
	\end{align*}
	where
	\begin{equation*}
		\left\{\begin{aligned}
			& I _ { 1 }:= \rho(x) ^{ - \delta '  }  \int_{ 0 }^{t}  \int_{ 0 }^ { s } t^{\beta / \alpha}
			\frac{ (t+u) ^{ \delta' /( 2 \alpha) } }{ t+u }d u \frac{ds}{s^{1+\beta / \alpha}} ;   \\
			& I _ { 2 }:= \rho(x) ^{ - \delta '  }  \int_{ t }^{\infty}  \int_{ 0  }^ { s } t^{\beta / \alpha}
			\frac{ (t+u) ^{ \delta' /( 2 \alpha) } }{ t+u }d u \frac{ds}{s^{1+\beta / \alpha}} .
		\end{aligned}\right.
	\end{equation*}
	As in the estimates for $ I $ and $ II $, a straightforward calculation yields
	\begin{align*}
		\left| \int _ { \mathbb { R } ^ { n } } \widetilde{D}_{\alpha, t} ^ {L, \beta }  ( x , y ) d y \right|
		\lesssim  \left( t ^ { 1 /( 2 \alpha) }  / \rho ( x ) \right) ^ { \delta ^ { ' } }  .
	\end{align*}
	
	If $ t ^ { 1 /( 2 \alpha) } > \rho (x) $, since $N > \delta'$, it follows that
	\begin{align*}
		& \left| \int _ { \mathbb { R } ^ { n } } \widetilde{D}_{\alpha, t} ^ {L, \beta }  ( x , y ) d y \right|
		\lesssim \int_{ 0 }^{\infty} t^{\beta/\alpha} \int_{ 0  }^ { s } \rho(x) ^{N - \delta '  }  \frac{ \int_{0}^{\infty}  \tau ^{ \delta' / 2 - N / 2 }
			\eta_{ 1 }^{\alpha} (\tau)  d \tau }{(t+u) ^{ - \delta' /( 2 \alpha) +  N /( 2 \alpha) + 1 }} d u \frac{ds}{s^{1+\beta/\alpha}}  \\
		& \lesssim I_{3} + I_{4} ,
	\end{align*}
	where
	\begin{equation*}
		\left\{\begin{aligned}
			& I _ { 3 }:= \rho(x) ^{N - \delta '  }  \int_{ 0 }^{t}  \int_{ 0  }^ { s } t^{\beta/\alpha}
			\frac{ (t+u) ^{ \delta' /( 2 \alpha) } }{(t+u) ^{  N /( 2 \alpha) + 1 }}d u \frac{ds}{s^{1+\beta/\alpha}} ;   \\
			& I _ { 4 }:= \rho(x) ^{N - \delta '  }  \int_{ t }^{\infty}  \int_{ 0  }^ { s } t^{\beta/\alpha}
			\frac{ (t+u) ^{ \delta' /( 2 \alpha) } }{(t+u) ^{  N /( 2 \alpha) + 1 }}d u \frac{ds}{s^{1+\beta/\alpha}} .
		\end{aligned}\right.
	\end{equation*}
	Similarly to $I _ { 1 }$ and $I _ { 2 }$, a straightforward calculation yields
	\begin{align*}
		\left| \int _ { \mathbb { R } ^ { n } } \widetilde{D}_{\alpha, t} ^ {L, \beta }  ( x , y ) d y \right| \lesssim  \left(  t ^ { 1 / ( 2 \alpha) } / \rho ( x ) \right) ^ { \delta ^ { ' } - N } .
	\end{align*}
	Based on the estimates for $I_{i}, \; i=1,2,3,4$, we can derive (\ref{R bdd D/}) with $ \beta \in ( 0 , \alpha ) $.
	When $ \beta \geq \alpha $, let $ l \in [ 0 , 1 )$ such that
	$ l = \beta/\alpha - [ \beta/\alpha ] $. By (\ref{D2}), we have
	\begin{align*}
		\widetilde{D}_{\alpha, t} ^ {L, \beta }
		= t ^ { \beta / \alpha } \int_{ 0 }^{ \infty } \int_{ 0 }^{s} \int_{ 0 }^{ \infty }
		\eta_{ 1 } ^ { \alpha } ( \tau )   \partial_{ r }^ { 1+ [ \beta / \alpha ] }
		exp( - ( t + r )^{ 1 / \alpha } \tau L  ) d \tau  d r  \frac{ d s }{ s ^ { 1 + l } }  .
	\end{align*}
	By Proposition \ref{bdd D} and the higher-order derivative formula for composite functions, we obtain
	\begin{align*}
		& \left| \int _ { \mathbb { R } ^ { n } } \widetilde{D}_{\alpha, t} ^ {L, \beta }  ( x , y ) d y \right|  \\
		& \lesssim   \int _ { \mathbb { R } ^ { n } } t ^ { \beta / \alpha } \int_{ 0 }^{ \infty }
		\int_{ 0 }^{s}   \left|  \int_{ 0 }^{ \infty }  \eta_{ 1 } ^ { \alpha } ( \tau )
		\partial_{ r }^ { 1+ [ \beta / \alpha ] }  K _{ ( t + r )^{ 1 / \alpha }  \tau } ^ { L } ( x , y )
		d \tau  \right|   \frac{  d r  d s  d y  }{ s ^ { 1 + l } }    ,
	\end{align*}
	Similarly to $I _ { 1 }$ and $I _ { 2 }$, a straightforward calculation yields
	\begin{align*}
		\left| \int _ { \mathbb { R } ^ { n } } \widetilde{D}_{\alpha, t} ^ {L, \beta }  ( x , y ) d y \right|  \lesssim  \frac{  \left( t ^ { 1 / (2\alpha ) } / \rho ( x ) \right) ^ { \delta ^ { ' } } } { \left( 1 +  t ^ { 1 / (2\alpha ) } / \rho ( x ) \right) ^ { N } } .
	\end{align*}
\end{proof}

\begin{remark}
	For the case $ \alpha = 1 $, since the polynomial-type upper bound estimate
	$$ \frac { C _ { N } t ^ { \beta } } { ( t ^ { 1 / ( 2 \alpha ) } + | x - y | )
		^ { n + 2 \alpha \beta } } \left( 1 + \frac { t ^ { 1 / ( 2 \alpha ) } } { \rho ( x ) }
	+ \frac { t ^ { 1 / ( 2 \alpha ) } } { \rho ( y ) } \right) ^ { - N } $$ in Proposition \ref{bdd D} is weaker than the Lipschitz-type upper bound estimate $$ \frac { C _ { N } } { t ^ { n / 2 } } exp( - c | x - y | ^ { 2 } / t ) \left( 1 + \frac { \sqrt { t } } { \rho ( x ) } + \frac { \sqrt { t } } { \rho ( y ) } \right) ^ { - N } $$ in Proposition \ref{Q bdd}, we similarly obtain the above result.
\end{remark}

\section{Characterizations of Campanato-type spaces associated with $L$ } \label{sec 4}

\subsection{Campanato type spaces associated with $L$ } \label{sec 4.1}
In this section, we derive the Carleson type characterization of Campanato type spaces related with $L$, which will be defined as follows.
\begin{definition}\label{def BMO}
	The space $ \Lambda _ { L , \gamma } \left( \mathbb { R } ^ { n } \right) , 0 \leq \gamma \leq 1$, is defined as the set of all locally integrable functions $f $ in $ \mathbb { R } ^ { n }$ satisfying that there exists a constant $C $ such that
	\begin{equation}\label{equ def BMO}
		\frac { 1 } { | B |  } \int _ { B } | f ( x ) - f ( B , V ) | d x \leq C | B | ^ {  \gamma / n } ,
	\end{equation}
	where the supremum is taken over all balls $B$ centered at $x _ { B }$ with radius $r _ { B } $, and
	\begin{equation*}
		f ( B , V ) : =
		\left\{ \begin{aligned}
			&f _ { B } ,   &r _ { B } < \rho \left( x _ { B } \right) ;\\
			&0,    &r _ { B } \geq \rho \left( x _ { B } \right).
		\end{aligned} \right.
	\end{equation*}
	The norm $\| f \| _ {  \Lambda _ { L , \gamma }  }$ is defined as the infimum of the constants $C $ such that (\ref{equ def BMO}), above, holds.
\end{definition}

\begin{remark}
	Let $ \Lambda_{L} $ denote the space $ \Lambda_{L,0} $, and let $ \Lambda_{\gamma} $ represent the traditional $ BMO ^{\gamma} $ space. By the classical John-Nirenberg inequality, we can establish that replacing the $L^{1}$-norm in (\ref{equ def BMO}) with $L^p$-norm for $1 < p < \infty$ yields equivalent characterizations of $ \Lambda _ { L , \gamma }$ with comparable norms. Specifically, the modified condition become:
	\begin{equation}\label{equ def BMO p}
		\left( \frac { 1 } { | B |  } \int _ { B } | f ( x ) - f ( B , V ) | ^ {p} d x \right) ^ {1/p} \leq C | B | ^ {  \gamma / n } .
	\end{equation}
\end{remark}

It is a classical result that the Hardy spaces $H ^ { p } \left( \mathbb { R } ^ { n } \right) $, for $0 < p \leq 1 $, arise as the predual spaces of the Campanato spaces (cf. \cite{fef 2}). In the subsequent development of harmonic analysis and operator theory, particularly in the early 21st century, this fundamental duality was rigorously generalized to function spaces associated with operators; see \cite{bong, duong 3, dziub 1, dziub 4, dziub 5, hofm, liu, yang 0, yang 1, yang 3, yang 4}.
For a degenerate Schr\"odinger operator $L $, Dziuba\'nski \cite{dziub 3} employed a perturbation argument to estimate the kernel of the semigroup $\{ exp(-tL) \}_{t > 0}$, which reduces to uniformly elliptic operators by taking $\omega = 1$. As an application, the author obtained the atomic characterization of the Hardy space associated with $L$ denoted by $H^1_L(\mathbb{R}^n)$.
Building upon this result, Huang, Li, and Liu \cite{huang 0} further characterized the Hardy space $H^1_L(\mathbb{R}^n)$ using square functions generated by $\{ exp(-tL) \}_{t > 0}$.
In a related development, Harboure, Salinas, and Viviani \cite{harb 2} studied the boundedness of operators linked to $L$ on various functional spaces, including weighted Lebesgue spaces, weighted Morrey spaces, and $BMO$-type spaces.
For further information, we refer readers to \cite{bui, hofm1, huang 1, huang 2, kur 0, kur 2}.

\begin{definition}\label{def HL}
	Let $0 < p \leq 1$. The Hardy space $H_L^p(\mathbb{R}^n)$ associated with the operator $L$ consists of all measurable functions $f$ satisfying:
	\begin{equation*}
		\|f\|_{H_L^p} := \left\| \sup_{s>0} |T_s^L f| \right\|_{L^p} < \infty,
	\end{equation*}
	where $\{T_s^L\}_{s>0}$ denotes the $L$-associated semigroup operators. Here, the maximal operator $T^*(f)(x) := \sup_{s>0} |T_s^L f(x)|$ is required to belong to $L^p(\mathbb{R}^n)$.
\end{definition}

Let $\delta _ { 0 } = \min \{ 1 , 2 - n / q \}$ and $ n / \left( n + \delta _ { 0 } \right) < p \leq 1 $. An atom of $H _ { L } ^ { p } \left( \mathbb { R } ^ { n } \right)$ associated with a ball $B \left( x _ { B } , r _ { B } \right)$ is a function $ a $ such that
\begin{equation*}\label{atom}
	\left\{ \begin{aligned}
		&supp \; a \subseteq B \left( x _ { B } , r _ { B } \right), &r_{B}\leq \rho(x_{B}) ;\\
		&\|a\|_{L^{\infty }} \leq |B \left( x _ { B } , r _ { B } \right)|^{-1/p}  ;\\
		&\int_{ \mathbb{ R } ^ { n } } a(x) dx =0,  &r_{B}\leq \rho(x_{B})/4   .
	\end{aligned} \right.
\end{equation*}
In \cite{yang 3}, Yang-Yang-Zhou obtained the following atomic characterization of $ H _ { L } ^ { p } \left( \mathbb { R } ^ { n } \right) $:

\begin{proposition}{ (\cite [Definition 2.6] {yang 3})}
	Let $n / \left( n + \delta _ { 0 } \right) < p \leq 1 $. $f \in H _ { L } ^ { p } \left( \mathbb { R } ^ { n } \right)$ if and only if $f = \sum _ { j } \lambda _ { j } a _ { j  } $, where
	$\{ a _ { j } \}$ are $H _ { L } ^ { p }$ -atoms and $\sum _ { j } \left| \lambda _ { j } \right| ^ { p } < \infty $.
\end{proposition}

\begin{theorem}{ (\cite [Theorem 2.1] {yang 3})}
	Let $0 \leq \gamma < 1$. Then the dual space of $H _ { L } ^ { n / ( n + \gamma ) } \left( \mathbb { R } ^ { n } \right)$ is $\Lambda _ { L ,\gamma } \left( \mathbb { R } ^ { n } \right) \left( \mathbb { R } ^ { n } \right) $. More precisely, for some function $f \in \Lambda _ { L ,\gamma } \left( \mathbb { R } ^ { n } \right)$ and all $H _ { L } ^ { n / ( n + \gamma ) }$-atoms a, continuous linear functional $\Phi$ over $H _ { L } ^ { n / ( n + \gamma ) } \left( \mathbb { R } ^ { n } \right)$ can be represented as
	\[\Phi ( a ) = \int _ { \mathbb{ R } ^ { n } } f ( x ) a ( x ) d x . \]
	Moreover, the operator norm $\| \Phi \| _ { o p } \sim \| f \| _ { \Lambda _ { L ,\gamma } } $.
\end{theorem}

We next define the Campanato–Sobolev type space associated with the operator $ L $.

\begin{definition}
	Let $0 \leq \gamma < 1$ and $\kappa \geq 0 $. The Campanato-Sobolev type space $\Lambda
	^ { \kappa  } _ { L , \gamma } \left( \mathbb { R } ^ { n } \right)$ is the collection of all measurable functions $f$ such that $L ^ { \kappa } f$ belongs to the space $\Lambda _ { L ,\gamma }
	\left( \mathbb { R } ^ { n } \right) $, and equips $f$ with the norm
	\begin{equation*}
		\displaystyle \left\| f \right\| _ { \Lambda ^ {  \kappa } _ { L , \gamma } }
		: = \displaystyle  \left\| L ^ {  \kappa } f \right\| _ { \Lambda _ { L ,\gamma } }
		<\infty.
	\end{equation*}
\end{definition}

To establish the Carleson characterization of $\Lambda_{L,\gamma}(\mathbb R^{n})$ and $\Lambda^{\kappa}_{L,\gamma}(\mathbb R^{n})$, we need the following reproducing formula related with $L$.

\begin{lemma} \label{isometry}
	For $\alpha \in ( 0 , 1 ) $ and $ \beta > 0 $, the operator $ \widetilde{D}_{\alpha, t} ^ {L, \beta } = t^{\beta/\alpha} L^{\beta} exp(-tL^{\alpha}) $
	defines an isometry from $L ^ { 2 } \left( \mathbb { R } ^ { n } \right)$ into
	$L ^ { 2 } \left( \mathbb { R } _ { + } ^ { n + 1 } , d x d t / t \right) $. Moreover, in the sense of $L ^ { 2 } \left( \mathbb { R } ^ { n } \right) $, the following holds
	\[f ( x ) = c _ { \alpha , \beta } \lim _ { N \rightarrow \infty }
	\lim _ { \epsilon \rightarrow 0  }   \int _ { \epsilon } ^ { N }
	\left( \widetilde{D}_{\alpha, t} ^ {L, \beta } \right) ^ { 2 }
	( f ) ( x ) \frac { d t } { t } .\]
\end{lemma}
\begin{proof}
	Given the spectral resolution \( dE(\lambda) \) of the operator \( L \), we have
	\[ \widetilde{D}_{\alpha, t}^{L, \beta}
	= \int_{0}^{\infty} \left( t\lambda^\alpha \right)^{\beta/\alpha} exp(-t\lambda^\alpha) \, dE(\lambda). \]
	Then for $f \in L ^ { 2 } \left( \mathbb { R } ^ { n } \right) ,$ we can obtain
	\begin{align*}
		\left\| \widetilde{D}_{\alpha, t} ^ {L, \beta } ( f ) ( x ) \right\|
		^ { 2 }  _ { L ^ { 2 } \left( \mathbb { R } _ { + } ^ { n + 1 } , d x d t / t \right) }
		& = \int _ { 0 } ^ { \infty } \left\langle \widetilde{D}_{\alpha, t} ^ {L, \beta } ( f ) , \widetilde{D}_{\alpha, t} ^ {L, \beta } ( f ) \right\rangle \frac { d t } { t }   \\  
		& = \int _ { 0 } ^ { \infty }  \int _ { 0 } ^ { \infty }  t ^ { 2 \beta / \alpha }  \lambda ^ { 2 \beta } exp( - 2 t \lambda ^ { \alpha } ) \frac { d t } { t } d E _ { f , f } ( \lambda )  \\
		& = C_{\alpha , \beta }  \| f \| _ { 2 } ^ { 2 } .
	\end{align*}
	Next, we only prove that for every pair of sequences $n _ { k } \uparrow \infty $ and
	$\epsilon _ { k }  \downarrow  0$ as $k \rightarrow \infty ,$
	\begin{equation}  \label{equL2}
		\begin{aligned}
			\lim _ { k \rightarrow \infty } \int _ { n _ { k } } ^ { n _ { k + m } }
			\left( \widetilde{D}_{\alpha, t} ^ {L, \beta }  \right) ^ { 2 }
			f ( x ) \frac { d t } { t } = \lim _ { k \rightarrow \infty }
			\int _ { \epsilon _ { k } } ^ { \epsilon _ { k + m } }
			\left( \widetilde{D}_{\alpha, t} ^ {L, \beta } \right) ^ { 2 }
			f ( x ) \frac { d t } { t } = 0 .
		\end{aligned}
	\end{equation}
	If (\ref{equL2}) holds, there exists a function $h \in L ^ { 2 } \left( \mathbb { R } ^ { n } \right) $ such that
	\[\lim _ { k \rightarrow \infty } \int _ { \epsilon _ { k } } ^ { n  _ { k } }
	\left(  \widetilde{D}_{\alpha, t} ^ {L, \beta } \right) ^ { 2 }
	f ( x ) \frac { d t } { t } = h ( x ) ,\]
	which implies that $ \forall g \in L ^ { 2 } \left( \mathbb { R } ^ { n } \right) ,$
	\begin{align*}
		\langle h , g \rangle   
		= \lim _ { k \rightarrow \infty } \int _ { \epsilon _ { k } } ^ { n _ { k } } \left\langle
		\widetilde{D}_{\alpha, t} ^ {L, \beta }  f , \widetilde{D}_{\alpha, t} ^ {L, \beta }  g
		\right\rangle \frac { d t } { t }
		= C _ { \alpha , \beta } \langle f , g \rangle .
	\end{align*}
	
	This means $h = C _ { \alpha , \beta } f $. Then we verify (\ref{equL2}).
	As $k \rightarrow \infty $, we can deduce that
	\begin{align*}
		& \left\| \int _ { n _ { k } } ^ { n _ { k + m } }
		\left( \widetilde{D}_{\alpha, t} ^ {L, \beta }  \right) ^ { 2 }
		f ( x ) \frac { d t } { t }  \right\| ^ { 2 }  _ { L ^ { 2 }  }    \\
		&= \left\| \int _ { 0 } ^ { \infty }     \left( \int _ { n _ { k } } ^ { n _ { k + m } } t ^ { 2 \beta / \alpha }  \lambda ^ { 2 \beta } exp( - 2 t \lambda ^ { \alpha } ) \frac { d t } { t } \right) d E _ { f } ( \lambda ) \right\| ^ { 2 }  _ { L ^ { 2 }  }    \\
		&\leq \int _ { 0 } ^ { \infty } \left( \int _ { n _ { k } } ^ { n _ { k + m } }
		t ^ { 2 \beta / \alpha }  \lambda ^ { 2 \beta } exp( - 2 t \lambda ^ { \alpha } )
		\frac { d t } { t } \right) ^ { 2 }     d E _ { f , f } ( \lambda ) \rightarrow 0
	\end{align*}
	since
	\[\lim _ { k \rightarrow \infty }    \left| \int _ { n _ { k } } ^ { n _ { k + m } }
	t ^ { 2 \beta / \alpha }  \lambda ^ { 2 \beta } exp( - 2 t \lambda ^ { \alpha } )
	\frac { d t } { t } \right| = 0 .\]
	
	Similarly, we can deal with the integral
	\[ \lim _ { k \rightarrow \infty } \int _ { \epsilon _ { k } } ^ { \epsilon _ { k + m } } \left( \widetilde{D}_{\alpha, t} ^ {L, \beta } \right) ^ { 2 } 	f ( x ) \frac { d t }  { t } . \]
\end{proof}

\begin{lemma}  \label{FG}
	Assume that $0 < \gamma \leq 1 $. For any pair of measurable functions $\widetilde{D}_{\alpha, s} ^ {L, \beta } (f) (x)$ and $\widetilde{D}_{\alpha, s} ^ {L, \beta } (g) (x)$ on
	$\mathbb { R } _ { + } ^ { n + 1 } $, we have
	\begin{align} \label{Sfunction}
		&\iint _ { \mathbb { R } ^ { n + 1 } _ { + } }
		\left|  \widetilde{D}_{\alpha, s} ^ {L, \beta } (f) (x) \right| \cdot
		\left|  \widetilde{D}_{\alpha, s} ^ {L, \beta } (g) (x) \right|    \frac { d x d s } { s }    \\
		&\lesssim  \sup _ { B } \left( \frac { 1 } { | B | ^ { 1 + 2 \gamma / n } }
		\int _ { 0 } ^  { r _ { B } ^ { 2 \alpha } }   \int _ { B }
		\left|  \widetilde{D}_{\alpha, s} ^ {L, \beta } (f) (x) \right| ^ { 2 } \frac { d x d s } { s } \right) ^ { 1 / 2 }  \nonumber     \\
		&\quad \times \left\{  \int _ { \mathbb{ R } ^ { n } }
		\left(  \iint_{ | x - y |< s ^ { 1 / ( 2 \alpha )  } }
		\left|  \widetilde{D}_{\alpha, s} ^ {L, \beta } (g) (x) \right| ^ { 2 } \frac{ s ^ { 1 / ( 2 \alpha ) - 1 } d y d s }
		{ s ^ {( n + 1) / ( 2 \alpha ) } } \right) ^ { \frac{n}{ 2 ( n + \gamma ) } }  d x  \right\}
		^ { 1 + \gamma / n }  . \nonumber
	\end{align}
\end{lemma}
\begin{proof}
	It is easy to see from \cite [(5.3)] {harb 1} that for any pair of measurable functions $F$ and $G$ on $\mathbb { R } _ { + } ^ { n + 1 } $, the following inequality holds:
	\begin{equation*}
		\begin{aligned}
			&\iint _ { { \mathbb{ R } } ^ { n + 1 } _ { + } }  | F ( x , t ) | \cdot | G ( x , t ) |
			\frac { d x d t } { t }  \\
			& \lesssim \sup _ { B } \left\{ \frac { 1 } { | B | ^ { 1 + 2 \gamma / n } }
			\iint _ { \hat{ B } } | F ( x , t ) | ^{2}\frac { d x d t } { t } \right\} ^ { 1 / 2 }  \\
			& \quad \times  \left\{ \int _ { \mathbb { R } ^ { n } }    \left( \int _ { 0 } ^ { \infty } \int _ { | x - y | < t } | G ( y , t ) | ^ { 2 } \frac { d y d t }
			{ t ^ { n + 1 } } \right) ^ { \frac { n } { 2 ( n + \gamma ) } } d x \right\} ^ { 1 + \gamma / n } .
		\end{aligned}
	\end{equation*}
	By taking
	\begin{equation*}
		F ( x , t ) : = t ^ { 2 \beta } L ^ { \beta } \exp\left( - t ^ {2\alpha}  L^ { \alpha } \right)
		( f )  ( x )   ; \quad
		G ( x , t ) : = t ^ { 2 \beta } L ^ { \beta } \exp\left( - t ^ {2\alpha}  L^ { \alpha } \right)
		( g )  ( x )   ,
	\end{equation*}
	we can get
	\begin{align} \label{QfQg}
		&\iint _ { \mathbb { R } ^ { n + 1 } _ { + } }
		\left| t ^ { 2 \beta } L ^ { \beta } \exp\left( - t ^ {2\alpha}  L^ { \alpha } \right)
		( f )  ( x ) \right|  \nonumber
		\cdot \left| t ^ { 2 \beta } L ^ { \beta } \exp\left( - t ^ {2\alpha}  L^ { \alpha } \right)
		( g )  ( x )  \right|   \frac { d x d t } { t }       \\ 
		&\leq C \sup _ { B } \left( \frac { 1 } { | B | ^ { 1 + 2 \gamma / n } } \iint _ { \hat{ B } }
		\left| t ^ { 2 \beta } L ^ { \beta } \exp\left( - t ^ {2\alpha}  L^ { \alpha } \right)
		( f )  ( x )  \right| ^ { 2 }
		\frac { d x d t } { t } \right) ^ { 1 / 2 } \nonumber    \\ 
		& \quad \times
		\left\{ \int _ { \mathbb { R } ^ { n } }  \left( \int _ { 0 } ^ { \infty }  \int _ { | x - y | < t }
		\left|  t ^ { 2 \beta } L ^ { \beta } \exp\left( - t ^ {2\alpha}  L^ { \alpha } \right)
		( g )  ( x ) \right| ^ { 2 } \frac { d y d t }
		{ t ^ { n + 1 } } \right) ^ { \frac { n } { 2 ( n + \gamma ) } } d x \right\} ^ { 1 + \gamma / n } . 
	\end{align}
	
	On the one hand, via the change of variables, we can obtain
	\begin{align*}
		& \iint _ { \mathbb { R } ^ { n + 1 } _ { + } }
		\left| t ^ { 2 \beta } L ^ { \beta }
		\exp\left( - t ^ { 2 \alpha }  L ^ { \alpha } \right)   ( f ) ( x ) \right| \cdot
		\left| t ^ { 2 \beta } L ^ { \beta }
		\exp\left( - t ^ { 2 \alpha }  L ^ { \alpha } \right)   ( g ) ( x ) \right|    \frac { d x d t } { t }   \\ 
		& \simeq \iint _ { \mathbb { R } ^ { n + 1 } _ { + } }  \left| s ^ { \beta/\alpha } L ^ { \beta }
		\exp\left( - s L^ { \alpha } \right) ( f ) ( x ) \right| \cdot  \left| s ^ { \beta/\alpha } L ^ { \beta }
		\exp\left( - s L^ { \alpha } \right) ( g ) ( x ) \right|  \frac { d x  d s } { s }  .
	\end{align*}
	
	On the other hand, using change of variables again, we can get
	\begin{align*}
		& \sup _ { B } \left( \frac { 1 } { | B | ^ { 1 + 2 \gamma / n } } \int _ { 0 } ^ { r _ { B } }
		\int _ { B } \left| t ^ { 2 \beta } L ^ { \beta }
		\exp\left( - t ^ { 2 \alpha } L ^ { \alpha } \right) ( f ) ( x ) \right| ^ { 2 }
		\frac { d x d t } { t } \right) ^ { 1 / 2 }    \\ 
		& \lesssim  \sup _ { B } \left( \frac { 1 } { | B | ^ { 1 + 2 \gamma / n } }
		\int _ { 0 } ^  { r _ { B } ^ { 2 \alpha } }   \int _ { B } \left|  s ^ { \beta/\alpha } L ^ { \beta }
		\exp\left( - s L ^ { \alpha } \right) ( f ) ( x ) \right| ^ { 2 } \frac { d x d s } { s } \right) ^ { 1 / 2 }  .
	\end{align*}
	Meanwhile,
	\begin{align*}
		& \left( \int _ { \mathbb{ R } ^ { n } } \left(  \int_{0}^{ \infty }
		\int_{ | x - y |< t }  \left| t ^ { 2 \beta } L ^ { \beta }
		\exp\left( - t ^ { 2 \alpha }  L ^ { \alpha } \right) ( g ) ( y ) \right| ^ { 2 } \frac{ d y d t }
		{ t ^ { n + 1 } } \right) ^ { \frac{n}{ 2 ( n + \gamma ) } }  d x  \right)  ^ { 1 + \gamma / n }  \\
		&\lesssim \left(  \int _ { \mathbb{ R } ^ { n } } \left(  \int_{0}^{ \infty }
		\int\limits_{ | x - y |< s ^ { 1 / ( 2 \alpha )  } }   \frac{ \left|  s ^ { \beta/\alpha } L ^ { \beta }  exp \left( -  s L ^ {\alpha } \right) ( g ) ( y ) \right| ^ { 2 }  }
		{ s ^ { n / ( 2 \alpha ) + 1 } } d y d s \right) ^ { \frac{n}{ 2 ( n + \gamma ) } }  d x  \right)
		^ { 1 + \frac{ \gamma }{ n } }  .
	\end{align*}
	Then, we complete the proof of (\ref{Sfunction}).
	\end{proof}

For $\alpha \in ( 0 , 1 )$ and $\beta > 0 $, we define the area function
\begin{equation*}
	S ^ {L} _ { \alpha , \beta } ( h ) ( x ) : = \left( \iint_{ \left\{ ( y , t ) : | x - y | < t ^ { 1 / ( 2 \alpha ) } \right\}  }
	\left| \widetilde{D}_{\alpha, t} ^ {L, \beta } ( h ) ( y ) \right| ^ { 2 }
	\frac{  d y d t } { t ^ { n / ( 2 \alpha ) + 1 } }  \right) ^ { 1 /  2 }    .
\end{equation*}

By the functional calculus, it is easy to verify that the area function $S_ { \alpha , \beta}^{L}$ is bounded on $L ^ { 2 } \left( \mathbb { R } ^ { n } \right).$ Below we prove that the area function $S^{L}_{\alpha,\beta}$ is bounded from $H_{L}^{n / ( n + \gamma )} (\mathbb R^{n})$ to $L^{ n / ( n + \gamma ) }(\mathbb R^{n})$. The boundedness of square functions associated with operators have been studied in various settings. Let $L$ be a classical Schr\"odinger operators $L=-\Delta + V $. Denote by $S_{Q}$ the area function related with the heat semigroup $\{exp(-tL)\}_{t>0}$. Dziuba\'nski et al. obtained the $H^1_L-L^{1}$ boundedness of $S_{Q}$ in \cite[Lemma 6]{dziub 1}. For the degenerate Schr\"odinger operators $L=-{\omega}^{-1} \operatorname{div}(A(x)\nabla)+V$, Huang, Li and Liu obtained the $L^{2}$ boundedness of the area function $S^{L}_{P}$ related with the Poisson semigroup $\{P^{L}_{t}\}_{t>0}$ in \cite[Theorem 3.10]{huang 0} and the $H^1_L-L^{1}$ boundedness of $S_{h}^{L}$ in \cite[Theorem 3.9]{ huang 2}, where $S^{L}_{h}$ denotes the area function generated by the heat semigroup related with $L$.
For divergence form elliptic operators $L= -div(A \cdot \nabla )$, Yang-Meng \cite{yang-meng} studied the boundedness of Lusin-area function and the Littlewood-Paley $g^{*}_{\lambda}$-function related with $L$ on BMO spaces.
Later, Yang-Yang \cite{yang-yang} derived a square function characterization of the Orlicz-Hardy space via the Lusin area function $S_{h}$ associated with $\{ T_{t}\}_{t\geq 0}$.

\begin{theorem} \label{bdd S}
	Assume that $\alpha \in ( 0 , 1 ) $, $ \beta > 0 $ and $0 < \gamma < \min \{ 1 , 2 \alpha , 2 \alpha \beta \} $. Let $f$ be a linear combination of $H _ { L } ^ { n / ( n + \gamma ) }$-atoms.
	There exists a constant $ { C } $ such that
	\[\left\| S _ { \alpha , \beta } ^ { L } ( f ) \right\| _ { L ^ { n / ( n + \gamma ) }  }
	\leq C \displaystyle \| f \| _ { H _ { L } ^ { n / ( n + \gamma ) } } .\]
\end{theorem}
\begin{proof}
	Let $a$ be an $H _ { L } ^ { n / ( n + \gamma ) }$-atom associated with a ball
	$B = B \left( x _ { 0 } , r \right) .$ Then we write
	$\left\| S _ { \alpha , \beta } ^ { L } ( a ) \right\| _ { L ^ { n / ( n + \gamma ) }  }
	^ { n / ( n + \gamma ) }   \leq I + I I ,$
	where
	\[\left\{ \begin{aligned}
		I &: = \int _ { 8 B } \left| S _ { \alpha , \beta } ^ { L } ( a ) ( x ) \right|
		^ { n / ( n + \gamma ) } d x ; \\
		I I &: = \int _ { ( 8 B ) ^ { c } } \left| S _ { \alpha , \beta } ^ { L } ( a ) ( x ) \right| ^ { n / ( n + \gamma ) } d x .
	\end{aligned} \right.\]
	
	By the $L^{2}$-boundedness of $ S^{L}_{\alpha,\beta} $ and H\"older's inequality, we can obtain
	\begin{align*}
		I \lesssim
		\left( \int _ { 8 B } \left| S _ { \alpha , \beta } ^ { L } a ( x ) \right| ^ { 2 } d x \right)
		^ { \frac {n} {2 ( n + \gamma )} } | B | ^ { \frac { n + 2 \gamma } {  2 ( n + \gamma ) } }
		\lesssim \| a \| _ { 2 } ^ { n / ( n + \gamma ) } | B | ^ { \frac { n + 2 \gamma } {  2 ( n + \gamma ) }} \lesssim 1 .
	\end{align*}
	
	Next, we deal with $I I$ in the following two cases.
	
	Case 1: $r < \rho \left( x _ { 0 } \right) / 4 .$ For this case,
	$\int _ { \mathbb { R } ^ { n } } a ( x ) d x = 0 $. Let $$ L (y,t) = \int _ { \mathbb { R } ^ { n } }
	\frac{ \left( \widetilde{D}_{\alpha, t} ^ {L, \beta }  ( y , x' )
		- \widetilde{D}_{\alpha, t} ^ {L, \beta }  ( y , x_ { 0 } ) \right)  a ( x' )} { t ^ { n / ( 2 \alpha ) + 1 } }  d x' .$$
    Then
	$\left( S _ { \alpha , \beta } ^ { L } ( a ) ( x ) \right) ^ { 2 }
	\lesssim I _ { 1 } ( x ) +  I _ { 2 } ( x )$, where
	\[\left\{ \begin{aligned}
		I_{ 1 } ( x ) &: = \int_{0}^{ ( | x- x_ { 0 } | / 2 ) ^ { 2 \alpha } }
		\int\limits_{ | x - y |< t ^ { 1 / ( 2 \alpha )  } }     L (y,t) ^ { 2 }
		d y d  t      ; \\
		I_{ 2 } ( x ) &: = \int_ { ( | x- x_ { 0 } | / 2 ) ^ { 2 \alpha } } ^{ \infty }
		\int\limits_{ | x - y |< t ^ { 1 / ( 2 \alpha )  } }   L (y,t) ^ { 2 }
		d y d  t    .
	\end{aligned} \right.\]
	
	For $I _ { 1 }$, since $| x - y | \leq \left| x - x _ { 0 } \right| / 2 $,
	$\left| y - x ^ { \prime } \right| \sim \left| y - x _ { 0 } \right|$ for
	$x ^ { \prime } \in B \left( x _ { 0 } , r \right)$ and $x \notin B \left( x _ { 0 } , 8 r \right) .$
	Because $0 < t < \left| x - x _ { 0 } \right| ^ { 2 \alpha } / 2 ^ { 2 \alpha }$ and
	$| x - y | < t ^ { 1 / ( 2 \alpha ) } $, $| x - y | < \left| x - x _ { 0 } \right| / 2$.
	This implies that $\left| y - x _ { 0 } \right| \gtrsim \left| x - x _ { 0 } \right| / 2 $.
	We can use Lemma \ref{bdd} to deduce that there exists
	$\delta ^ { \prime } > \gamma$ such that
	\begin{align*}
		\left| I_{ 1 } ( x ) \right|
		\lesssim   \int_{0}^{ ( | x- x_ { 0 } | / 2 ) ^ { 2 \alpha } }
		\int_{ | x - y |< t ^ { 1 / ( 2 \alpha )  } }
		\frac {  \left( { r } / { t ^ {1 / ( 2 \alpha ) } } \right) ^ { 2 \delta '}
			{ | B | ^ { - 2 \gamma / n } }  t ^ { - n / ( 2 \alpha ) - 1 }  d y d t }
		{  t ^ { n /  \alpha  } ( 1 + | y - x _ { 0 } | / t ^ { 1 / ( 2 \alpha ) } ) ^ { 2 n + 4 \beta } }  	  ,
	\end{align*}
	which, via a direct computation, gives
	\[\int _ { ( 8 B ) ^ c } \left| I _ { 1 } ( x ) \right| ^ { { n } / {  ( 2 n + 2 \gamma ) } } d x
	\lesssim \int _ { ( 8 B ) ^ c } \left( \frac { r ^ {  \delta ' - \gamma  } }
	{ \left| x - x _ { 0 } \right| ^ { n + \delta ^ { \prime } } } \right) ^ { n / ( n + \gamma ) } d x
	\lesssim 1 .\]
	
	Let us continue with $I _ { 2 }$. Similarly, it follows from Lemma \ref{bdd} that
	\[\int _ { ( 8 B ) ^ c } \left| I _ { 2 } ( x ) \right| ^ { { n } / {  ( 2 n + 2 \gamma ) } } d x
	\lesssim \int _ { ( 8 B ) ^ c } \left( \frac { r ^ {  \delta ' - \gamma  } }
	{ \left| x - x _ { 0 } \right| ^ { n + \delta ^ { \prime } } } \right) ^ { n / ( n + \gamma ) } d x
	\lesssim 1 .\]
	
	Case 2: $\rho \left( x _ { 0 } \right) / 4 < r < \rho \left( x _ { 0 } \right) .$ For this case, the atom $a$ has no canceling condition. We have
	$\left( S _ { \alpha , \beta } ^ { L } ( a ) ( x ) \right) ^ { 2 }
	\lesssim  I _ { 3 } ( x ) + I _ { 4 } ( x ) + I _ { 5 } ( x ) $, where
	\[\left\{ \begin{aligned}
		I_{ 3 } ( x ) &: = \int_{0}^{ ( r / 2 ) ^ { 2 \alpha } }
		\int_{ | x - y |< t ^ { 1 / ( 2 \alpha )  } }  \left| \int _ { \mathbb { R } ^ { n } }
		\widetilde{D}_{\alpha, t} ^ {L, \beta }  ( y , x' ) a ( x' ) d x'  \right| ^ { 2 }
		\frac{  d y d  t } { t ^ { n / ( 2 \alpha ) + 1 } }  ; \\
		I_{ 4 } ( x ) &: = \int_{ ( r / 2 ) ^ { 2 \alpha } } ^ { ( | x- x_ { 0 } | / 2 ) ^ { 2 \alpha } }
		\int_{ | x - y |< t ^ { 1 / ( 2 \alpha )  } }  \left| \int _ { \mathbb { R } ^ { n } }
		\widetilde{D}_{\alpha, t} ^ {L, \beta }  ( y , x' ) a ( x' ) d x'  \right| ^ { 2 }
		\frac{  d y d  t } { t ^ { n / ( 2 \alpha ) + 1 } }  ; \\
		I_{ 5 } ( x ) &: = \int _ { ( | x- x_ { 0 } | / 2 ) ^ { 2 \alpha } } ^ { \infty }
		\int_{ | x - y |< t ^ { 1 / ( 2 \alpha )  } }  \left| \int _ { \mathbb { R } ^ { n } }
		\widetilde{D}_{\alpha, t} ^ {L, \beta }  ( y , x' ) a ( x' ) d x'  \right| ^ { 2 }
		\frac{  d y d  t } { t ^ { n / ( 2 \alpha ) + 1 } }  .
	\end{aligned} \right.\]
	
	For $x \in ( 8 B ) ^ { c }$ and $x ^ { \prime } \in B$, $\left| x ^ { \prime } - x _ { 0 } \right| < \left| x - x _ { 0 } \right| / 8 .$ On the other hand,
	for $ t \in \left( 0 , ( r / 2 \right) ^ { 2 \alpha })$, $ | x - y |
	\leq  | x- x_ { 0 } | / 8 $, which means that $\left| y - x ^ { \prime } \right| \geq c \left| x - x _ { 0 } \right| $. So we get
	\[\int _ { ( 8 B ) ^ c } \left| I_{ 3 } ( x ) \right| ^ { { n } / {  ( 2 n + 2 \gamma ) } } d x \lesssim \int _ { ( 8 B ) ^ { c } } \left( \frac { r ^ { 2 \beta - \gamma } }
	{ \left| x - x _ { 0 } \right| ^ { n +  2 \beta } } \right)
	^ { n / ( n + \gamma ) } d x \lesssim 1 .\]
	
	Notice that $r / 2 \leq t ^ { 1 / ( 2 \alpha) } \lesssim \left| x - x _ { 0 } \right| / 4$ for $t \in \left( ( r / 2 ) ^ { 2 \alpha } ,  ( | x- x_ { 0 } | / 2 ) ^ { 2 \alpha } \right) .$ It can be deduced from the triangle inequality that
	$\left| y - x ^ { \prime } \right| \sim \left| x - x _ { 0 } \right|$.
	Similarly to $ I_{ 4 } ( x ) $, we can obtain
	\[\int _ { ( 8 B ) ^ c } \left| I_{ 3 } ( x ) \right| ^ { { n } / {  ( 2 n + 2 \gamma ) } } d x \lesssim \int _ { ( 8 B ) ^ { c } } \left( \frac { r ^ { 2 \beta - \gamma } }
	{ \left| x - x _ { 0 } \right| ^ { n +  2 \beta } } \right)	^ { n / ( n + \gamma ) } d x \lesssim 1 .\]
	
	Owing to $r \sim \rho \left( x _ { 0 } \right) $, the estimate for $I _ { 5 }$ is similar to that for $I _ { 4 }$, and thus we have
	\[\int _ { ( 8 B ) ^ c } \left| I _ { 5 } ( x ) \right|
	^ {  { n } / { ( 2 n + 2 \gamma ) } } d x
	\lesssim \int _ { ( 8 B ) ^ c }  \left(  \frac { r ^ { N - \gamma  } } { \left| x - x _ { 0 } \right| ^ { n + N  } } \right) ^ { n / ( n + \gamma ) }  d x
	\lesssim 1 .\]
\end{proof}

\begin{lemma} \label{sup R q}
	Let $\alpha \in ( 0 , 1 ) $ and $q_{ t } ( \cdot , \cdot )$ be a function of
	$x , y \in \mathbb { R } ^ { n }$ and $t > 0 $. Assume that for each $N > 0$, there exists a constant
	${ C } _ { N } $ such that for $\theta > \gamma $,
	\begin{equation} \label{bdd q}
		\left| q _ { t } ( x , y ) \right| \leq C _ { N } \left( 1 + \frac { t ^ { 1 / ( 2 \alpha) } }
		{ \rho ( x ) } + \frac { t ^ { 1 / ( 2 \alpha) } } { \rho ( y ) } \right) ^ { - N }
		t ^ { - n / ( 2 \alpha) } \left( 1 + \frac { | x - y | } {t ^ { 1 / ( 2 \alpha) }  } \right)
		^ { - ( n + \theta ) } .
	\end{equation}
	Then for any $H _ { L } ^ { n / ( n + \gamma ) }$-atom a supported on $B \left( x _ { 0 } , r \right) ,$ there exists a constant $C _ { x _ { 0 } , r } $ such that
	\[\sup _ { t > 0 } \left| \int _ { \mathbb{ R } ^ { n } } q _ { t } ( x , y ) a ( y ) d y \right|
	\leq C _ { N , x _ { 0 } , r } ( 1 + | x | ) ^ { - n - \theta } , x \in \mathbb { R } ^ { n } .\]
\end{lemma}

\begin{proof}
	If $x \in B \left( x _ { 0 } , 2 r \right) $, $ 1 + |x|
	\leq  1 + 2r + | x _ {0}| $. It follows from the condition
	$\| a \| _ { \infty } \leq \left| B \left( x _ { 0 } , r \right) \right| ^ { - 1 - \gamma / n }$ that
	\begin{align*}
		\left| \int _ { \mathbb { R } ^ { n } } q _ { t } ( x , y ) a ( y ) d y \right|
		\leq   \int _ { B \left( x _ { 0 } , r \right) } 
		\frac { C_ { N } t ^ { - n /  ( 2 \alpha) } r ^ { - n - \gamma } } {\left( 1 +  { | x - y | } / { t ^ { 1 /  ( 2 \alpha) } } \right) ^ {  n + \theta }} d y
		\leq \frac{ C _ { N , x _ { 0 } , r } } { ( 1 + | x | ) ^ {  n + \theta } }.
	\end{align*}
	
	If $x \notin B \left( x _ { 0 } , 2 r \right) $, for any $y \in B \left( x _ { 0 } , r \right) $,
	$| x - y | \sim \left| x - x _ { 0 } \right| .$ On the other hand,
	$\rho ( y ) \sim \rho \left( x _ { 0 } \right) $ since $r < \rho \left( x _ { 0 } \right)$ and
	$\left| y - x _ { 0 } \right| < r $. Then by (\ref{bdd q}), we can obtain
	\begin{align*}
		\left| \int _ { \mathbb { R } ^ { n } } q _ { t } ( x , y ) a ( y ) d y \right|
		&\leq C_ { \theta }  \left(  \rho \left( x _ { 0 } \right) \right) ^ { \theta }  r ^ { - \gamma }
		\left| x - x _ { 0 } \right| ^ { - ( n + \theta ) }
		: = C_ { \theta , x _ { 0 } , r }  \left| x - x _ { 0 } \right| ^ { - ( n + \theta ) } .
	\end{align*}
	
	Because $x \notin B \left( x _ { 0 } , 2 r \right) $, set $x = x _ { 0 } + 2 r z $, where
	$| z | \geq 1 $. Then $ 1 + | x | \leq 1 + \left| x _ { 0 } \right| + 2 r | z | $
	and
	\begin{equation*}
		\frac{ 1 + \left| x _ { 0 } \right| + 2 r	}{ 2 r } \left| x - x _ { 0 } \right|
		=  ( 1 + \left| x _ { 0 } \right| + 2 r ) | z | \geq 1 + \left| x _ { 0 } \right| + 2 r | z |  ,
	\end{equation*}
	which implies that $\left| x _ { 0 } - x \right| \geq ( 1 + | x | ) / C _ { x _ { 0 } , r } .$
\end{proof}

\subsection{Characterizations of Campanato type spaces associated with $L$ } \label{sec 4.2}

Consider the Cauchy initial value problem
\begin{equation} \label{cauchy}
	\left\{ \begin{aligned}
		& \partial_{ t } u (x,t) - L^{\alpha} u (x,t) = 0 , & (x,t) \in \mathbb{ R }^{n+1}_{+}  ; \\
		& u (x,0) = f(x) , & x \in \mathbb{ R }^{n} .
	\end{aligned} \right.
\end{equation}

\begin{definition}
	Let $\alpha\in ( 0 , 1)$, $\beta>0$ and $0 < \gamma < \min \{ 1 , 2 \alpha , 4 \beta \}$. The space	$H_{L,\gamma}^{\alpha, \beta}(\mathbb{R}^{n+1}_{+} ) $ is defined  as the set of all measurable functions $ u : \mathbb{ R } ^{n+1}_{+}  \rightarrow \mathbb{ R }$ satisfying
	\begin{equation}
		\| u \|_{ H _{ L,\gamma }^{\alpha , \beta} (\mathbb{ R }^{n+1}_{+}) } :=
		\left( \frac{1}{|B|^{1+2\gamma / n }} \int_{0}^{r_{B}^{2\alpha}} \int_{ B }
		| t ^ { \beta / \alpha } L ^{ \beta }   u ( x , t ) | ^ { 2 }  \frac{ d x d t }{ t } \right) ^ { 1 / 2 }<\infty.
	\end{equation}
\end{definition}

\begin{theorem} \label{fe}
	Let $V \in B _ { q } $, $q > n $. Assume that $\alpha \in ( 0 , 1 ) , \beta > 0 $, and $0 < \gamma < \min \{ 1 , 2 \alpha , 4 \beta \} $. Let $f$ be a function such that for some $\epsilon > 0$,
	\begin{equation} \label{Rf1}
		\int _ { \mathbb { R } ^ { n } }
		\frac { | f ( x ) | } { ( 1 + | x | ) ^ { n + \gamma + \epsilon } } d x < \infty .
	\end{equation}
	\item[\rm (i) ] If $ f \in  \Lambda  _ { L , \gamma }  \left( \mathbb { R } ^ { n } \right) $, $ u (x,t) = exp ( - t L ^ {\alpha } ) f (x) $ is a solution to equations (\ref{cauchy}) with $ \| u \|_{ H _{ L,\gamma }^{\alpha, \beta} (\mathbb{ R }^{n+1}_{+}) } < \infty $.
	\item[\rm (ii) ] If $u(\cdot , \cdot)$ is a solution to equations (\ref{cauchy}) and satisfies $ \| u \|_{ H _{ L,\gamma }^{\alpha,\beta} (\mathbb{ R }^{n+1}_{+}) } < \infty $, then $ f \in  \Lambda  _ { L , \gamma }  \left( \mathbb { R } ^ { n } \right) $.
\end{theorem}

\begin{proof}
	$ (i) \Longrightarrow ( i i )$.
	Let $ f \in \Lambda _ { L , \gamma } \left( \mathbb { R } ^ { n } \right) $. We have
	$\left|  \widetilde{D}_{\alpha, t} ^ {L, \beta }   f ( x ) \right|
	\lesssim I + I I $, where
	\[\left\{ \begin{aligned}
		I : & = \left | \int _ { \mathbb{ R } ^ { n } } \widetilde{D}_{\alpha, t} ^ {L, \beta }  ( x , y )
		\left(  f  ( y )  -  f  ( x )  \right)  d  y  \right | ;  \\
		I I : & = \left | f ( x )  \int _ { \mathbb{ R } ^ { n } } \widetilde{D}_{\alpha, t} ^ {L, \beta } ( x , y ) d  y  \right | .
	\end{aligned} \right.\]
	
	Similarly to \cite[Proposition 4.6]{ma 1}, we can prove that if $ f \in \Lambda _ { L , \gamma } \left( \mathbb { R } ^ { n } \right) $ then
	$$|f(x)-f(y)|\lesssim \|f\|_{\Lambda _ { L , \gamma }}|x-y|^{\gamma},$$
	which indicates that
	\begin{align*}
		I \lesssim  \| f \| _ { \Lambda _ { L , \gamma } }   \int _ { \mathbb{ R } ^ { n } }
		\frac {  t ^ {\beta / \alpha}  | x - y | ^ { \gamma }  }
		{ ( t ^ { 1 / ( 2 \alpha ) }  + | x - y | ) ^ { n + 2  \beta } }  d y
		\lesssim   t ^ { \gamma / ( 2 \alpha ) }  \| f \| _ { \Lambda _ { L , \gamma } }  .
	\end{align*}
	
	Below we estimate II. The proof of this part is divided into two cases.
	
	Case 1: $\rho ( x ) \leq  t ^ { 1 / ( 2 \alpha ) }  $. Noting that $|f(x)|\lesssim \| f \| _ {\Lambda _ { L , \gamma } }   \rho ( x ) ^ { \gamma },$ by Lemma \ref{bdd}, we can get
	\begin{align*}
		I I   \lesssim  \| f \| _ {\Lambda _ { L , \gamma } }   \rho ( x ) ^ { \gamma }
		\left| \int _ { \mathbb { R } ^ { n } }  \widetilde{D}_{\alpha, t} ^ {L, \beta } ( x , y ) d y \right|
		\lesssim   \| f \| _ { \Lambda _ { L , \gamma } }   t ^ {  \gamma / ( 2 \alpha ) }.
	\end{align*}
	
	Case 2: $\rho ( x ) > t ^ { 1 / ( 2 \alpha ) } $. We use Lemma \ref{bdd} to obtain that there exists $\delta ^ { \prime } > \gamma $ such that
	\begin{align*}
		I I
		\lesssim  \| f \| _ { \Lambda _ { L , \gamma } }    t ^ { \gamma / ( 2 \alpha ) }
		\frac { \left(  t ^ { 1 / ( 2 \alpha ) } / \rho ( x ) \right) ^ { \delta ' - \gamma } }
		{ \left( 1 + t ^ { 1 / ( 2 \alpha ) } / \rho ( x ) \right) ^ { N } }
		\lesssim   \| f \| _ { \Lambda _ { L , \gamma } }   t ^ {  \gamma / ( 2 \alpha ) }.
	\end{align*}
	
	We therefore conclude that
	\begin{align} \label{BDL1}
		\left( \frac { 1 } { | B | } \int _ { 0 } ^ { r _ { B } ^ { 2 \alpha }} \int _ { B }
		\left| \widetilde{D}_{\alpha, t} ^ {L, \beta } ( f ) ( x ) \right| ^ { 2 }
		\frac { d x d t } { t } \right)  ^ { 1 / 2 }    
		\lesssim | B | ^ { \gamma / n } ,
	\end{align}
	which implies that the function $ u (x,t) = exp ( - t L ^ {\alpha } ) f (x) $ solves equations  (\ref{cauchy}) with $ \| u \|_{ H _{ L,\gamma }^{\alpha, \beta} (\mathbb{ R }^{n+1}_{+}) } < \infty $.

	$( i i ) \Longrightarrow ( i ) $. Assume that  $u(\cdot , \cdot)$  be a solution to equations (\ref{cauchy}) satisfying $ \| u \|_{ H _{ L,\gamma }^{\alpha,\beta} (\mathbb{ R }^{n+1}_{+}) } < \infty $. Before proceeding with the proof, we first need to establish the following proposition:
	Assume that $\alpha \in ( 0 , 1 ) $, $ \beta > 0 $ and $0 < \gamma < \min \{ 1 , 2 \alpha , 2 \alpha \beta \} $. Let $f \in L ^ { 1 } \left( \mathbb { R } ^ { n } , ( 1 + | x | ) ^ { - ( n + \gamma + \epsilon ) } d x \right)$ for any $\epsilon > 0 $, and let a be an $H _ { L } ^ { n / ( n + \gamma ) }$-atom. We claim that for $\widetilde{D}_{\alpha, t} ^ {L, \beta } ( f )  ( x )$ and $ \widetilde{D}_{\alpha, t} ^ {L, \beta } ( a )  ( x ) $,
	there exists a constant $C_{ \alpha , \beta }$ such that
	\begin{equation}\label{eq-1111}
		C_{ \alpha , \beta } \int _ { \mathbb { R } ^ { n } } f ( x ) \overline { a ( x ) } d x
		= \iint _ { \mathbb { R } ^ { n + 1 } _ { + } } \widetilde{D}_{\alpha, t} ^ {L, \beta } ( f )  ( x ) \overline  { \widetilde{D}_{\alpha, t} ^ {L, \beta } ( a )  ( x ) }
		\frac { d x d t } { t } .
	\end{equation}
	If (\ref{eq-1111}) holds, let $a$ be an
	$H _ { L } ^ { n / ( n + \gamma ) }$-atom associated with $B = B \left( x _ { B } , r _ { B } \right) $. Then
	it can be deduced from 	(\ref{eq-1111}), (\ref{Sfunction}) and Theorem \ref{bdd S} that
	\begin{align*}
		\left| \int _ { \mathbb { R } ^ { n } } f ( x ) \overline { a ( x ) } d x \right|
		& \lesssim \left\| S _ { \alpha , \beta } ^ { L } ( a ) \right\|
		_ { L ^ { \frac{n}{( n + \gamma ) } }  }
		\sup _ { B } \left(   \int _ { B }
		\int _ { 0 } ^ {  r _ { B } ^ { 2 \alpha } }
		\frac { \left| \widetilde{D}_{\alpha, t} ^ {L, \beta }  ( f ) ( x ) \right| ^ { 2 }  } { | B | ^ { 1 + 2 \gamma / n } }
		\frac { d x d t } { t } \right) ^ { 1 / 2 }  \\
		& \lesssim  \left \|  a  \right \| _ { H _ { L } ^ { \frac{n}{( n + \gamma ) }  } }  .
	\end{align*}
	Therefore,
	\[T ( g ) : = \int _ { \mathbb { R } ^ { n } } f ( x ) \overline { g ( x ) } d x\]
	is a bounded linear functional on $H _ { L } ^ { n / ( n + \gamma ) } \left( \mathbb { R } ^ { n } \right) $, i.e., $f \in \left( H _ { L } ^ { n / ( n + \gamma ) } \left( \mathbb { R } ^ { n } \right) \right) ^ { * } = \Lambda _ { L , \gamma } \left( \mathbb { R } ^ { n } \right) $.
	
	At last we give the proof of (\ref{eq-1111}).
	Assume that $a$ is an $H _ { L } ^ { n / ( n + \gamma ) }$-atom associated with a ball
	$B \left( x _ { 0 } , r \right) .$ By Lemma \ref{FG} and Theorem \ref{bdd S}, we have
	\begin{align*}
		I & = \iint _ { \mathbb { R } ^ { n + 1 } _ { + } } \widetilde{D}_{\alpha, t} ^ {L, \beta } ( f )  ( x ) \overline  { \widetilde{D}_{\alpha, t} ^ {L, \beta } ( a )  ( x ) }
		\frac { d x d t } { t }  \\
		& = \lim _ { \epsilon \rightarrow 0 } \int _ { \epsilon } ^ { 1 / \epsilon }
		\int _ { \mathbb{ R } ^ { n } } \widetilde{D}_{\alpha, t} ^ {L, \beta } ( f )  ( x )
		\overline  { \widetilde{D}_{\alpha, t} ^ {L, \beta } ( a )  ( x ) } \frac { d x d t } { t }   .
	\end{align*}
	
	The inner integration satisfies
	\begin{align*}
		\left|  \int _ { \mathbb{ R } ^ { n } } \widetilde{D}_{\alpha, t} ^ {L, \beta } ( f )  ( x )  \overline  { \widetilde{D}_{\alpha, t} ^ {L, \beta } ( a )  ( x ) }  d  x  \right|
		\leq \int _ { \mathbb{ R } ^ { n } }  \big | \widetilde{D}_{\alpha, t} ^ {L, \beta } ( f )  ( x )  \big |
		\left\{ \sup_{ t > 0 } \big | \overline { \widetilde{D}_{\alpha, t} ^ {L, \beta } ( a )  ( x )  } \big | \right\} d x .
	\end{align*}
	By Lemma \ref{bdd}, we obtain
	\begin{align*}
		\left| \widetilde{D}_{\alpha, t} ^ {L, \beta } ( x , y ) \right|
		\leq  C _ { N } t ^ { - n / ( 2 \alpha ) }   \left( 1 + \frac {| x - y |}  { t ^ { 1 / ( 2 \alpha ) } } \right) ^ { - n - 2  \beta }   \left( 1 + \frac { t ^ { 1 / ( 2 \alpha ) } } { \rho ( x ) }
		+ \frac { t ^ { 1 / ( 2 \alpha ) } } { \rho ( y ) } \right) ^ { - N }   .
	\end{align*}
	
	If $x \in B \left( x _ { 0 } , 2 r \right) $, $1+|x| \leq 1 + 2r +|x_{0}|$. It follows from the condition
	$\| a \| _ { \infty } \leq \left| B \left( x _ { 0 } , r \right) \right| ^ { - 1 - \gamma / n }$ that
	\begin{align} \label{R frac D}
		\left| \int _ { \mathbb{ R } ^ { n } } \widetilde{D}_{\alpha, t} ^ {L, \beta } ( x , y )  a ( y ) d y  \right|
		\leq \frac{ C_ { N } }{ r ^ { n + \gamma } }
		\frac{ \left( 1 + 2 r + \left| x _ { 0 } \right| \right) ^ { n + 2  \beta } }
		{ \left( 1 + 2 r + \left| x _ { 0 } \right| \right) ^ { n + 2   \beta } }  \leq C _ { \gamma , x _ { 0 } , r  } ( 1 + | x | ) ^ { - n - 2  \beta }    .
	\end{align}
	
	If $x \notin B \left( x _ { 0 } , 2 r \right) $, for any $y \in B \left( x _ { 0 } , r \right) $,
	$ | x - y | \sim \left| x - x _ { 0 } \right| .$ On the other hand, $\rho ( y ) \sim \rho \left( x _ { 0 } \right) $ since $r < \rho \left( x _ { 0 } \right)$ and $\left| y - x _ { 0 }   \right| < r$.
	By Lemma \ref{bdd}, we can get
	\begin{align*}
		\left| \widetilde{D}_{\alpha, t} ^ {L, \beta } ( x , y ) \right|
		\leq  C _ { N }  t ^ { - n / ( 2 \alpha ) }  \left( 1 + \frac {| x - y |}  { t ^ { 1 / ( 2 \alpha ) } } \right) ^ { - n - 2  \beta } \left(  \frac { t ^ { 1 / ( 2 \alpha ) } }
		{ \rho ( x _ { 0 } ) } \right) ^ { - 2  \beta }      ,
	\end{align*}
	which implies that
	\begin{align*}
		\left| \int _ { \mathbb{ R } ^ { n } } \widetilde{D}_{\alpha, t} ^ {L, \beta } ( x , y )  a ( y ) d y  \right| \leq \frac { C _ { N } }{ r ^ {  \gamma } } \left( \rho ( x _ { 0 } ) \right) ^ { 2  \beta }  
		\left| x - x _ { 0 } \right|  ^ { - ( n + 2 \beta ) }
		: = \frac { C _ {  x _ { 0 } , r , \gamma } }{ \left| x - x _ { 0 } \right|  ^ {  n + 2 \beta ) }}  .
	\end{align*}
	
	Because $x \notin B \left( x _ { 0 } , 2 r \right)$, set $x = x _ { 0 } + 2 r z $, where $| z | \geq 1 $. Thus, $1 + | x | \leq 1 + \left| x _ { 0 } \right| + 2 r | z |$ and
	\begin{equation*}
		\frac{ 1 + \left| x _ { 0 } \right| + 2 r	}{ 2 r } \left| x - x _ { 0 } \right|
		=  ( 1 + \left| x _ { 0 } \right| + 2 r ) | z | \geq 1 + \left| x _ { 0 } \right| + 2 r | z |  ,
	\end{equation*}
	which implies that $\left| x _ { 0 } - x \right| \geq ( 1 + | x | ) / C _ { x _ { 0 } ,  r }$, and
	\begin{equation} \label{equ R D}
		\left| \int _ { \mathbb{ R } ^ { n } }  \widetilde{D}_{\alpha, t} ^ {L, \beta } ( x , y )  a ( y ) d y   \right| \leq C _ { \gamma , x _ { 0 } , r } ( 1 + | x | ) ^ { - n - 2  \beta } .
	\end{equation}
	
	The above estimate indicates that $\widetilde{D} ^ { L , \beta } _ { \alpha ,t } ( \cdot , \cdot )$ satisfies
	(\ref{bdd q}) with $\theta = 2 \beta$. On the other hand, it can be deduced from (\ref{R frac D}) and (\ref{equ R D}) that
	\begin{align*}
		\int _ { \mathbb{ R } ^ { n } }  \big | \widetilde{D}_{\alpha, t} ^ {L, \beta } ( f )  ( x )  \big |
		\cdot  \big | { \widetilde{D}_{\alpha, t} ^ {L, \beta }( a )  ( x )  } \big |  d x
		\lesssim  I _ { 1 } +  I _ { 2 }  ,
	\end{align*}
	where
	\[\left\{ \begin{aligned}
		I_{ 1 } &: = \iint_{ | x - y | > | y | / 2 }   \frac { | f ( y ) |  t ^ { \beta/\alpha } }
		{ ( t ^ { 1 / ( 2 \alpha ) } + | x - y | ) ^ { n + 2  \beta } }
		\frac{ d x d y } { ( 1 + | x | ) ^ { n + 2  \beta } } ; \\
		I_{ 2 } &: = \iint_{ | x - y | \leq  | y | / 2 }   \frac { | f ( y ) |  t ^ { \beta /\alpha } }
		{ ( t ^ { 1 / ( 2 \alpha ) } + | x - y | ) ^ { n + 2 \beta } }
		\frac{ d x d y } { ( 1 + | x | ) ^ { n + 2  \beta } } .
	\end{aligned} \right.\]
	
	If $| x - y | < | y | / 2 $, then $| y | \leq | y | / 2 + | x | $, i.e.,
	$ | y | \leq 2 | x | $. A direct computation gives
	\begin{align*}
		I_{ 2 }  \lesssim  \int _ { \mathbb{ R } ^ { n } }  \left(  \int _ { \mathbb{ R } ^ { n } }
		\frac { t ^ { \beta/\alpha } } { ( t ^ { 1 / ( 2 \alpha ) } + | x - y | ) ^ { n + 2  \beta } } d x \right)
		\frac{ | f ( y ) | }{ ( 1 + | y | ) ^ { n + 2  \beta } } d y   <  \infty  .
	\end{align*}
	
	For $I _ { 1 } $, since $| x - y | > | y | / 2 $, we can get
	\begin{align*}
		\int _ { \mathbb{ R } ^ { n } }  \frac { | f ( y ) |  t ^ { \beta/\alpha } }
		{ ( t ^ { 1 / ( 2 \alpha ) } + | x - y | ) ^ { n + 2 \beta } } d y
		\lesssim  \frac{ 1 } {  t ^ { n / ( 2 \alpha ) } }  \int _ { \mathbb{ R } ^ { n } }
		\frac { | f ( y ) | } { ( 1 + | y | /  t ^ { 1 / ( 2 \alpha ) }  ) ^ { n + 2  \beta } } d y  .
	\end{align*}
	
	Thus, there exists a constant $ C _ { t } $ such that
	\begin{align*}
		I_{ 1 }
		\lesssim  C _ { t }  \int _ { \mathbb{ R } ^ { n } }  \left(  \int _ { \mathbb{ R } ^ { n } }
		\frac {  | f ( y ) |  } { ( 1 + | y | ) ^ { n + 2  \beta } } d y \right)
		\frac{ d x  }{ ( 1 + | x | ) ^ { n + 2  \beta } }  <  \infty  .
	\end{align*}
	
	Noticing that
	\begin{align*}
		\int _ { \mathbb{ R } ^ { n } } \widetilde{D}_{\alpha, t} ^ {L, \beta } ( f )  ( x )
		\overline  { \widetilde{D}_{\alpha, t} ^ {L, \beta } ( a )  ( x ) }  d  x
		= \int _ { \mathbb{ R } ^ { n } }  f ( x )
		\overline { \widetilde{D}_{\alpha,2t} ^ {L,2\beta }  ( a )  ( x ) }  d  x   ,
	\end{align*}
	we can apply the Fubini theorem to obtain
	\begin{align*}
		I & = \lim _ { \epsilon \rightarrow 0 }   \int _ { \epsilon } ^ { 1 / \epsilon } \left\{
		\int _ { \mathbb{ R } ^ { n } }  f ( y )  \overline { \widetilde{D}_{\alpha,2t} ^ {L,2\beta }
			( a )  ( y ) }  d  y  \right\}     \frac { d t } { t }  \\
		& = \lim _ { \epsilon \rightarrow 0 }  \int _ { \mathbb{ R } ^ { n } }  f ( y )   \left\{
		\int _ { \epsilon } ^ { 1 / \epsilon }   \overline { \widetilde{D}_{\alpha,2t} ^ {L,2\beta }  ( a )  ( y )  }   \frac { d t } { t }  \right\}   d  y  .
	\end{align*}	
	
	Since
	\begin{align*}
		\left| \int _ { \epsilon } ^ { 1 / \epsilon }  \widetilde{D}_{\alpha,2t} ^ {L,2\beta }  ( a )  ( y )     \frac { d t } { t } \right|
		& \leq   \left| \int _ { \mathbb{ R } ^ { n } }  \int _ { \epsilon } ^ { \infty }   \widetilde{D}_{\alpha,2t} ^ {L,2\beta }   ( x , y )
		\frac { d t } { t }   a  ( x ) d x  \right| \\
		& \quad  + \left| \int _ { \mathbb{ R } ^ { n } }
		\int _ { 1 / \epsilon  } ^ { \infty }   \widetilde{D}_{\alpha,2t} ^ {L,2\beta }   ( x , y )
		\frac { d t } { t }   a  ( x ) d x  \right|  ,
	\end{align*}
	a change of variables yields
	\begin{align*}
		\left| \int _ { \epsilon } ^ { \infty }   \widetilde{D}_{\alpha,2t} ^ {L,2\beta }   ( x , y )
		\frac { d t } { t }  \right|
		\simeq \left| C _ { \beta }  \int _ { 2 \epsilon } ^ { \infty }
		\widetilde{D}_{\alpha,t} ^ {L,2\beta }  ( x , y )
		\frac { d t } { t }  \right|
		\simeq \left|  \int _ { 2 \epsilon } ^ { \infty }  \widetilde{D}_{\alpha,t} ^ {L,2\beta }
		( x , y )   \frac { d t } { t }  \right|  .
	\end{align*}
	
	The rest of the proof will be divided into three cases.
	
	Case 1: $2 \beta \in (0,\alpha) $. Since $ [ 2 \beta ] + 1 = 1 $,
	\begin{align*}
		t ^ { 2 \beta /\alpha }  L ^ { 2 \beta }	 exp ( - t L ^ {\alpha } )
		= t ^ { 2 \beta / \alpha } \int_{ 0 }^{ \infty } \int_{0}^{s} \partial_{ r }  exp ( - (t+r) L ^ {\alpha } )
		\frac{ d r d s }{ s ^ { 1 + 2 \beta / \alpha }} ,
	\end{align*}
	then
	\begin{align*}
		\int _ { 2 \epsilon } ^ { \infty } \widetilde{D}_{\alpha,t} ^ {L,2\beta }  ( x , y )
		\frac { d t } { t }
		&= \int _ { 2 \epsilon } ^ { \infty }   t ^ { 2 \beta/\alpha }  \int _ { 0 } ^ { \infty }
		\int_{r}^{\infty}  \frac{ d s }{ s ^ { 1 + 2 \beta/\alpha } }  \partial_{ r } K_{\alpha, t+r}^{L}
		( x , y ) d r   \frac { d t } { t }  \\
		&= \int _ { 2 \epsilon } ^ { \infty }  \partial_{ u }  K_{\alpha, u}^{L}  ( x , y )
		\left( \int  _ { 2 \epsilon  /  u  } ^ {  1  }   \left( \frac{ w } { 1 - w } \right) ^ { 2 \beta/\alpha }
		\frac { d w } { w }  \right) d u  .
	\end{align*}
	
	Note that
	\[\lim _ { u \rightarrow ( 2 \epsilon ) ^ { + } } K_{ \alpha, 2 \epsilon } ^ {L} ( x , y )
	\int  _ { 2 \epsilon  /  u  } ^ {  1  }  \left( \frac{ w } { 1 - w } \right) ^ { 2 \beta/\alpha }
	\frac { d w } { w }  = 0 ,\]
	and, as $u \rightarrow \infty $,
	\[  \left |   K_{\alpha, u}^{L}  ( x , y )  \right |
	\leq  \frac { u }  { ( u ^ { 1 / ( 2 \alpha ) } + | x - y | ) ^ { n + 2 \alpha } }
	\leq  \frac { 1 }  { u ^ { n / ( 2 \alpha ) } }   \rightarrow  0  .  \]
	
	An application of integration by parts gives
	\begin{align*}
		\int _ { 2 \epsilon } ^ { \infty } \widetilde{D}_{\alpha,t} ^ {L,2\beta } ( x , y )
		\frac { d t } { t }
	  &	= - C \int _ { 2 \epsilon } ^ { \infty }   K_{\alpha, u}^{L}  ( x , y )  \partial_{ u }
		\left( \int  _ { 2 \epsilon  /  u  } ^ {  1  }   \left( \frac{ w } { 1 - w } \right) ^ { 2 \beta/\alpha }
		\frac { d w } { w }  \right) d u   \\
	  &	= I + II  ,
	\end{align*}
	where
	\[\left\{ \begin{aligned}
		I : & = C \int _ { 2 \epsilon } ^ { \infty }   K_{\alpha, u}^{L}  ( x , y )  \left(
		\frac{ 2 \epsilon } { u - 2 \epsilon  } \right)  ^ { 2 \beta/\alpha }  \chi_{ A } ( u )  \frac { d u } { u } ; \\
		I I : & = C \int _ { 2 \epsilon } ^ { \infty }   K_{\alpha, u}^{L}  ( x , y ) \left(
		\frac{ 2 \epsilon } { u - 2 \epsilon  } \right) ^ { 2 \beta /\alpha}  \chi_{ A ^ c } ( u ) \frac { d u } { u } .
	\end{aligned} \right.\]
	Here $A : = \left\{ u : u - 2 \epsilon \leq \epsilon + | x - y | ^ { 2 \alpha } \right\} $.
	
	By (i) of Proposition \ref{bdd frac diff} and a straightforward calculation, we have
	\[\left\{ \begin{aligned}
		| I |  & &\lesssim  \frac { 1 } { \epsilon ^ { n / ( 2 \alpha ) } }
		\frac { 1 } { ( 1 + | x - y | / \epsilon ^ { 1 / ( 2 \alpha ) } ) ^ { n + 4 \beta } }
		\left( 1 + \frac { \epsilon ^ { 1 / ( 2 \alpha ) } } { \rho ( x ) }
		+ \frac { \epsilon ^ { 1 / ( 2 \alpha ) } } { \rho ( y ) } \right) ^ { - N }    ; \\
		| II | & &\lesssim  \frac { 1 } { \epsilon ^ { n / ( 2 \alpha ) } }
		\frac { 1 } { ( 1 + | x - y | / \epsilon ^ { 1 / ( 2 \alpha ) } ) ^ { n + 4 \beta } }
		\left( 1 + \frac { \epsilon ^ { 1 / ( 2 \alpha ) } } { \rho ( x ) }
		+ \frac { \epsilon ^ { 1 / ( 2 \alpha ) } } { \rho ( y ) } \right) ^ { - N }    .
	\end{aligned} \right.\]

	Case 2: $2 \beta = \alpha $. A direct computation gives
	\begin{align*}
		\left |  \int _ { 2 \epsilon } ^ { \infty } 	\widetilde{D}_{\alpha,t} ^ {L, 1} ( x , y )
		\frac { d t } { t }  \right |
		\lesssim  \left |  K_{ \alpha , 2 \epsilon } ^ { L } ( x , y ) \right |
		\lesssim \frac{ \left( 1 +  { \epsilon ^ { 1 / (2\alpha) } }  / { \rho ( x ) }
			+ { \epsilon ^ { 1 / (2\alpha) } } / { \rho ( y ) } \right) ^ { - N }} { \epsilon^{n/(2\alpha)} ( 1 + |x-y|/\epsilon^{1/(2\alpha)} ) ^{n+2\alpha }}
		.
	\end{align*}
	
	Case 3: $2 \beta > \alpha $. There exists a constant $k \in \mathbb { Z } _ { + }$ such that $k - 1 < 2 \beta / \alpha \leq k $, $ k \geq 2$. Then there exists a constant $ l \in [ 0 , 1 ) $ such that $ l = 2 \beta /\alpha - [ 2 \beta / \alpha ] = 2 \beta / \alpha + 1 - k $.
	Since
	\begin{align*}
		t ^ { 2 \beta / \alpha }  L ^ { 2 \beta }	 exp ( - t L ^ {\alpha } )
		= t ^ { 2 \beta / \alpha}  L ^ { \alpha [ 2 \beta/\alpha ] } \left( \int_{ 0 } ^ { \infty } \int_{ r } ^ { \infty }
		\frac{ d s }{ s ^ { 1 + l } }  \partial_{ r } exp( -  ( t + r )  L^ { \alpha } ) d r  \right),
	\end{align*}
	we can obtain
	\begin{align*}
		M & = C \int _ { 2 \epsilon } ^ { \infty }   t ^ { 2 \beta /\alpha}  \left(  \int_{ t } ^ { \infty}
		\partial_{ u } ^ { k }    K_{ \alpha , u } ^ { L }  ( x , y )
		\frac { d u } { ( u - t ) ^ { 1 + 2 \beta/\alpha - k } }  \right) \frac { d t } { t }   \\
		& = C \int _ { 2 \epsilon } ^ { \infty }  u ^ { k - 1 } \partial_{ u } ^ { k }    K_{ \alpha , u } ^ { L }
		( x , y )  \left( \int  _ { 2 \epsilon  /  u  } ^ {  1  }   w  ^ { 2 \beta / \alpha }
		\left(  1 - w  \right) ^ { k - 2 \beta /\alpha - 1 }   \frac { d w } { w }  \right) d u  ,
	\end{align*}
	where, in the last step, we have used the change of variables: $w = t / u $. Notice that
	\begin{align*}
		u ^ { k - 1 }  \partial_{ u } ^ { k }    K_{ \alpha , u } ^ { L } ( x , y )
		& = \partial_{ u } \left( { u } ^ { k - 1 }  \partial_{ u } ^ { k - 1 }
		K_{ \alpha , u } ^ { L }  ( x , y ) \right) - (k-1)\partial_{ u }\left( u ^{k-2} \partial_{ u }^{k-2}  K_{ \alpha , u } ^ { L }  ( x , y )  \right) \\
		& \quad +  \cdots  +  ( - 1 ) ^ { k - 1 }  ( k - 1 ) !    K_{ \alpha , u } ^ { L }  ( x , y ) .
	\end{align*}
	
	An application of integration by part gives
	\begin{align*}
		M & = C \int _ { 2 \epsilon } ^ { \infty } u ^ { k - 1 } \partial _ { u } ^ { k - 1 }
		K_{ \alpha , u } ^ { L } ( x , y )
		\frac { ( 2 \epsilon ) ^ { 2 \beta / \alpha } u ^ { 1 - k } }
		{ ( u - 2 \epsilon ) ^ { 1 + 2 \beta / \alpha - k } } \frac { d u } { u }
		+ \cdots   \\
		& \quad + C \int _ { 2 \epsilon } ^ { \infty } u \partial _ { u }
		K_{ \alpha , u } ^ { L } ( x , y )  \frac { ( 2 \epsilon ) ^ { 2 \beta / \alpha }
			u ^ { 1 - k } } { ( u - 2 \epsilon ) ^ { 1 + 2 \beta / \alpha - k } }
		\frac { d u } { u }  \\
		& \quad + C \int _ { 2 \epsilon } ^ { \infty }   K_{ \alpha , u } ^ { L } ( x , y )
		\frac { ( 2 \epsilon ) ^ { 2 \beta / \alpha }  u ^ { 1 - k } } { ( u - 2 \epsilon )
			^ { 1 + 2 \beta / \alpha - k } } 	\frac { d u } { u } .
	\end{align*}
	
	For any $1 \leq m \leq k - 1 $, by Proposition \ref{bdd D}, we have
	\begin{align*}
		\left | \int _ { 2 \epsilon } ^ { \infty } u ^ { m } \partial _ { u } ^ { m }
		K_{ \alpha , u } ^ { L } ( x , y ) \frac { ( 2 \epsilon ) ^ { 2 \beta / \alpha }
			u ^ { 1 - k } } { ( u - 2 \epsilon ) ^ { 1 + 2 \beta / \alpha - k } }  \frac { d u } { u }  \right |  \lesssim   I_ { m } ^ { ( 1 ) } + I_ { m } ^ { ( 2 ) } ,
	\end{align*}
	where
	\[\left\{ \begin{aligned}
		I_ { m } ^ { ( 1 ) } : & =
		\frac { \epsilon ^ { 2 \beta/\alpha } \left( 1  +  { \epsilon ^ { 1 / ( 2 \alpha ) } } / { \rho ( x ) }
			+  { \epsilon ^ { 1 / ( 2 \alpha ) } } / { \rho ( y ) } \right) ^ { - N } } { ( \epsilon ^ { 1 / ( 2 \alpha ) } + | x - y | ) ^ { n + 2 \alpha m } }   \int _ { 2 \epsilon }
		^ { 3 \epsilon } ( u - 2 \epsilon ) ^ { k - 2 \beta/\alpha - 1 } \frac { d u } { u ^ { k - m } }  ; \\
		I_ { m } ^ { ( 2 ) } : & =
		\frac { \epsilon ^ { 2 \beta/\alpha } \left( 1  +  { \epsilon ^ { 1 / ( 2 \alpha ) } } / { \rho ( x ) }
			+  { \epsilon ^ { 1 / ( 2 \alpha ) } } / { \rho ( y ) } \right) ^ { - N } } { ( \epsilon ^ { 1 / ( 2 \alpha ) } + | x - y | )
			^ { n + 2 \alpha m } }   \int _ { 3 \epsilon }
		^ { \infty } ( u - 2 \epsilon ) ^ { k - 2 \beta/\alpha - 1 } \frac { d u } { u ^ { k - m } }  .
	\end{aligned} \right.\]
	
	For $ I_ { m } ^ { ( 1 ) } $, since $2 \epsilon < u < 3 \epsilon $,
	\begin{align*}
		I_ { m } ^ { ( 1 ) } \lesssim  \frac { 1 } { \epsilon ^ { n / ( 2 \alpha ) } }
		\frac { 1 } { ( 1 + | x - y | / \epsilon ^ { 1 / ( 2 \alpha ) } ) ^ { n + 2 \alpha m } }
		\left( 1 + \frac { \epsilon ^ { 1 / ( 2 \alpha ) } } { \rho ( x ) }
		+ \frac { \epsilon ^ { 1 / ( 2 \alpha ) } } { \rho ( y ) } \right) ^ { - N }    .
	\end{align*}
	
	Similarly, for $  I_ { m } ^ { ( 2 ) } $, since $u \in ( 3 \epsilon , \infty ) $,
	$1 / u \lesssim 1 / ( u - 2 \epsilon ) $. Noting that $m < 2 \beta / \alpha $, we can get
	\begin{align*}
		I_ { m } ^ { ( 2 ) }
		\lesssim  \frac { 1 } { \epsilon ^ { n / ( 2 \alpha ) } }
		\frac { 1 } { ( 1 + | x - y | / \epsilon ^ { 1 / ( 2 \alpha ) } ) ^ { n + 2 \alpha m } }
		\left( 1 + \frac { \epsilon ^ { 1 / ( 2 \alpha ) } } { \rho ( x ) }
		+ \frac { \epsilon ^ { 1 / ( 2 \alpha ) } } { \rho ( y ) } \right) ^ { - N }    .
	\end{align*}
	By Lemma \ref{sup R q}, the above estimates in Cases $ 1-3 $ indciate that
	\begin{equation*}
		\sup_{ \epsilon > 0 }   \left| \int _ { \epsilon } ^ { 1 / \epsilon }  \widetilde{D}_{\alpha, 2t}
		^ {L, 2\beta }  ( a )  ( y )  \frac { d t } { t } \right|
		\lesssim \left( 1 + | y | \right)  ^ { - ( n + \gamma + \epsilon ) }  ,
	\end{equation*}
	where
	\begin{equation*}
		\epsilon = \left\{ \begin{aligned}
			&4 \beta - \gamma , & 2 \beta \in ( 0 , \alpha)  ; \\
			&2 \alpha - \gamma , & 2 \beta = \alpha ; \\
			&2 \alpha m - \gamma , & 2 \beta \in ( \alpha , \infty ) .
		\end{aligned} \right.
	\end{equation*}
	Therefore, we can use Lemma \ref{isometry} to complete the proof.

\end{proof}

\begin{theorem} \label{fD1}
	Define a general gradient as $ \nabla _ { \alpha } :
	= \left( \nabla _ { x } , \partial _ { t } ^ { 1 / ( 2  \alpha ) } \right) $.
	Let $V \in B _ { q } $, $q > n  $ and $ \beta > 0 $. Assume that $ u(x,t) $ be a solution to equations (\ref{cauchy})
	with the initial data $ f \in \Lambda _ { L , \gamma } \left( \mathbb { R } ^ { n } \right) $. Then the following characterizations of $ H_{L,\gamma}^{\alpha, \beta} ( \mathbb{ R }^{n+1}_{+} ) $ are equivalent.
	\item[\rm (i) ] $ u \in H_{L,\gamma}^{\alpha, \beta} ( \mathbb{ R }^{n+1}_{+} ) $.
	
	\item[\rm (ii) ] For $\alpha \in ( 0 , 1 ) $, $ \beta > 0 $, and
	$0 < \gamma < \min \{ 1 , 2 \alpha , 2  \alpha \beta \}$, $ \forall B = B \left( x _ { B } , r _ { B } \right) \subset \mathbb { R } ^ { n } $,
	\begin{equation} \label{BDL}
		\left( \frac { 1 } { | B |^ { 1 + 2 \gamma / n } } \int _ { 0 } ^ { r _ { B } ^ { 2 \alpha }} \int _ { B } \left| t ^ { \beta } \partial_{ t } ^ { \beta } u ( x , t ) \right| ^ { 2 } \frac { d x d t } { t } \right)	^ { 1 / 2 } \lesssim \| f \| _ {\Lambda _ { L , \gamma } } .
	\end{equation}
	\item [\rm (iii) ] For $\alpha \in ( 0 , 1 / 2 - n / ( 2 q ) )$, $ \beta > 0 $, and
	$0 < \gamma < \min \{ 1 , 2 \alpha , 2  \alpha \beta \}$, $ \forall B = B \left( x _ { B } , r _ { B } \right) \subset \mathbb { R } ^ { n } $,
	\begin{equation} \label{BeL}
		\left( \frac { 1 } { | B |^ { 1 + 2 \gamma / n } } \int _ { 0 } ^ { r _ { B } ^ { 2 \alpha }} \int _ { B } \left| t ^ { 1 / ( 2 \alpha ) } \nabla_ { \alpha } u ( x , t ) \right| ^ { 2 } \frac { d x d t } { t } \right)	^ { 1 / 2 } \lesssim \| f \| _ {\Lambda _ { L , \gamma } } .
	\end{equation}
\end{theorem}
\begin{proof}
	(i) $\Longleftrightarrow$ (ii).
	By Theorem \ref{fe}, $ u \in H_{L,\gamma}^{\alpha, \beta} ( \mathbb{ R }^{n+1}_{+} ) $ is equivalent to $ f \in \Lambda _ { L , \gamma } \left( \mathbb { R } ^ { n } \right) $, then proof is similar to that in \cite[page 33 - page 46]{li}, so we omit the details here.

	(i) $ \Longrightarrow $ (iii).
	By Theorem \ref{fe}, $f \in \Lambda _ { L , \gamma } \left( \mathbb { R } ^ { n } \right) $.
	Since $ ( i ) \Longleftrightarrow ( i i ) $, we can obtain
	$$\left\| D _ { \alpha , t } ^ { L , 1 / ( 2 \alpha ) } ( f ) \right\| _ { \infty }
	\leq C _ { \alpha , \beta }  t ^ { \gamma / ( 2 \alpha ) }  .  $$
	One writes
	\begin{align*}
		D_{ \alpha , t } ^ { L }  f ( x )   : = I ( x ) + I I ( x ) ,
	\end{align*}
	where
	\[\left\{ \begin{aligned}
		I ( x ) : & = \int _ { \mathbb { R } ^ { n } }  D_{ \alpha , t } ^ { L } ( x , z ) \left( f ( z ) - f ( x ) \right) d z ;  \\
		I I ( x ) : & = f ( x )  D_{ \alpha , t } ^ { L } ( 1 ) ( x )  .
	\end{aligned} \right.\]
	
	For $ I(x) $, similarly to \cite[Proposition 4.6]{ma 1}, we can prove that if $ f \in \Lambda _ { L , \gamma } \left( \mathbb { R } ^ { n } \right) $ then
	$$| f ( x ) - f ( z ) | \leq \| f \| _ { \Lambda _ { L , \gamma }  } | x - z | ^ { \gamma } ,$$
	which, together with Proposition \ref{t x frac}, implies that
	\begin{align*}
		| I ( x ) |  \lesssim \| f \| _ {\Lambda _ { L , \gamma } }  \int _ { \mathbb { R } ^ { n } }
		\left |  t  ^ { 1 / ( 2  \alpha ) }  \nabla _ { x }  K _ { \alpha , t } ^ { L }  ( x , z ) \right |
		\cdot | x - z | ^ { \gamma } d z
		\lesssim t ^ { \gamma / ( 2 \alpha ) }  \| f \| _ {\Lambda _ { L , \gamma } } .
	\end{align*}

	For $II(x)$, we split the discussion into the following two cases:
	
	Case 1: $\rho ( x ) \leq t ^ { 1 / ( 2 \alpha ) } $. Similarly to \cite[Proposition 4.6]{ma 1}, we can prove that if $ f \in \Lambda _ { L , \gamma }\left( \mathbb { R } ^ { n } \right) $ then
	$$| f ( x ) | \lesssim  \rho ^ { \gamma } ( x )\| f \| _ { \Lambda _ { L , \gamma } },$$
	which, together with Proposition \ref{t x frac}, implies that
	\begin{align*}
		I I ( x )  \lesssim \| f \| _ { \Lambda _ { L , \gamma }  }  \rho ^ { \gamma } ( x )
		\left | D_{ \alpha , t } ^ { L }  ( 1 ) ( x ) \right |
		\lesssim  \frac{ t  ^ { \gamma  / ( 2  \alpha ) }  \| f \| _ { \Lambda _ { L , \gamma }  }  }
		{\left(  { t ^ { 1 / ( 2 \alpha ) }  } / { \rho ( x ) } \right)  ^ {  N + \gamma }}
		\lesssim \| f \| _ { \Lambda _ { L , \gamma } }    t  ^ { \gamma  / ( 2  \alpha ) } .
	\end{align*}
	
	Case 2: $\rho ( x ) > t ^ { 1 / ( 2 \alpha ) } $. Similarly, we can obtain
	\begin{align*}
		I I ( x )
		\lesssim \| f \| _ { \Lambda _ { L , \gamma }  }  \rho ^ { \gamma } ( x )
		\left( \frac { t ^ { 1 / ( 2 \alpha ) }  } { \rho ( x ) } \right) ^ { 1 + 2 \alpha  }      \lesssim \| f \| _ { \Lambda _ { L , \gamma } }    t  ^ { \gamma  / ( 2  \alpha ) } .
	\end{align*}
	
	Then for every ball $B = B \left( x _ { B } , r _ { B } \right) ,$
	\[ \int _ { 0 } ^ {  r _ { B } ^ { 2 \alpha } } \int _ { B }    \left| t ^ { 1 / ( 2 \alpha ) } \nabla_ { \alpha } u ( x , t ) \right| ^ { 2 } \frac { d x d t } { t }
	\lesssim  \int _ { 0 } ^ {  r _ { B } ^ { 2 \alpha } } \int _ { B }  t^{\gamma/\alpha-1}dxdt  \lesssim  r _ { B } ^ { n + 2 \gamma } ,\]
	which implies that (\ref{BeL}) holds.
	
	(iii) $\Longrightarrow$ (i).
	If (\ref{BeL}) holds, for every ball
	$B = B \left( x _ { B } , r _ { B } \right) $, we can obtain
	\begin{equation*}
		\sup _ { B }  { r _ { B } } ^ { - ( n+ 2 \gamma ) } \int _ { 0 } ^ {  r _ { B } ^ { 2 \alpha } }
		\int _ { B \left( x _ { B } , r _ { B } \right) }
		\left| t  ^ { 1 / ( 2  \alpha ) }  \partial_{ t } ^ { 1 / ( 2  \alpha ) }   u ( x , t ) \right| ^ { 2 } \frac { d x d t } { t }  <  \infty ,
	\end{equation*}
	which is a corollary of (i) $ \Longleftrightarrow$ (ii)  that
	$f \in \Lambda _ { L , \gamma } \left( \mathbb { R } ^ { n } \right) $ with
	\begin{align*}
		\| f \| _ { \Lambda _ { L , \gamma }  }   \lesssim \sup _ { B }
		{ r _ { B } } ^ { - ( n+ 2 \gamma ) } \int _ { 0 } ^ {  r _ { B } ^ { 2 \alpha } }
		\int _ { B \left( x _ { B } , r _ { B } \right) }
		\left| t  ^ { 1 / ( 2  \alpha ) }  \partial_{ t } ^ { 1 / ( 2  \alpha ) }  u ( x , t ) \right| ^ { 2 } \frac { d x d t } { t }<\infty.
	\end{align*}
	By Theorem \ref{fe}, the function $ u (x,t) = exp ( - t L ^ {\alpha } ) f (x) $ solves equations (\ref{cauchy}) with $ \| u \|_{ H _{ L,\gamma }^{\alpha, \beta} (\mathbb{ R }^{n+1}_{+}) } < \infty $.
\end{proof}

By Theorems \ref{fe} \& \ref{fD1}, we will give several equivalent characterizations of $ \Lambda ^ { \kappa } _ { L , \gamma } $.
\begin{corollary} \label{Pt character}
	Let $ \gamma \in (0,1) $ and $\kappa \geq 0 $. For each measurable function $f$ on $\mathbb { R } ^ { n }$, it holds that
	\begin{align*}
		& \displaystyle \left\| f \right\| _ { \Lambda ^ { \kappa } _ { L , \gamma } }
		 = \displaystyle  \left\| L ^ { \kappa } f \right\| _ { \Lambda _ { L , \gamma } }  \\
		& \sim \left( \sup _ { ( x _ {0} , r ) \in \mathbb { R } ^ { n + 1 } _ {+} } r ^ { - ( 2 \gamma + n ) }
		\int _ { B ( x_{0} , r ) } \int _ { 0 } ^ { r ^ { 2 \alpha } }
		\left| t ^ { \beta / \alpha } L ^ {\beta} exp ( - t L ^ {\alpha } ) L ^ { \kappa }  f ( x ) \right|
		^ { 2 } \frac{ d t d x }{ t }  \right) ^ { 1 / 2 }   \\
		& \sim \left( \sup _ { ( x _ {0} , r ) \in \mathbb { R } ^ { n + 1 } _ {+} } r ^ { - ( 2 \gamma + n ) }
		\int _ { B ( x_{0} , r ) } \int _ { 0 } ^ { r ^ { 2 \alpha } }
		\left|   t ^ { \beta } \partial _ { t } ^ { \beta } exp ( - t L ^ {\alpha } ) L ^ { \kappa }  f ( x ) \right|
		^ { 2 } \frac{ d t  d x }{ t }  \right) ^ { 1 / 2 }   \\
		& \sim \left( \sup _ { ( x _ {0} , r ) \in \mathbb { R } ^ { n + 1 } _ {+}} r ^ { - ( 2 \gamma + n ) }
		\int _ { B ( x_{0} , r ) } \int _ { 0 } ^ { r ^ { 2 \alpha } } \left| t ^ { 1 / ( 2 \alpha ) } \nabla_{ \alpha }  exp ( - t L ^ {\alpha } )  L ^ { \kappa } f ( x ) \right|
		^ { 2 } \frac{ d t d x }{ t }  \right) ^ { 1 / 2 }   ;
	\end{align*}
\end{corollary}

\subsection*{Funding information}


The research was supported by Shandong Natural Science Foundation of China (No.ZR2024MA016) and the National Natural Science Foundation of China (No.12471093).

\subsection*{Conflict of Interests}

This work does not have any conflicts of interest.



\end{document}